\pdfoutput=1

\documentclass{article}

\usepackage[british]{babel}
\usepackage[latin1]{inputenc}
\usepackage[T1]{fontenc}
\usepackage{lmodern}
\usepackage{amsmath}
\usepackage{amssymb}
\usepackage{amsthm}
\usepackage{mathtools}
\usepackage{upgreek}
\usepackage{isomath}
\usepackage{paralist}
\usepackage[
  pdfpagemode = UseNone,
  colorlinks,
  linkcolor   = black,
  anchorcolor = black,
  citecolor   = black,
  filecolor   = black,
  menucolor   = black,
  urlcolor    = black,
]{hyperref}
\usepackage{ellipsis}
\usepackage{fixltx2e}
\usepackage{mparhack}

\makeatletter
\def\star@command@name#1{#1@star}
\def\nostar@command@name#1{#1@nostar}
\def\newstarcommand{%
  \newif\if@star@newstarcommand@
  \@ifstar{\@star@newstarcommand@true\@newstarcommand}{\@star@newstarcommand@false\@newstarcommand}%
}
\def\@newstarcommand#1{%
  \count@=\escapechar
  \escapechar=\m@ne
  \edef\@newstarcommand@command@string{\string#1}%
  \escapechar=\count@
  \def\@newstarcommand@command{#1}%
  \expandafter\def\expandafter\@newstarcommand@command@star\expandafter{\csname\star@command@name\@newstarcommand@command@string\endcsname}%
  \expandafter\def\expandafter\@newstarcommand@command@nostar\expandafter{\csname\nostar@command@name\@newstarcommand@command@string\endcsname}%
  \expandafter\expandafter\expandafter\expandafter\expandafter\expandafter\expandafter
  \newcommand\expandafter\expandafter\expandafter\expandafter\expandafter\expandafter\expandafter
  *\expandafter\expandafter\expandafter\expandafter\expandafter\expandafter\expandafter{%
    \expandafter\expandafter\expandafter\@newstarcommand@command\expandafter\expandafter\expandafter
  }\expandafter\expandafter\expandafter{%
    \expandafter\expandafter\expandafter\@ifstar\expandafter\@newstarcommand@command@star\@newstarcommand@command@nostar
  }%
  \@testopt\@@newstarcommand0%
}
\def\@@newstarcommand[#1]{%
  \@ifnextchar[{\@@newstarcommand@opt[{#1}]}{\@@newstarcommand@noopt[{#1}]}%
}
\def\@@newstarcommand@noopt[#1]#2#3{%
  \if@star@newstarcommand@
    \expandafter\newcommand\expandafter*\@newstarcommand@command@star[{#1}]{#2}%
    \expandafter\newcommand\expandafter*\@newstarcommand@command@nostar[{#1}]{#3}%
  \else
    \expandafter\newcommand\@newstarcommand@command@star[{#1}]{#2}%
    \expandafter\newcommand\@newstarcommand@command@nostar[{#1}]{#3}%
  \fi
}
\def\@@newstarcommand@opt[#1][#2]#3#4{%
  \if@star@newstarcommand@
    \expandafter\newcommand\expandafter*\@newstarcommand@command@star[{#1}][{#2}]{#3}%
    \expandafter\newcommand\expandafter*\@newstarcommand@command@nostar[{#1}][{#2}]{#4}%
  \else
    \expandafter\newcommand\@newstarcommand@command@star[{#1}][{#2}]{#3}%
    \expandafter\newcommand\@newstarcommand@command@nostar[{#1}][{#2}]{#4}%
  \fi
}
\makeatother

\theoremstyle{plain}
\newtheorem{theorem}{Theorem}
\newtheorem{corollary}[theorem]{Corollary}
\newtheorem{proposition}[theorem]{Proposition}
\newtheorem{lemma}[theorem]{Lemma}
\theoremstyle{definition}
\newtheorem{definition}[theorem]{Definition}

\theoremstyle{remark}
\newtheorem{remark}[theorem]{Remark}
\newtheorem{remarks}[theorem]{Remarks}
\theoremstyle{plain}

\newcommand*{\theoremref}[1]{\hyperref[#1]{Theorem~\ref*{#1}}}
\newcommand*{\corollaryref}[1]{\hyperref[#1]{Corollary~\ref*{#1}}}
\newcommand*{\propositionref}[1]{\hyperref[#1]{Proposition~\ref*{#1}}}
\newcommand*{\lemmaref}[1]{\hyperref[#1]{Lemma~\ref*{#1}}}
\newcommand*{\definitionref}[1]{\hyperref[#1]{Definition~\ref*{#1}}}
\newcommand*{\exampleref}[1]{\hyperref[#1]{Example~\ref*{#1}}}
\newcommand*{\remarkref}[1]{\hyperref[#1]{Remark~\ref*{#1}}}
\newcommand*{\remarksref}[1]{\hyperref[#1]{Remarks~\ref*{#1}}}
\newcommand*{\sectionref}[1]{\hyperref[#1]{Section~\ref*{#1}}}
\makeatletter
\renewcommand{\eqref}[1]{\hyperref[#1]{\textup{\tagform@{\ref*{#1}}}}}
\makeatother

\newcommand*{\N}{\ensuremath{\mathbb{N}}}
\newcommand*{\Z}{\ensuremath{\mathbb{Z}}}

\newcommand*{\R}{\ensuremath{\mathbb{R}}}
\newcommand*{\C}{\ensuremath{\mathbb{C}}}

\newcommand*{\HH}{\ensuremath{\mathbb{H}}}

\newcommand*{\D}{\ensuremath{\mathrm{d}}}
\newcommand*{\E}{\ensuremath{\mathrm{e}}}
\newcommand*{\I}{\ensuremath{\mathrm{i}}}

\newcommand*{\Lie}{\ensuremath{\mathcal{L}}}
\newcommand*{\CC}{\ensuremath{\mathrm{C}}}
\newcommand*{\GL}{\ensuremath{\mathrm{GL}}}
\newcommand*{\SO}{\ensuremath{\mathrm{SO}}}

\newcommand*{\SU}{\ensuremath{\mathrm{SU}}}

\newcommand*{\Spin}{\ensuremath{\mathrm{Spin}}}
\newcommand*{\vol}{\ensuremath{\mathrm{vol}}}

\newcommand*{\sections}{\ensuremath{\mathrm{\Gamma}}}
\newcommand*{\altforms}{\ensuremath{\mathrm{\Lambda}}}
\newcommand*{\forms}{\ensuremath{\mathrm{\Omega}}}
\newcommand*{\interior}{\ensuremath{\mathbin\lrcorner}}
\let\eps=\varepsilon
\newcommand*{\wolog}{w.l.o.g.}

\let\Re=\relax
\let\Im=\relax
\DeclareMathOperator{\Re}{Re}
\DeclareMathOperator{\Im}{Im}
\DeclareMathOperator{\id}{id}

\DeclareMathOperator{\supp}{supp}
\DeclareMathOperator{\ind}{ind}

\DeclareMathOperator{\Id}{Id}

\DeclareMathOperator{\coker}{coker}

\newstarcommand*{\defeq}{\ensuremath{\phantom{\mathrel{\mathop:}}=}}{\ensuremath{\mathrel{\mathop:}=}}
\newstarcommand*{\eqdef}{\ensuremath{=\phantom{\mathrel{\mathop:}}}}{\ensuremath{=\mathrel{\mathop:}}}
\newcommand*{\dash}[1]{\nobreakdash#1\hspace{0pt}}
\newcommand*{\normal}[2][]{\ensuremath{\nu_{#1}#2}}
\newcommand*{\normalhat}[2][]{\ensuremath{\hat{\nu}_{#1}#2}}
\newcommand*{\normaltilde}[2][]{\ensuremath{\tilde{\nu}_{#1}#2}}
\newcommand*{\scriptwidetilde}[1]{\raisebox{-0.4ex}{\ensuremath{\scriptstyle\mathrlap{\,\widetilde{\phantom{#1}}}}}#1}

\makeatletter
\newstarcommand*{\norm}{\@norm{{}\cdot{}}}{\@norm}
\DeclarePairedDelimiter{\@norm}{\lVert}{\rVert}
\newstarcommand*{\abs}{\@abs{{}\cdot{}}}{\@abs}
\DeclarePairedDelimiter{\@abs}{\lvert}{\rvert}
\newstarcommand*{\inner}{\@inner{{}\cdot{},{}\cdot{}}}{\@inner}
\DeclarePairedDelimiter{\@inner}{\langle}{\rangle}
\newstarcommand*{\set}{\emptyset}{\@set}
\DeclarePairedDelimiter{\@set}{\lbrace}{\rbrace}
\newstarcommand*{\paren}{\@paren{{}\dots{}}}{\@paren}
\DeclarePairedDelimiter{\@paren}{\lparen}{\rparen}
\newstarcommand*{\bracket}{\@bracket{{}\dots{}}}{\@bracket}
\DeclarePairedDelimiter{\@bracket}{\lbrack}{\rbrack}
\makeatother

\title{Deformations of Compact Cayley Submanifolds with Boundary}
\author{Matthias Ohst \\[1.5ex]
DPMMS, University of Cambridge, United Kingdom \\
\href{mailto:M.Ohst@dpmms.cam.ac.uk}{\nolinkurl{M.Ohst@dpmms.cam.ac.uk}}}
\date{}

\numberwithin{theorem}{section}
\numberwithin{equation}{section}
\hypersetup{
  pdftitle    = {Deformations of Compact Cayley Submanifolds with Boundary},
  pdfauthor   = {Matthias Ohst},
  pdfsubject  = {},
  pdfkeywords = {},
}

\begin{document}

\maketitle

\begin{abstract}
  Let $M$ be an $8$\dash-manifold with a $\Spin(7)$-structure. We first show that closed Cayley submanifolds of~$M$ form a smooth moduli space for a generic $\Spin(7)$-structure. Then we study the deformations of a compact, connected Cayley submanifold~$X$ of~$M$ with non-empty boundary contained in a given submanifold~$W$ of~$M$ such that $X$ and $W$ meet orthogonally. We show that they are rigid for a generic $\Spin(7)$-structure. We further show that they are also rigid for a generic deformation of~$W$.
\end{abstract}

\section{Introduction}
\label{sec:introduction}

Cayley submanifolds are $4$\dash-dimensional submanifolds which may be defined in an $8$\dash-manifold~$M$ equipped with a differential $4$\dash-form~$\Phi$ invariant at each point under the spin representation of $\Spin(7)$. The latter representation identifies $\Spin(7)$ as a subgroup of $\SO(8)$, and a $\Spin(7)$-structure determined by~$\Phi$ induces a Riemannian metric and orientation on~$M$. See \sectionref{sec:preliminaries} for details. Cayley submanifolds of~$\R^8$ were introduced by Harvey and Lawson \cite{HL82} as an instance of calibrated submanifolds, extending the volume-minimising properties of complex submanifolds in Kähler manifolds. Other classes of calibrated submanifolds given in \cite{HL82} are the special Lagrangian submanifolds of~$\C^n$ and the associative and coassociative submanifolds of~$\R^7$.

Calibrated submanifolds often arise in Riemannian manifolds with reduced holonomy. In particular, Cayley submanifolds in an $8$\dash-manifold are calibrated and minimal whenever the respective ``$\Spin(7)$-structure''~$\Phi$ is closed. In that case, the holonomy of the Riemannian metric induced by~$\Phi$ reduces to a subgroup of $\Spin(7)$; in particular, the metric then is Ricci-flat. The first examples of closed Riemannian $8$\dash-manifolds with holonomy $\Spin(7)$ were constructed by Joyce \cite{Joy96}. He also provided examples of closed Cayley submanifolds inside these manifolds \cite{Joy00}.

McLean \cite{McL98} studied the deformations of closed calibrated submanifolds for the calibrations introduced in \cite{HL82}. He showed, among other results, that the deformation problem in each case is elliptic or overdetermined elliptic and described the respective finite-dimensional Zariski tangent spaces. In the case of closed Cayley submanifolds, the deformations may in general be obstructed, and the Zariski tangent space can be given in terms of harmonic spinors of a certain twisted Dirac operator.

Later, other authors extended McLean's results to larger classes of submanifolds, including, in the special Lagrangian, coassociative, and associative cases, compact submanifolds with boundary. See, respectively, Butscher \cite{But03}, Kovalev and Lotay \cite{KL09}, and Gayet and Witt \cite{GW11}. The corresponding moduli spaces and Zariski tangent spaces of deformations were shown to be finite-dimensional if one assumes a suitable constraint on the boundary to lie in an appropriately chosen fixed submanifold.

In this article, we study the deformations of compact Cayley submanifolds with boundary contained in a given submanifold~$W$. We require that the deformations meet the submanifold~$W$ orthogonally, but unlike the previous results for special Lagrangian, coassociative, and associative calibrations, we allow a range of dimensions for~$W$.

Ideally, one would like to get an elliptic first-order boundary problem as for compact associative submanifolds with boundary. But it turns out (as the author checked in his doctoral thesis) that an appropriate first-order boundary problem will not be elliptic. Instead, we will consider a second-order boundary problem. A similar approach was used in \cite{KL09} for the deformation theory of compact coassociative submanifolds with boundary. On the other hand, we prove genericity results similar to the those in \cite{Gay14} for compact associative submanifolds with boundary.

We start with some preliminaries in \sectionref{sec:preliminaries}. Then we present the deformation theory of closed Cayley submanifolds in \sectionref{sec:cayley-closed}. After reviewing McLean's result and its generalisation to $\Spin(7)$-structures with torsion, we prove the following theorem in \sectionref{subsec:varying-spin7}. The proof largely amounts to finding how McLean's deformation map depends on the $\Spin(7)$-structure.

\begin{theorem} \label{thm:main-spin7-closed}
  Let $M$ be an $8$\dash-manifold with a $\Spin(7)$-structure~$\Phi$, and let $X$ be a closed Cayley submanifold of~$M$.
  
  Then for every generic $\Spin(7)$\dash-structure~$\Psi$ that is close to~$\Phi$, the moduli space of all Cayley submanifolds of $(M, \Psi)$ that are close to~$X$ is either empty or a smooth manifold of dimension $\ind D$, where $D$ is an operator of Dirac type and defined in~\eqref{eq:def-D}.
\end{theorem}

Note that $X$ itself need not be in the moduli space since $X$ need not be Cayley with respect to the $\Spin(7)$\dash-structure~$\Psi$. Also, if $\ind D < 0$, then the moduli space is necessarily empty for a generic $\Spin(7)$\dash-structure~$\Psi$.

In \sectionref{sec:cayley-boundary} we investigate the deformation theory of compact Cayley submanifolds with boundary. We first prove the following proposition in \sectionref{subsec:cayley-boundary-deformations}, which is a key technical result needed for the proof of the next two theorems.

\begin{proposition} \label{prop:main-boundary}
  Let $M$ be an $8$\dash-manifold with a $\Spin(7)$-structure, let $X$ be a compact, connected Cayley submanifold of~$M$ with non-empty boundary, and let $W$ be a submanifold of~$M$ with $\partial X \subseteq W$ such that $X$ and $W$ meet orthogonally.
  
  Then the moduli space of all local deformations of~$X$ as a Cayley submanifold of~$M$ with boundary on~$W$ and meeting~$W$ orthogonally can be embedded into the solution space of the boundary problem~\eqref{eq:boundary-second}, which is a second-order elliptic boundary problem with index~$0$.
\end{proposition}

In particular, the Zariski tangent space at~$X$ is finite-dimensional. The usefulness of the boundary problem~\eqref{eq:boundary-second} lies in the fact that the solution space is a smooth manifold under appropriate genericity assumptions. In particular, we show the following two theorems (they have the same hypotheses as in \propositionref{prop:main-boundary}) in \hyperref[subsec:varying-spin7-boundary]{Sections~\ref*{subsec:varying-spin7-boundary}} and \ref{subsec:varying-scaffold}, respectively, where we vary the $\Spin(7)$-structure (\theoremref{thm:main-spin7-boundary}) and the submanifold~$W$ (\theoremref{thm:main-scaffold}).

\begin{theorem} \label{thm:main-spin7-boundary}
  Let $M$ be an $8$\dash-manifold with a $\Spin(7)$-structure~$\Phi$, let $X$ be a compact, connected Cayley submanifold of~$M$ with non-empty boundary, and let $W$ be a submanifold of~$M$ with $\partial X \subseteq W$ such that $X$ and $W$ meet orthogonally.
  
  Then for every generic $\Spin(7)$\dash-structure~$\Psi$ that is close to~$\Phi$, the moduli space of all Cayley submanifolds of $(M, \Psi)$ that are close to~$X$ with boundary on~$W$ and meeting~$W$ orthogonally (with respect to the metric induced by~$\Psi$) is a finite set (possibly empty).
\end{theorem}

\begin{theorem} \label{thm:main-scaffold}
  Let $M$ be an $8$\dash-manifold with a $\Spin(7)$-structure, let $X$ be a compact, connected Cayley submanifold of~$M$ with non-empty boundary, and let $W$ be a submanifold of~$M$ with $\partial X \subseteq W$ such that $X$ and $W$ meet orthogonally.
  
  Then for every generic local deformation~$W^\prime$ of~$W$, the moduli space of all Cayley submanifolds of~$M$ that are close to~$X$ with boundary on~$W^\prime$ and meeting~$W^\prime$ orthogonally is a finite set (possibly empty).
\end{theorem}

\begin{remarks}
  \mbox{}
  \begin{compactenum}[(i)]
    \item The statements of \hyperref[thm:main-spin7-closed]{Theorems~\ref*{thm:main-spin7-closed}} and~\ref{thm:main-spin7-boundary} remain true if we restrict to the smaller class of all $\Spin(7)$\dash-structures~$\Psi$ inducing the same metric as~$\Phi$.
    \item The statement of \theoremref{thm:main-spin7-boundary} remains true if we restrict to the smaller class of all torsion-free $\Spin(7)$\dash-structures~$\Psi$ (assuming $\Phi$ is torsion-free).
    \item The statement of \theoremref{thm:main-scaffold} remains true if we restrict to the smaller class of all local deformations~$W^\prime$ of~$W$ with $\partial X \subseteq W^\prime$.
    \item If we allow to vary both~$\Phi$ and~$W$, then we also get a genericity statement.
  \end{compactenum}
\end{remarks}

We give some examples for this deformation theory in \sectionref{sec:examples}. We further show versions of the volume minimising property and relate our deformation theory to the deformation theories of compact special Lagrangian, coassociative, and associative submanifolds with boundary.

\paragraph{Acknowledgements.} I am grateful to my Ph.D. supervisor, Alexei Kovalev, for his support and guidance.

\section{Preliminaries}
\label{sec:preliminaries}

In this section we recall some basic facts about $\Spin(7)$-structures on $8$\dash-manifolds and Cayley submanifolds (see, for example, \cite{HL82}, \cite{Joy00}).

Let $(x_1, \dotsc, x_8)$ be coordinates on~$\R^8$, and write $\D \vectorsym{x}_{i \dots j}$ for $\D x_i \wedge \dotsb \wedge \D x_j$. Define a $4$\dash-form $\Phi_0$ on~$\R^8$ by
\begin{equation}
  \begin{split}
    \! \Phi_0 &\defeq \D \vectorsym{x}_{1234} + \D \vectorsym{x}_{1256} - \D \vectorsym{x}_{1278} + \D \vectorsym{x}_{1357} + \D \vectorsym{x}_{1368} + \D \vectorsym{x}_{1458} - \D \vectorsym{x}_{1467} \\
    &\phantom{{}\defeq{}} {}- \D \vectorsym{x}_{2358} + \D \vectorsym{x}_{2367} + \D \vectorsym{x}_{2457} + \D \vectorsym{x}_{2468} - \D \vectorsym{x}_{3456} + \D \vectorsym{x}_{3478} + \D \vectorsym{x}_{5678} \, \text{.} \!
  \end{split} \label{eq:def-Phi-0}
\end{equation}
The subgroup of $\GL(8, \R)$ preserving~$\Phi_0$ is isomorphic to $\Spin(7)$, viewed as a subgroup of $\SO(8)$. Note that $\Phi_0$ is self-dual.

Let $M$ be an $8$\dash-manifold. Suppose that there is a $4$\dash-form~$\Phi$ on~$M$ such that for each $x \in M$ there is a linear isomorphism $i_x \colon T_x M \to \R^8$ with $(i_x)^\ast(\Phi_0) = \Phi_x$ (in a neighbourhood of each point, this can be chosen to depend smoothly on~$x$). Then $\Phi$ induces a $\Spin(7)$\dash-structure on~$M$. Conversely, if $M$ has a $\Spin(7)$\dash-structure, then there is such a $4$\dash-form~$\Phi$. Via such an identification $i_x \colon T_x M \to \R^8$ of~$\Phi_x$ with~$\Phi_0$, the metric~$g_0$ of~$\R^8$ induces a metric $(i_x)^\ast(g_0)$ on~$T_x M$. Since $\Spin(7) \subseteq \SO(8)$, this metric is independent of the chosen identification, and we get a well-defined Riemannian metric~$g = g(\Phi)$ and orientation on~$M$. By abuse of notation, we will refer to the $4$\dash-form~$\Phi$ as a $\Spin(7)$\dash-structure. The $\Spin(7)$\dash-structure is called \emph{torsion-free} if $\nabla \Phi = 0$, where $\nabla$ is the Levi-Civita connection of~$(M, g)$. This is equivalent to $\D \Phi = 0$ \cite[Theorem~5.3]{Fer86}.

If $M$ is an $8$\dash-manifold, then there exists a $\Spin(7)$-structure on~$M$ if and only if $M$ is orientable and spin and
\begin{equation}
  p_1(M)^2 - 4 p_2(M) + 8 \chi(M) = 0
\end{equation}
for some orientation of~$M$ \cite[Theorem~10.7 in Chapter~IV]{LM89}, where $p_i(M)$ is the $i$\dash-th Pontryagin class and $\chi(M)$ is the Euler characteristic of~$M$.

Now let $M$ be an $8$\dash-manifold with a $\Spin(7)$-structure~$\Phi$. Then we have pointwise orthogonal splittings \cite[Lemmas~3.1 and 3.3]{Fer86}
\begin{equation}
  \begin{split}
    \altforms^2 M &= \altforms^2_7 M \oplus \altforms^2_{21} M \quad \text{and} \\
    \altforms^4 M &= \altforms^4_1 M \oplus \altforms^4_7 M \oplus \altforms^4_{27} M \oplus \altforms^4_{35} M \, \text{.}
  \end{split} \label{eq:forms-splitting}
\end{equation}
Here $\altforms^p M \defeq \altforms^p T^\ast M$, and $\altforms^k_\ell M$ corresponds to an irreducible representation of~$\Spin(7)$ of dimension~$\ell$. Furthermore, $M$ possesses a $2$\dash-fold cross product $T M \times T M \to \altforms^2_7 M$,
\begin{equation}
  v \times w \defeq 2 \pi_7(v^\flat \wedge w^\flat) = \tfrac{1}{2} (v^\flat \wedge w^\flat - \mathord\ast (v^\flat \wedge w^\flat \wedge \Phi)) \label{eq:def-cross-2}
\end{equation}
and a $3$\dash-fold cross product $T M \times T M \times T M \to T M$,
\begin{equation}
  u \times v \times w \defeq (u \interior (v \interior (w \interior \Phi)))^\sharp \, \text{.} \label{eq:def-cross-3}
\end{equation}
They satisfy $\abs{v \times w} = \abs{v \wedge w}$, $\abs{u \times v \times w} = \abs{u \wedge v \wedge w}$, and
\begin{equation}
  h(a \times b, c \times d) = - \Phi(a, b, c, d) + g(a, c) g(b, d) - g(a, d) g(b, c) \, \text{,} \label{eq:inner-cross-2}
\end{equation}
where $h$ is the induced metric on~$\altforms^2_7 M$. There is also a vector-valued $4$\dash-form $\tau \in \forms^4(M, \altforms^2_7 M)$ (also called the $4$\dash-fold cross product),
\begin{equation}
  \tau(a, b, c, d) \defeq - a \times (b \times c \times d) + g(a, b) (c \times d) + g(a, c) (d \times b) + g(a, d) (b \times c) \, \text{,} \label{eq:def-tau}
\end{equation}
which satisfies
\begin{equation}
  h(\tau, v \times w) = w^\flat \wedge (v \interior \Phi) - v^\flat \wedge (w \interior \Phi) \, \text{.} \label{eq:inner-tau}
\end{equation}
This formula can be checked by using the invariance properties of the cross products and checking it for $v = e_1$, $w = e_2$, where $(e_1, \dotsc, e_8)$ is a $\Spin(7)$\dash-frame (see below for the definition). Note that
\begin{equation}
  w^\flat \wedge (v \interior \Phi) - v^\flat \wedge (w \interior \Phi) \in \altforms^4_7 M \, \text{,} \label{eq:forms-4-7}
\end{equation}
which follows from \cite[page~548]{Bry87}.

Let $x \in M$, and let $(e_1, \dotsc, e_8)$ be a basis of~$T_x M$. We call $(e_1, \dotsc, e_8)$ a \emph{$\Spin(7)$\dash-frame} if
\begin{equation}
  \begin{split}
    \Phi &= e^{1234} + e^{1256} - e^{1278} + e^{1357} + e^{1368} + e^{1458} - e^{1467} \\
         &\phantom{{}={}} {}- e^{2358} + e^{2367} + e^{2457} + e^{2468} - e^{3456} + e^{3478} + e^{5678} \, \text{,}
  \end{split} \label{eq:phi-spin7-frame}
\end{equation}
where $(e^1, \dotsc, e^8)$ is the dual coframe. Note that if $(e_1, \dotsc, e_8)$ is a $\Spin(7)$-frame, then it is an orthonormal frame since $\Spin(7) \subseteq \SO(8)$. Furthermore, $e_i \times e_j = \pm e_k \times e_\ell$ if and only if $\Phi(e_i, e_j, e_k, e_\ell) = \mp 1$ for $i, j, k, \ell \in \set{1, \dotsc, 8}$ different by~\eqref{eq:inner-cross-2}. So \eqref{eq:phi-spin7-frame} shows that
\begin{equation}
  \begin{split}
    e_1 \times e_5 &= \phantom{+} e_2 \times e_6 = \phantom{+} e_3 \times e_7 = \phantom{+} e_4 \times e_8 \, \text{,} \\
    e_1 \times e_6 &= - e_2 \times e_5 = \phantom{+} e_3 \times e_8 = - e_4 \times e_7 \, \text{,} \\
    e_1 \times e_7 &= - e_2 \times e_8 = - e_3 \times e_5 = \phantom{+} e_4 \times e_6 \, \text{,} \\
    e_1 \times e_8 &= \phantom{+} e_2 \times e_7 = - e_3 \times e_6 = - e_4 \times e_5 \, \text{.}
  \end{split} \label{eq:cross-product-table}
\end{equation}

If $e_1, e_2, e_3 \in T_x M$ are orthogonal unit vectors and $e_5 \in T_x M$ is a unit vector that is orthogonal to~$e_1, e_2, e_3, e_1 \times e_2 \times e_3$, then there are (uniquely determined) $e_4, e_6, e_7, e_8 \in T_x M$ such that $(e_1, \dotsc, e_8)$ is a $\Spin(7)$\dash-frame, namely $e_4 = - e_1 \times e_2 \times e_3$, $e_6 = - e_1 \times e_2 \times e_5$, $e_7 = - e_1 \times e_3 \times e_5$, and $e_8 = e_2 \times e_3 \times e_5$.

We have $\Phi_x \vert_V \le \vol_V$ for all~$x \in M$ and every oriented $4$\dash-dimensional subspace~$V$ of~$T_x M$, where $\vol_V$ is the volume form (induced by the metric~$g$ and the orientation on~$V$) and $\varphi_x \vert_V \le \vol_V$ means that $\varphi_x \vert_V = \lambda \, \vol_V$ with~$\lambda \le 1$. An orientable $4$\dash-dimensional submanifold~$X$ of~$M$ is called \emph{Cayley} if $\Phi \vert_X = \vol_X$ for some orientation of~$X$. This is equivalent to~$\tau \vert_X = 0$ \cite[Corollary~1.29 in Chapter~IV]{HL82}. If the $\Spin(7)$-structure~$\Phi$ is torsion-free, then $\Phi$ is a calibration on~$M$, and Cayley submanifolds are minimal submanifolds \cite[Theorem~4.2 in Chapter~II]{HL82}.

Now suppose that $X$ is a Cayley submanifold of~$M$. Then $\altforms^2_- X$ is isomorphic to a subbundle of~$\altforms^2_7 M \vert_X$ \cite[Section~6]{McL98} via the embedding
\begin{equation}
  \altforms^2_- X \to \altforms^2_7 M \vert_X \, \text{,} \quad \alpha \mapsto 2 \pi_7(\alpha) = \tfrac{1}{2} (\alpha - \mathord\ast (\alpha \wedge \Phi)) \, \text{,} \label{eq:subbundle-2-7}
\end{equation}
where we extend $\alpha \in \altforms^2_- X$ to $\altforms^2 M \vert_X$ by $v \interior \alpha = 0$ for all~$v \in \normal[M]{X}$. Let $E$ denote the orthogonal complement of~$\altforms^2_- X$ in~$\altforms^2_7 M \vert_X$. So
\begin{equation}
  \altforms^2_7 M \vert_X \cong \altforms^2_- X \oplus E \, \text{,} \label{eq:splitting-2-7}
\end{equation}
and $E$ has rank~$4$. Furthermore,
\begin{equation}
  E = \set{\alpha \in \altforms^2_7 M \vert_X \colon \alpha \vert_{T X} = 0} \, \text{.} \label{eq:def-E}
\end{equation}
The cross products restrict to
\begin{equation}
  T X \times T X \to \altforms^2_- X \, \text{,} \quad T X \times \normal[M]{X} \to E \, \text{,} \quad \normal[M]{X} \times \normal[M]{X} \to \altforms^2_- X \label{eq:cross-2-restriction}
\end{equation}
and
\begin{equation}
  \begin{split}
    &\phantom{T X \times \normal[M]{X} \times \normal[M]{X} \to T X \, \text{,} \quad}\mathllap{T X \times T X \times T X \to T X \, \text{,} \quad} \phantom{\normal[M]{X} \times \normal[M]{X} \times \normal[M]{X} \to \normal[M]{X}} \mathllap{T X \times T X \times \normal[M]{X} \to \normal[M]{X}} \, \text{,} \\
    &T X \times \normal[M]{X} \times \normal[M]{X} \to T X \, \text{,} \quad \normal[M]{X} \times \normal[M]{X} \times \normal[M]{X} \to \normal[M]{X} \, \text{.}
  \end{split} \label{eq:cross-3-restriction}
\end{equation}

\section{Closed Cayley Submanifolds}
\label{sec:cayley-closed}

In this section we present the deformation theory of closed Cayley submanifolds. We first review in \sectionref{subsec:cayley-closed-deformations} McLean's foundational result and its extension to $\Spin(7)$-structures that are not necessarily torsion-free. Then we prove that for a generic $\Spin(7)$\dash-structure, closed Cayley submanifolds form a smooth moduli space in \sectionref{subsec:varying-spin7}. We end this section with a remark (\sectionref{subsec:remark-torsion-free-closed}) about why we use the wider space of all $\Spin(7)$-structures in \theoremref{thm:main-spin7-closed} rather than just torsion-free $\Spin(7)$-structures.

\subsection{Deformations of Closed Cayley Submanifolds}
\label{subsec:cayley-closed-deformations}

McLean \cite{McL98} proved that if $M$ is an $8$\dash-manifold with a torsion-free $\Spin(7)$-structure and $X$ is a closed Cayley submanifold of~$M$, then the Zariski tangent space to the moduli space of all local deformations of~$X$ as a Cayley submanifold of~$M$ is isomorphic to the space of harmonic twisted spinors. Gutowski, Ivanov, and Papadopoulos \cite{GIP03} generalised this result to manifolds with $\Spin(7)$\dash-structures that have torsion. Here we give an explicit formula for the operator of Dirac type that arises as the linearisation of the deformation map in terms of the Levi-Civita connection. The equivalence of the operators can also be seen using the explicit formula for a $\Spin(7)$\dash-connection in \cite[Theorem~1.1]{Iva04}.

\begin{theorem}[{cf.~\cite[Theorem~6--3]{McL98}, \cite[Section~13]{GIP03}}] \label{thm:explicit-formula}
  Let $M$ be an $8$\dash-manifold with a $\Spin(7)$-structure~$\Phi$, and let $X$ be a closed Cayley submanifold of~$M$.
  
  Then the Zariski tangent space to the moduli space of all local deformations of~$X$ as a Cayley submanifold of~$M$ can be identified with the kernel of the operator $D \colon \sections(\normal[M]{X}) \to \sections(E)$,
  \begin{equation}
    D s \defeq \sum_{i = 1}^4 e_i \times \nabla_{e_i}^\perp s + \sum_{i = 5}^8 (\nabla_s \Phi)(e_i, e_2, e_3, e_4) (e_i \times e_1) \, \text{,} \label{eq:def-D}
  \end{equation}
  where $E$ is a vector bundle of rank~$4$ over~$X$ as defined in~\eqref{eq:def-E}, $(e_i)_{i = 1, \dotsc, 4}$ is any positive local orthonormal frame of~$X$, $(e_i)_{i = 5, \dotsc, 8}$ is any local orthonormal frame of~$\normal[M]{X}$, and $\nabla^\perp$ is the induced connection on~$\normal[M]{X}$.
\end{theorem}

Note that the second part of the operator~\eqref{eq:def-D} is an operator of order~$0$ that vanishes if the $\Spin(7)$\dash-structure~$\Phi$ is torsion-free (i.e.,~if~$\nabla \Phi = 0$).

\begin{remark}
  If $X$ has a spin structure, then one can identify $E \cong \mathbb{S}_+ \otimes_\HH F$ and $\normal[M]{X} \cong \mathbb{S}_- \otimes_\HH F$ for some vector bundle~$F$ of rank~$4$ over~$X$ as in~\cite[Section~6]{McL98}. If the $\Spin(7)$\dash-structure is torsion-free, then $D$ can be identified with a twisted Dirac operator \cite[Theorem~6--3]{McL98}, that is, the Dirac operator associated to the bundle $\mathbb{S} \otimes_\HH F$ with the tensor product connection (cf. \cite[Proposition~5.10 in Chapter~II]{LM89}).
\end{remark}

We now give a proof of \theoremref{thm:explicit-formula}. This proof is mostly based on \cite{McL98}. Besides introducing some notation, the proofs of our main theorems also build up on this proof. In particular, we need the deformation map from \cite{McL98} later.

For an open subset~$U$ of a vector bundle~$E$ over a manifold~$X$, define
\begin{equation}
  \sections(U) \defeq \set{s \in \sections(E) \colon s(X) \subseteq U} \, \text{.} \label{eq:def-sections-subset}
\end{equation}
We will use a similar definition for the Hölder spaces $\CC^{k, \alpha}$ with $k \ge 0$ an integer and $0 \le \alpha \le 1$.

\begin{proof}
  The Tubular Neighbourhood Theorem (cf. the proof of \cite[Theorem IV.5.1]{Lan95}) asserts that there is an open tubular neighbourhood $U \subseteq \normal[M]{X}$ of the $0$\dash-section such that the exponential map $\exp \vert_U \colon U \to \exp(U)$ is a diffeomorphism. For a normal vector field $s \in \sections(U)$, let~$\exp_s \colon X \to M$, $x \mapsto \exp_x(s(x))$. Furthermore, let
  \begin{equation}
    F \colon \sections(U) \to \forms^4(X, \altforms^2_7 M \vert_X) \, \text{,} \quad s \mapsto (\exp_s)^\ast(\tau) \, \text{,}
  \end{equation}
  where $\tau \in \forms^4(M, \altforms^2_7 M)$ is defined as in~\eqref{eq:def-tau}.
  
  Local deformations of~$X$ are parametrised by sections $s \in \sections(U)$ via the exponential map. With this identification, $X_s \defeq \exp_s(X)$ is Cayley if and only if $(\exp_s)^\ast(\tau) = \tau \vert_{X_s} = 0$. So the moduli space of all local deformations of~$X$ as a Cayley submanifold of~$M$ can be identified with $F^{-1}(\set{0})$.
  
  Let $s \in \sections(\normal[M]{X})$. We have
  \begin{equation}
    (\D F)_0(s) = \frac{\D}{\D t} F(t s) \bigg\vert_{t = 0} = \frac{\D}{\D t} (\exp_{t s})^\ast(\tau) \bigg\vert_{t = 0} = (\Lie_s \tau) \vert_X \, \text{.}
  \end{equation}
  Let $(e_1, \dotsc, e_8)$ be a local orthonormal frame of~$M$ such that $(e_1, \dotsc, e_4)$ is a positive orthonormal frame of~$X$. Then
  \begin{equation}
    \tau = \sum_{i = 2}^8 (e^i \wedge (e_1 \interior \Phi) - e^1 \wedge (e_i \interior \Phi)) \otimes (e_1 \times e_i)
  \end{equation}
  by~\eqref{eq:inner-tau}, where $(e^1, \dotsc, e^8)$ is the dual coframe of~$(e_1, \dotsc, e_8)$. Let
  \begin{equation*}
    \tau_i \defeq e^i \wedge (e_1 \interior \Phi) - e^1 \wedge (e_i \interior \Phi) \, \text{.}
  \end{equation*}
  So $\tau = \sum_{i = 2}^8 \tau_i \otimes (e_1 \times e_i)$. Then
  \begin{align*}
    (\Lie_s \tau) \vert_X &= \sum_{i = 2}^8 (\Lie_s \tau_i) \vert_X \otimes (e_1 \times e_i) + \sum_{i = 2}^8 \tau_i \vert_X \otimes \tilde{\nabla}_s (e_1 \times e_i) \\
    &= \sum_{i = 2}^8 (\Lie_s \tau_i) \vert_X \otimes (e_1 \times e_i)
  \end{align*}
  since $\tau_i \vert_X = 0$ as $X$ is Cayley, where $\tilde{\nabla}$ is the induced connection on~$\altforms^2_7 M$. In particular, the result is independent of~$\tilde{\nabla}$. Now
  \begin{align*}
    \Lie_s \tau_i &= (\Lie_s e^i) \wedge (e_1 \interior \Phi) + e^i \wedge ((\Lie_s e_1) \interior \Phi) + e^i \wedge (e_1 \interior \Lie_s \Phi) \\
    &\phantom{{}={}} {}- (\Lie_s e^1) \wedge (e_i \interior \Phi) - e^1 \wedge ((\Lie_s e_i) \interior \Phi) - e^1 \wedge (e_i \interior \Lie_s \Phi) \, \text{.}
  \end{align*}
  Note that $(\Lie_s e^i)(e_j) = s . (e^i(e_j)) - e^i(\Lie_s e_j) = - g(\Lie_s e_j, e_i)$ for~$i, j = 1, \dotsc, 8$ and that
  \begin{equation*}
    (e^i \wedge (e_j \interior \Phi)) \vert_X =
    \begin{cases}
      \delta_{i j} \vol_X & \text{for~$i = 1, \dotsc, 4$, $j = 1, \dotsc, 8$,} \\
      0 & \text{for~$i = 5, \dotsc, 8$, $j = 1, \dotsc, 8$.}
    \end{cases}
  \end{equation*}
  Indeed, $e^i \vert_X = 0$ for~$i = 5, \dotsc, 8$, and if $i \le 4$, let $(f_1, f_2, f_3)$ be such that $(e_i, f_1, f_2, f_3)$ is a positive orthonormal frame of $X$. Then
  \begin{align*}
    (e^i \wedge (e_j \interior \Phi))(e_i, f_1, f_2, f_3) &= \Phi(e_j, f_1, f_2, f_3) = g(e_j, f_1 \times f_2 \times f_3) \\
    &= g(e_j, e_i) = \delta_{i j}
  \end{align*}
  for~$j = 1, \dotsc, 8$.
  
  So
  \begin{align*}
    ((\Lie_s e^i) \wedge (e_1 \interior \Phi)) \vert_X &= - \sum_{j = 1}^8 g(\Lie_s e_j, e_i) (e^j \wedge (e_1 \interior \Phi)) \vert_X \\
    &= - g(\Lie_s e_1, e_i) \vol_X \quad \text{for~$i = 2, \dotsc, 8$,} \\
    (e^i \wedge ((\Lie_s e_1) \interior \Phi)) \vert_X &= \sum_{j = 1}^8 g(\Lie_s e_1, e_j) (e^i \wedge (e_j \interior \Phi)) \vert_X \\
    &=
    \begin{cases}
      g(\Lie_s e_1, e_i) \vol_X & \text{for~$i = 2, 3, 4$,} \\
      0 & \text{for~$i = 5, \dotsc, 8$,}
    \end{cases} \\
    - ((\Lie_s e^1) \wedge (e_i \interior \Phi)) \vert_X &= \sum_{j = 1}^8 g(\Lie_s e_j, e_1) (e^j \wedge (e_i \interior \Phi)) \vert_X \\
    &=
    \begin{cases}
      g(\Lie_s e_i, e_1) \vol_X & \text{for~$i = 2, 3, 4$,} \\
      0 & \text{for~$i = 5, \dotsc, 8$,}
    \end{cases} \\
    - (e^1 \wedge ((\Lie_s e_i) \interior \Phi)) \vert_X &= - \sum_{j = 1}^8 g(\Lie_s e_i, e_j) (e^1 \wedge (e_j \interior \Phi)) \vert_X \\
    &= - g(\Lie_s e_i, e_1) \vol_X \quad \text{for~$i = 2, \dotsc, 8$.}
  \end{align*}
  Hence
  \begin{equation*}
    ((\Lie_s e^i) \wedge (e_1 \interior \Phi) + e^i \wedge ((\Lie_s e_1) \interior \Phi)) \vert_X =
    \begin{cases}
      0 & \text{for~$i = 2, 3, 4$,} \\
      - g(\Lie_s e_1, e_i) \vol_X & \text{for~$i = 5, \dotsc, 8$}
    \end{cases}
  \end{equation*}
  and
  \begin{equation*}
    - ((\Lie_s e^1) \wedge (e_i \interior \Phi) + e^1 \wedge ((\Lie_s e_i) \interior \Phi)) \vert_X =
    \begin{cases}
      0 & \text{for~$i = 2, 3, 4$,} \\
      - g(\Lie_s e_i, e_1) \vol_X & \text{for~$i = 5, \dotsc, 8$.}
    \end{cases}
  \end{equation*}
  Since $\nabla$ is the Levi-Civita connection of~$(M, g)$,
  \begin{align*}
    - g(\Lie_s e_1, e_i) - g(\Lie_s e_i, e_1) &= s . (g(e_1, e_i)) + g(\nabla_{e_1} s, e_i) + g(\nabla_{e_i} s, e_1) \\
    &= g(\nabla_{e_1} s, e_i) + g(\nabla_{e_i} s, e_1) \, \text{.}
  \end{align*}
  Note that
  \begin{equation*}
    (e^i \wedge (e_j \interior \Lie_s \Phi)) \vert_X = e^i \wedge (e_j \interior (\Lie_s \Phi) \vert_X) = \delta_{i j} (\Lie_s \Phi) \vert_X \quad \text{for~$i, j = 1, \dotsc, 4$}
  \end{equation*}
  since $e^i \wedge (e_j \interior \vol_X) = \delta_{i j} \vol_X$ for~$i, j = 1, \dotsc, 4$. So
  \begin{equation*}
    (e^i \wedge (e_1 \interior \Lie_s \Phi) - e^1 \wedge (e_i \interior \Lie_s \Phi)) \vert_X =
    \begin{cases}
      0 & \text{for~$i = 2, 3, 4$,} \\
      - e^1 \wedge (e_i \interior \Lie_s \Phi) \vert_X & \text{for~$i = 5, \dotsc, 8$.}
    \end{cases}
  \end{equation*}
  Let $i \in \set{5, \dotsc, 8}$. Now
  \begin{align*}
    (\Lie_s \Phi)(e_i, e_2, e_3, e_4) &= (\nabla_s \Phi)(e_i, e_2, e_3, e_4) + \Phi(\nabla_{e_i} s, e_2, e_3, e_4) \\
    &\phantom{{}={}} {}+ \Phi(e_i, \nabla_{e_2} s, e_3, e_4) + \Phi(e_i, e_2, \nabla_{e_3} s, e_4) \\
    &\phantom{{}={}} {}+ \Phi(e_i, e_2, e_3, \nabla_{e_4} s) \, \text{.}
  \end{align*}
  We have
  \begin{equation*}
    \Phi(\nabla_{e_i} s, e_2, e_3, e_4) = g(\nabla_{e_i} s, e_2 \times e_3 \times e_4) = g(\nabla_{e_i} s, e_1)
  \end{equation*}
  and
  \begin{align*}
    \Phi(e_i, \nabla_{e_2} s, e_3, e_4) &= - g(\nabla_{e_2} s, e_i \times e_3 \times e_4) \, \text{,} \\
    \Phi(e_i, e_2, \nabla_{e_3} s, e_4) &= - g(\nabla_{e_3} s, e_i \times e_4 \times e_2) \, \text{,} \\
    \Phi(e_i, e_2, e_3, \nabla_{e_4} s) &= - g(\nabla_{e_4} s, e_i \times e_2 \times e_3) \, \text{.}
  \end{align*}
  Further,
  \begin{align*}
    e_2 \times (e_i \times e_3 \times e_4) &= e_3 \times (e_i \times e_4 \times e_2) = e_4 \times (e_i \times e_2 \times e_3) \\
    &= \tau(e_i, e_2, e_3, e_4) = - e_i \times (e_2 \times e_3 \times e_4) = e_1 \times e_i
  \end{align*}
  by~\eqref{eq:def-tau} since $\tau$ is alternating.
  
  So together we get
  \begin{align*}
    &(\Lie_s \tau)(e_1, e_2, e_3, e_4) \\
    &= \sum_{i = 5}^8 (g(\nabla_{e_1} s, e_i) + g(\nabla_{e_i} s, e_1) - (\Lie_s \Phi)(e_i, e_2, e_3, e_4)) (e_1 \times e_i) \\
    &= \sum_{i = 5}^8 (g(\nabla_{e_1} s, e_i) + g(\nabla_{e_i} s, e_1) - (\nabla_s \Phi)(e_i, e_2, e_3, e_4)) (e_1 \times e_i) \\
    &\phantom{{}={}} {}- \sum_{i = 5}^8 (\Phi(\nabla_{e_i} s, e_2, e_3, e_4) + \Phi(e_i, \nabla_{e_2} s, e_3, e_4)) (e_1 \times e_i) \\
    &\phantom{{}={}} {}- \sum_{i = 5}^8 (\Phi(e_i, e_2, \nabla_{e_3} s, e_4) + \Phi(e_i, e_2, e_3, \nabla_{e_4} s)) (e_1 \times e_i) \\
    &= \sum_{i = 5}^8 g(\nabla_{e_1} s, e_i) (e_1 \times e_i) + \sum_{i = 5}^8 (\nabla_s \Phi)(e_i, e_2, e_3, e_4) (e_i \times e_1) \\
    &\phantom{{}={}} {}+ \sum_{i = 5}^8 g(\nabla_{e_2} s, e_i \times e_3 \times e_4) (e_2 \times (e_i \times e_3 \times e_4)) \\
    &\phantom{{}={}} {}+ \sum_{i = 5}^8 g(\nabla_{e_3} s, e_i \times e_4 \times e_2) (e_3 \times (e_i \times e_4 \times e_2)) \\
    &\phantom{{}={}} {}+ \sum_{i = 5}^8 g(\nabla_{e_4} s, e_i \times e_2 \times e_3) (e_4 \times (e_i \times e_2 \times e_3)) \\
    &= \sum_{i = 1}^4 \sum_{j = 5}^8 g(\nabla_{e_i} s, e_j) (e_i \times e_j) + \sum_{i = 5}^8 (\nabla_s \Phi)(e_i, e_2, e_3, e_4) (e_i \times e_1) \\
    &= \sum_{i = 1}^4 e_i \times \nabla_{e_i}^\perp s + \sum_{i = 5}^8 (\nabla_s \Phi)(e_i, e_2, e_3, e_4) (e_i \times e_1)
  \end{align*}
  since $(e_i \times e_k \times e_\ell)_{i = 5, \dotsc, 8}$ is an orthonormal frame of~$\normal[M]{X}$ for~$k, \ell \in \set{2, 3, 4}$ with~$k \ne \ell$. Hence
  \begin{equation}
    (\D F)_0(s) = \vol_X \otimes (0, D s) \, \text{,} \label{eq:linearisation-cayley}
  \end{equation}
  where we used the splitting~\eqref{eq:splitting-2-7}, that is, $(0, D s) \in \altforms^2_- X \oplus E \cong \altforms^2_7 M \vert_X$.
\end{proof}

The following proposition can be done by similar methods as in \cite{McL98}.

\begin{proposition} \label{prop:closed-surjective-smooth}
  Let $M$ be a smooth $8$\dash-manifold with a smooth $\Spin(7)$\dash-structure, let $X$ be a smooth closed Cayley submanifold of~$M$, and let $0 < \alpha < 1$.
  
  If the operator $D \colon \sections(\normal[M]{X}) \to \sections(E)$ defined in~\eqref{eq:def-D} is surjective, then the moduli space of all smooth Cayley submanifolds of~$M$ that are $\CC^{1, \alpha}$\dash-close to~$X$ is a smooth manifold of dimension $\dim \ker D = \ind D$.
\end{proposition}

\begin{proof}
  If $Y$ is a $4$\dash-dimensional submanifold of~$M$ that is $\CC^1$\dash-close to~$X$, then $\altforms^2_- Y \to \altforms^2_7 M \vert_Y$, $\alpha \mapsto 2 \pi_7(\alpha)$ is injective. So we can regard $\altforms^2_-Y$ as a subbundle of~$\altforms^2_7 M \vert_Y$. Note that $\tau \vert_Y$ is orthogonal to $\altforms^2_- Y$ by~\eqref{eq:inner-tau} since $(v^\flat \wedge (w \interior \Phi)) \vert_Y = g(v, w) \vol_Y$ for $v, w \in \sections(T Y)$. In particular, if
  \begin{equation}
    \pi_E \colon \altforms^2_7 M \vert_X \to E
  \end{equation}
  denotes the orthogonal projection, then $\pi_E((\exp_s)^\ast(\tau)) = 0$ implies $(\exp_s)^\ast(\tau) = 0$ for $s \in \sections(U)$, where $\tau \in \forms^4(M, \altforms^2_7 M)$ is defined as in~\eqref{eq:def-tau} and $U \subseteq \normal[M]{X}$ is defined as in the proof of \theoremref{thm:explicit-formula}. So if we modify the definition of~$F$ from the proof of \theoremref{thm:explicit-formula} to
  \begin{equation}
    F \colon \sections(U) \to \forms^4(X, E) \cong \sections(E) \, \text{,} \quad s \mapsto \pi_E((\exp_s)^\ast(\tau)) \, \text{,}
  \end{equation}
  then the moduli space of all local deformations of~$X$ as a Cayley submanifold of~$M$ can still be identified with $F^{-1}(\set{0})$.
  
  The symbol~$\sigma_D$ of~$D$ is given by $\sigma_D(x, \xi) s = \sum_{i = 1}^4 e_i \times \xi_i s = \xi \times s$ for $x \in X$, $\xi \in T_x X$. So $\sigma_D(x, \xi) \colon (\normal[M]{X})_x \to E_x$ is bijective if $\xi \ne 0$. Hence $D$ is an elliptic operator. As we will see in the the \hyperref[subsec:smoothness]{next section} (\corollaryref{cor:smooth}), $F$ extends to a smooth map $F_{1, \alpha} \colon \CC^{1, \alpha}(U) \to \CC^{0, \alpha}(E)$. So we can apply the Implicit Function Theorem to deduce that $(F_{1, \alpha})^{-1}(\set{0})$ is a smooth manifold near~$0$ (in the $\CC^{1, \alpha}$\dash-topology) of dimension $\dim \ker D$. Note that the equation $F_{1, \alpha}(s) = 0$ is a nonlinear partial differential equation of order~$1$ whose linearisation at~$0$ is elliptic, and hence the linearisation is elliptic near~$0$ (in the $\CC^{1, \alpha}$\dash-topology). So all elements of $(F_{1, \alpha})^{-1}(\set{0})$ near~$0$ are smooth by Elliptic Regularity \cite[Theorem~6.8.1]{Mor66}.
\end{proof}

\subsection{Smoothness of the Deformation Map}
\label{subsec:smoothness}

Here we show that the deformation map is smooth as a map between Banach spaces. The following proposition is a corollary of \cite[Theorem~2.2.15]{Bai01}, which we extend to vector-valued forms. Specific smoothness results for the deformations maps for other calibrations can also be found in the literature (e.g., \cite{JS05}, \cite{Gay14}).

\begin{proposition}[{cf.~\cite[Theorem~2.2.15]{Bai01}}] \label{prop:c-ell}
  Let $(M, g)$ be a Riemannian manifold, let $X$ be a compact submanifold of~$M$, let $\mathcal{T}_X \defeq \bigoplus_i (T^\ast X)^{\otimes i}$ and $\mathcal{T}_M \defeq \bigoplus_i (T^\ast M)^{\otimes i}$ be the tensor algebras of~$X$ and~$M$, respectively, let $E$ be a vector bundle over~$M$ equipped with a connection~$\nabla$, let $U$ be an open tubular neighbourhood of the $0$\dash-section in~$\normal[M]{X}$ such that the exponential map defines a diffeomorphism from~$U$ to an open neighbourhood~$V$ of~$X$ in~$M$, let $k \ge 1$, $\ell \ge 0$ be integers, and let $0 \le \alpha \le 1$. For $(x, v) \in U$, let $\pi_{x, v} \colon E_{\exp_x(v)} \to E_x$ denote the parallel transport along the curve $[0, 1] \to M$, $t \mapsto \exp_x((1 - t) v)$ (so for all~$s \in \sections(U)$, the map $\pi_s \colon (\exp_s)^\ast E \to E \vert_X$ is an isomorphism of vector bundles).
  
  Then the map
  \begin{equation}
    \begin{split}
      \Psi \colon \CC^{k, \alpha}(U) \oplus \CC^{k - 1 + \ell, \alpha}((\mathcal{T}_M \otimes E) \vert_V) &\to \CC^{k - 1, \alpha}(\mathcal{T}_X \otimes E \vert_X) \, \text{,} \\
      (s, \Theta) &\mapsto (\id_{\mathcal{T}_X} \! \otimes \pi_s)((\exp_s)^\ast \Theta)
    \end{split}
  \end{equation}
  is of class~$\CC^\ell$.
\end{proposition}

\begin{proof}
  Define $\pi \colon V \to X$ by~$\pi(\exp_x(v)) \defeq x$ for $(x, v) \in U$. Since $X$ is compact, there are open subsets $U_1, \dotsc, U_n \subseteq X$ with $U_1 \cup \dotsb \cup U_n = X$ such that $E \vert_{U_i}$ is trivial for~$i = 1, \dotsc, n$. Let $\tilde{\lambda}_1, \dotsc, \tilde{\lambda}_n \colon X \to \R$ be a partition of unity subordinate to the cover~$(U_1, \dotsc, U_n)$. Extend these to functions $\lambda_1, \dotsc, \lambda_n \colon V \to \R$ by $\lambda_i(y) \defeq \tilde{\lambda}_i(\pi(y))$ for~$y \in V$. Then $(\lambda_1, \dotsc, \lambda_n)$ is a partition of unity subordinate to the cover~$(V_1, \dotsc, V_n)$, where $V_i \defeq \pi^{-1}(U_i)$ for~$i = 1, \dotsc, n$. If the proposition is true for~$\lambda_1 \Theta, \dotsc, \lambda_n \Theta$, then it is also true for~$\Theta$ since $\Psi$ is linear in~$\Theta$. So \wolog\ we may assume that the support of~$\Theta$ is contained in some $V_i$, say in~$V_1$.
  
  Now let $\psi \colon E \vert_{U_1} \to \R^n$ be a trivialising map. Let $\rho \colon X \to \R$ be a smooth function with $\supp(\rho) \subseteq U_1$ and $\rho \vert_{\pi(\supp(\Theta))} \equiv 1$. Let $\tilde{\varphi} \colon E \vert_X \to \R^n$ be defined by $\tilde{\varphi}(x, e) \defeq \rho(x) \psi(x, e)$ for~$x \in U_1$, $e \in E_x$ and $\tilde{\varphi}(x, e) \defeq 0$ for~$x \in X \setminus U_1$, $e \in E_x$. Let $\varphi \colon E \vert_V \to \R^n$ be defined by~$\varphi(y, e) \defeq \tilde{\varphi}(\pi(y), \pi_{x, v}(e))$ for~$y \in V$, $e \in E_y$, where $v \in (\normal[M]{X})_x$ is such that $(x, v) \in U$ and $\exp_{\pi(y)}(v) = y$.
  
  Now $\Psi$ is of class~$\CC^\ell$ if and only if $(\id_{\mathcal{T}_X} \! \otimes \tilde{\varphi}) \circ \Psi$ is of class~$\CC^\ell$ since
  \begin{equation*}
    E \vert_{\pi(\supp(\Theta))} \to \pi(\supp(\Theta)) \times \R^n \, \text{,} \quad (x, e) \mapsto (x, \tilde{\varphi}(x, e))
  \end{equation*}
  is a smooth diffeomorphism. We have
  \begin{align*}
    (\id_{\mathcal{T}_X} \! \otimes \tilde{\varphi})(\Psi(s)) &= (\id_{\mathcal{T}_X} \! \otimes \tilde{\varphi})((\id_{\mathcal{T}_X} \! \otimes \pi_s)((\exp_s)^\ast \Theta)) \\
    &= (\id_{\mathcal{T}_X} \! \otimes \varphi)((\exp_s)^\ast \Theta) \\
    &= (\exp_s)^\ast((\id_{\mathcal{T}_M \vert_V} \! \otimes \varphi)(\Theta)) \, \text{.}
  \end{align*}
  But $(\id_{\mathcal{T}_M \vert_V} \! \otimes \varphi)(\Theta) \in \sections(\mathcal{T}_M \vert_V \otimes \R^n)$, and the map
  \begin{equation*}
    \CC^{k, \alpha}(U) \to \CC^{k - 1, \alpha}(\mathcal{T}_X \otimes \R^n) \, \text{,} \quad s \mapsto (\exp_s)^\ast((\id_{\mathcal{T}_M \vert_V} \! \otimes \varphi)(\Theta))
  \end{equation*}
  is of class~$\CC^\ell$ by \cite[Theorem~2.2.15]{Bai01} (note that the proof of \cite[Theorem~2.2.15]{Bai01} also works if $\Theta$ is just of class $\CC^{k - 1 + \ell, \alpha}$). Now the claim follows as $\Psi$ is linear in~$\Theta$.
\end{proof}

\begin{corollary} \label{cor:smooth}
  Let $(M, g)$ be a Riemannian manifold, let $X$ be a compact submanifold of~$M$, let $\mathcal{T}_X \defeq \bigoplus_i (T^\ast X)^{\otimes i}$ and $\mathcal{T}_M \defeq \bigoplus_i (T^\ast M)^{\otimes i}$ be the tensor algebras of~$X$ and~$M$, respectively, let $E$ be a vector bundle over~$M$ equipped with a connection~$\nabla$, let $U$ be an open tubular neighbourhood of the $0$\dash-section in~$\normal[M]{X}$ such that the exponential map defines a diffeomorphism from~$U$ to an open neighbourhood~$V$ of~$X$ in~$M$, let $\Theta \in \sections((\mathcal{T}_M \otimes E) \vert_V)$, let $k \ge 1$ be an integer, and let $0 \le \alpha \le 1$. For $(x, v) \in U$, let $\pi_{x, v} \colon E_{\exp_x(v)} \to E_x$ denote the parallel transport along the curve $[0, 1] \to M$, $t \mapsto \exp_x((1 - t) v)$ (so for all~$s \in \sections(U)$, the map $\pi_s \colon (\exp_s)^\ast E \to E \vert_X$ is an isomorphism of vector bundles).
  
  Then the map
  \begin{equation}
    \CC^{k, \alpha}(U) \to \CC^{k - 1, \alpha}(\mathcal{T}_X \otimes E \vert_X) \, \text{,} \quad s \mapsto (\id_{\mathcal{T}_X} \! \otimes \pi_s)((\exp_s)^\ast \Theta)
  \end{equation}
  is smooth.
\end{corollary}

\subsection{Varying the Spin(7)-Structure}
\label{subsec:varying-spin7}

Here we prove that for a generic $\Spin(7)$\dash-structure, the moduli space of closed Cayley submanifolds is smooth.

\begin{definition}
  Let $X$ be a topological space. We say that a statement holds for \emph{generic} $x \in X$ if the set of all $x \in X$ for which the statement is true is a residual set, that is, it contains a set which is the intersection of countably many open dense subsets.
\end{definition}

In the following theorem we use the $\CC^\infty$\dash-topology for the space of all $\Spin(7)$-structures.

\begin{theorem}[\theoremref{thm:main-spin7-closed}] \label{thm:main-theorem-spin7-closed}
  Let $M$ be a smooth $8$\dash-manifold with a smooth $\Spin(7)$-structure~$\Phi$, let $X$ be a smooth closed Cayley submanifold of~$M$, and let $0 < \alpha < 1$.
  
  Then for every generic smooth $\Spin(7)$\dash-structure~$\Psi$ that is $\CC^{1, \alpha}$\dash-close to~$\Phi$, the moduli space of all smooth Cayley submanifolds of $(M, \Psi)$ that are $\CC^{1, \alpha}$\dash-close to~$X$ is either empty or a smooth manifold of dimension $\ind D$, where $D$ is defined in~\eqref{eq:def-D}.
\end{theorem}

Note that if $\ind D < 0$, then the moduli space is necessarily empty for a generic $\Spin(7)$\dash-structure. As we will see in the proof, we may restrict the class of $\Spin(7)$\dash-structures~$\Psi$ to those inducing the same metric as~$\Phi$.

\begin{remark}
  The space of smooth sections of a vector bundle over a smooth manifold is a Fréchet space with the $\CC^\infty$\dash-topology, and hence a complete metric space. So a residual set is dense by the Baire Category Theorem.
\end{remark}

\begin{proof}
  Recall from~\eqref{eq:forms-splitting} that we have a pointwise orthogonal splitting
  \begin{equation*}
    \altforms^4 M \cong \altforms^4_1 M \oplus \altforms^4_7 M \oplus \altforms^4_{27} M \oplus \altforms^4_{35} M \, \text{.}
  \end{equation*}
  There are an open tubular neighbourhood $V \subseteq \altforms^4_1 M \oplus \altforms^4_7 M \oplus \altforms^4_{35} M$ of the $0$\dash-section and a smooth bundle morphism $\Theta \colon V \to \altforms^4 M$ (we will also denote the map $\sections(V) \to \forms^4(M)$, $\chi \mapsto \Theta \circ \chi$ by~$\Theta$) such that \cite[Definition~10.5.8 and Proposition~10.5.9]{Joy00}
  \begin{compactenum}[(i)]
    \item $\Theta(0) = \Phi$,
    \item $\Theta(\chi)$ is a $\Spin(7)$\dash-structure on~$M$ for each $\chi \in \sections(V)$,
    \item if $\Psi$ is a $\Spin(7)$\dash-structure on~$M$ that it $\CC^0$\dash-close to~$\Phi$, then $\Psi = \Theta(\chi)$ for some unique $\chi \in \sections(V)$, and
    \item $(\D \Theta)_0(\chi) = \chi$ for all $\chi \in \sections(\altforms^4_1 M \oplus \altforms^4_7 M \oplus \altforms^4_{35} M)$.
  \end{compactenum}
  So $\Theta$ parametrises $\Spin(7)$\dash-structures on~$M$ that are $\CC^0$\dash-close to~$\Phi$.
  
  Let
  \begin{equation}
    \pi_E \colon \altforms^2 M \vert_X \to E
  \end{equation}
  denote the orthogonal projection, where the vector bundle~$E$ of rank~$4$ over~$X$ is defined as in~\eqref{eq:def-E}. Furthermore, let $U \subseteq \normal[M]{X}$ be defined as in the proof of \theoremref{thm:explicit-formula}, and let
  \begin{equation}
    \tilde{F} \colon \sections(U) \oplus \sections(V) \to \forms^4(X, E) \cong \sections(E) \, \text{,} \quad (s, \chi) \mapsto \pi_E((\exp_s)^\ast(\tau_{\Theta(\chi)})) \, \text{,}
  \end{equation}
  where $\tau_{\Theta(\chi)} \in \forms^4(M, \altforms^2 M)$ is defined as in~\eqref{eq:def-tau} with respect to the $\Spin(7)$-structure $\Theta(\chi)$. So the moduli space of all Cayley submanifolds of $(M, \Theta(\chi))$ near~$X$ can be identified with $\tilde{F}({}\cdot{}, \chi)^{-1}(\set{0})$.
  
  \begin{lemma} \label{lemma:linearisation-F-tilde}
    Let $e \in \sections(\altforms^2_7 M)$, and let $\chi \defeq h(\tau_\Phi, e)$, where $h$ is the metric on~$\altforms^2_7 M$ (note that $\chi \in \forms^4_7(M)$ by \eqref{eq:inner-tau} and \eqref{eq:forms-4-7}). Then
    \begin{equation}
      (\D \tilde{F})_{(0, 0)}(0, \chi) = - \pi_E(e \vert_X) \, \text{.}
    \end{equation}
  \end{lemma}
  
  \begin{proof}
    Since $\chi \in \forms^4_7(M)$, there is a path $(\Phi_t)_{t \in (- \eps, \eps)}$ of $\Spin(7)$\dash-structures on~$M$ with $\Phi_0 = \Phi$ and $\frac{\D}{\D t} \Phi_t \big\vert_{t = 0} = \chi$ such that the metric induced by~$\Phi_t$ is the same metric as the metric induced by~$\Phi$ for $t \in (- \eps, \eps)$ \cite[Proposition~5.3.1]{Kar05}. Write $\Phi_t = \Theta(\chi_t)$ with $\frac{\D}{\D t} \chi_t \big\vert_{t = 0} = \chi$.
    
    Let $(e_1, \dotsc, e_8)$ be a local orthonormal frame of~$M$ such that $(e_1, \dotsc, e_4)$ is a positive frame of~$X$, and let $(e^1, \dotsc, e^8)$ be the dual coframe. Then
    \begin{align*}
      (\D \tilde{F})_{(0, 0)}(0, \chi) &= \frac{\D}{\D t} \tilde{F}(0, \chi_t) \bigg\vert_{t = 0} = \frac{\D}{\D t} \pi_E\bigl(\tau_{\Theta(\chi_t)} \big\vert_X\bigr) \bigg\vert_{t = 0} \\
      &= \frac{\D}{\D t} \sum_{i = 2}^8 (e^i \wedge (e_1 \interior \Phi_t) - e^1 \wedge (e_i \interior \Phi_t)) \vert_X \otimes \pi_E(e_1 \times e_i) \bigg\vert_{t = 0} \\
      &= \sum_{i = 2}^8 (e^i \wedge (e_1 \interior \chi) - e^1 \wedge (e_i \interior \chi)) \vert_X \otimes \pi_E(e_1 \times e_i) \\
      &= - \sum_{i = 5}^8 \chi(e_i, e_2, e_3, e_4) \vol_X \otimes (e_1 \times e_i) \\
      &= - \sum_{i = 5}^8 h(\tau_\Phi(e_i, e_2, e_3, e_4), e) \vol_X \otimes (e_1 \times e_i) \\
      &= \sum_{i = 5}^8 h(e_i \times (e_2 \times e_3 \times e_4), e) \vol_X \otimes (e_1 \times e_i) \\
      &= \sum_{i = 5}^8 h(e, e_i \times e_1) \vol_X \otimes (e_1 \times e_i) \\
      &= - \vol_X \otimes \pi_E(e \vert_X)
    \end{align*}
    by \eqref{eq:def-tau}.
  \end{proof}
  
  By \propositionref{prop:c-ell}, $\tilde{F}$ extends to a map
  \begin{equation}
    \tilde{F}_{1, \alpha} \colon \CC^{1, \alpha}(U) \oplus \CC^{1, \alpha}(V) \to \CC^{0, \alpha}(E)
  \end{equation}
  of class~$\CC^1$. For fixed $\chi \in \CC^{1, \alpha}(V)$, the equation
  \begin{equation}
    \tilde{F}_{1, \alpha}(s, \chi) = \pi_E((\exp_s)^\ast(\tau_{\Theta(\chi)})) = 0 \label{eq:F-tilde-chi}
  \end{equation}
  is a nonlinear partial differential equation of order~$1$ in~$s$. Furthermore, the linearisation at~$0$ is elliptic for $\chi = 0$. Hence there is a $\CC^{1, \alpha}$\dash-neighbourhood $\tilde{U}_1 \subseteq \CC^{1, \alpha}(U) \oplus \CC^{1, \alpha}(V)$ of $(0, 0)$ such that
  \begin{equation}
    (\D \tilde{F}_{1, \alpha})_{(s, \chi)} \vert_{\CC^{1, \alpha}(\normal[M]{X})} \colon \CC^{1, \alpha}(\normal[M]{X}) \to \CC^{0, \alpha}(E) \label{eq:F-tilde-chi-linearisation}
  \end{equation}
  is an elliptic differential operator of order~$1$ for all $(s, \chi) \in \tilde{U}_1$. In particular, if $(s, \chi) \in \tilde{U}_1$ satisfies~\eqref{eq:F-tilde-chi} and $\chi \in \CC^{k, \alpha}(V)$ (for some $k \ge 1$), then $s \in \CC^{k + 1, \alpha}(\normal[M]{X})$ by Elliptic Regularity \cite[Theorem~6.8.1]{Mor66} since $\tau_{\Theta(\chi)} \in \CC^{k, \alpha}(\altforms^4 M \otimes \altforms^2 M)$. Furthermore, the coefficients of~\eqref{eq:F-tilde-chi-linearisation} are in~$\CC^{k - 1, \alpha}$. The first-order coefficients are even in~$\CC^{k, \alpha}$ since $\tau_{\Theta(\chi)}$ is multilinear (and hence smooth) in each fibre and $\D (\exp_s) \in \CC^{k, \alpha}(T^\ast X \otimes (\exp_s)^\ast T M)$. So the formal adjoint of~\eqref{eq:F-tilde-chi-linearisation} has coefficients in~$\CC^{k - 1, \alpha}$, and hence the $L^2$\dash-orthogonal complement of the image of~\eqref{eq:F-tilde-chi-linearisation} is a (finite-dimensional) subspace of~$\CC^{k, \alpha}(E)$.
  
  Note that the map $\chi \mapsto \tau_{\Theta(\chi)}$ is a bundle morphism which is smooth in each fibre. Furthermore, the calculations in \lemmaref{lemma:linearisation-F-tilde} can be done fibre-wise. So there is a $\CC^{1, \alpha}$\dash-neighbourhood $\tilde{U}_2 \subseteq \CC^{1, \alpha}(U) \oplus \CC^{1, \alpha}(V)$ of $(0, 0)$ such that
  \begin{equation*}
    (\D \tilde{F}_{1, \alpha})_{(s, \chi)} \vert_{\CC^{k, \alpha}(\altforms^4_1 M \oplus \altforms^4_7 M \oplus \altforms^4_{35} M)} \colon \CC^{k, \alpha}(\altforms^4_1 M \oplus \altforms^4_7 M \oplus \altforms^4_{35} M) \to \CC^{k, \alpha}(E)
  \end{equation*}
  is surjective for all $(s, \chi) \in \tilde{U}_2$ and $k \ge 1$.
  
  For $k \ge 1$, let
  \begin{equation*}
    \tilde{F}_{k, \alpha} \defeq \tilde{F}_{1, \alpha} \vert_{\CC^{1, \alpha}(U) \oplus \CC^{k, \alpha}(V)} \colon \CC^{1, \alpha}(U) \oplus \CC^{k, \alpha}(V) \to \CC^{0, \alpha}(E) \, \text{.}
  \end{equation*}
  Then $\tilde{F}_{k, \alpha}$ if of class~$\CC^k$ by \propositionref{prop:c-ell}. Let $\tilde{U} \defeq \tilde{U}_1 \cap \tilde{U}_2$. Then the above argumentation shows that if $(s, \chi) \in \tilde{U} \cap (\CC^{1, \alpha}(U) \oplus \CC^{k, \alpha}(V))$ satisfies~\eqref{eq:F-tilde-chi}, then
  \begin{equation*}
    (\D \tilde{F}_{k, \alpha})_{(s, \chi)} \colon \CC^{1, \alpha}(\normal[M]{X}) \oplus \CC^{k, \alpha}(\altforms^4_1 M \oplus \altforms^4_7 M \oplus \altforms^4_{35} M) \to \CC^{0, \alpha}(E)
  \end{equation*}
  is surjective and
  \begin{equation*}
    (\D \tilde{F}_{k, \alpha})_{(s, \chi)} \vert_{\CC^{1, \alpha}(\normal[M]{X})} = (\D \tilde{F}_{1, \alpha})_{(s, \chi)} \vert_{\CC^{1, \alpha}(\normal[M]{X})} \colon \CC^{1, \alpha}(\normal[M]{X}) \to \CC^{0, \alpha}(E)
  \end{equation*}
  is Fredholm. Furthermore, the Fredholm index is the same for all of these operators as we may assume \wolog\ that $\tilde{U}$ is connected. So the Fredholm index is~$\ind D$ since $D = (\D \tilde{F})_{(0, 0)} \vert_{\sections(\normal[M]{X})}$.
  
  Let $\tilde{U}_3 \subseteq \CC^{1, \alpha}(U)$ and $\tilde{U}_4 \subseteq \CC^{1, \alpha}(V)$ be open neighbourhoods of~$0$ such that $\tilde{U}_3 \times \tilde{U}_4 \subseteq \tilde{U}$, let $U_{1, \alpha}$ be the set of all $\chi \in \tilde{U}_4$ such that the operator \eqref{eq:F-tilde-chi-linearisation} is surjective for all $s \in \tilde{U}_3$, and define $U_{k, \alpha} \defeq U_{1, \alpha} \cap \CC^{k, \alpha}(\altforms^4_1 M \oplus \altforms^4_7 M \oplus \altforms^4_{35} M)$. Let $k_0 \defeq \max \set{\ind D + 1, 1}$. Then \theoremref{thm:genericity-general} below implies that $U_{k, \alpha}$ is the intersection of countably many open dense subsets of $\CC^{k, \alpha}(\altforms^4_1 M \oplus \altforms^4_7 M \oplus \altforms^4_{35} M)$ for $k \ge k_0$. So if we define $U_\infty \defeq \bigcap_{k = k_0}^\infty U_{k, \alpha}$, then $U_\infty$ is the intersection of countably many open dense subsets of $\CC^\infty(\altforms^4_1 M \oplus \altforms^4_7 M \oplus \altforms^4_{35} M)$ by \lemmaref{lemma:residual-c-infty} below.
  
  So for every generic smooth $\Spin(7)$\dash-structure~$\Psi$ that is $\CC^{1, \alpha}$\dash-close to~$\Phi$, the moduli space of all smooth Cayley submanifolds of $(M, \Psi)$ that are $\CC^{1, \alpha}$\dash-close to~$X$ is either empty or a smooth manifold of dimension $\ind D$.
\end{proof}

Note that the proof also shows the following $\CC^{k, \alpha}$\dash-version, where we use the $\CC^{k, \alpha}$\dash-topology for the space of all $\Spin(7)$\dash-structures.

\begin{theorem} \label{thm:main-theorem-spin7-closed-c-k-alpha}
  Let $k \ge 1$, let $0 < \alpha < 1$, let $M$ be an $8$\dash-manifold of class~$\CC^{k + 1, \alpha}$ with a $\Spin(7)$-structure~$\Phi$ of class~$\CC^{k, \alpha}$, and let $X$ be a closed Cayley submanifold of~$M$ of class~$\CC^{k + 1, \alpha}$. Suppose that $k > \ind D$, where $D$ is defined in~\eqref{eq:def-D}.
  
  Then for every generic $\Spin(7)$\dash-structure~$\Psi$ of class~$\CC^{k, \alpha}$ that is $\CC^{1, \alpha}$\dash-close to~$\Phi$, the moduli space of all Cayley submanifolds of $(M, \Psi)$ of class~$\CC^{k + 1, \alpha}$ that are $\CC^{1, \alpha}$\dash-close to~$X$ is either empty or a $\CC^k$\dash-manifold of dimension $\ind D$.
\end{theorem}

The following theorem and the lemma afterwards were used in the proof of \theoremref{thm:main-theorem-spin7-closed}.

\begin{theorem}[{cf.~\cite[Proposition~2.24]{Schw93}}] \label{thm:genericity-general}
  Let $X$, $Y$, and $Z$ be separable Banach spaces, let $U \subseteq X$ and $V \subseteq Y$ be open subsets, let $f \colon U \times V \to Z$ be a $\CC^k$\dash-map ($1 \le k \le \infty$), let $(x_0, y_0) \in U \times V$, and let $\ell \in \Z$ be such that $\ell < k$. Suppose that for all $(x, y) \in U \times V$ with $f(x, y) = f(x_0, y_0)$:
  \begin{compactenum}[(i)]
    \item $(\D f)_{(x, y)} \colon T_x X \times T_y Y \to T_{f(x, y)} Z$ is surjective and
    \item $(\D f)_{(x, y)} \vert_{T_x X} \colon T_x X \to T_{f(x, y)} Z$ is Fredholm with index~$\ell$.
  \end{compactenum}
  Then the set of all $y \in V$ such that the operator $(\D f)_{(x, y)} \vert_{T_x X} \colon T_x X \to T_{f(x, y)} Z$ is surjective for all $x \in U$ with $f(x, y) = f(x_0, y_0)$ is the intersection of countably many open dense subsets.
\end{theorem}

\begin{lemma} \label{lemma:residual-c-infty}
  Let $M$ be a smooth manifold, let $E$ be a smooth vector bundle over~$M$, let $k_0 \ge 1$, let $0 \le \alpha \le 1$, let $\mathcal{M}_{k_0} \subseteq \CC^{k_0, \alpha}(E)$, let $\mathcal{M}_k \defeq \mathcal{M}_{k_0} \cap \CC^{k, \alpha}(E)$ for $k \ge k_0 + 1$, and let $\mathcal{M}_\infty \defeq \mathcal{M}_{k_0} \cap \CC^\infty(E)$. Suppose that $\mathcal{M}_k$ is the intersection of countably many open dense subsets of~$\CC^{k, \alpha}(E)$ for all $k \ge k_0$. Then $\mathcal{M}_\infty$ is the intersection of countably many open dense subsets of~$\CC^\infty(E)$.
\end{lemma}

\begin{proof}
  For each $k \ge k_0$, there are countable many open dense subsets $(U_{k, n})_{n \ge 1}$ of $\CC^{k, \alpha}(E)$ such that
  \begin{equation*}
    \mathcal{M}_k = \bigcap_{n = 1}^\infty U_{k, n} \, \text{.}
  \end{equation*}
  Let $k \ge k_0$ and $n \ge 1$. Then $U_{k, n} \cap \CC^{k^\prime, \alpha}(E)$ is open in $\CC^{k^\prime, \alpha}(E)$ for all $k^\prime \ge k$. Furthermore, $\mathcal{M}_{k^\prime} \subseteq U_{k, n} \cap \CC^{k^\prime, \alpha}(E)$ for all $k^\prime \ge k$, and hence $U_{k, n} \cap \CC^{k^\prime, \alpha}(E)$ is dense in $\CC^{k^\prime, \alpha}(E)$ for all $k^\prime \ge k$ since $\mathcal{M}_{k^\prime}$ is dense in $\CC^{k^\prime, \alpha}(E)$ by the Baire Category Theorem. So $U_{k, n} \cap \CC^\infty(E)$ is open and dense in $\CC^\infty(E)$.
  
  Therefore, $\mathcal{M}_\infty$ is the intersection of countably many open dense subsets of $\CC^\infty(E)$ since
  \begin{equation*}
    \mathcal{M}_\infty = \bigcap_{k = 1}^\infty \mathcal{M}_k = \bigcap_{k = 1}^\infty \bigcap_{n = 1}^\infty U_{k, n} \, \text{.} \qedhere
  \end{equation*}
\end{proof}

\subsection{Remark about Torsion-Free Spin(7)-Structures}
\label{subsec:remark-torsion-free-closed}

If $M$ is a closed $8$\dash-manifold with a torsion-free $\Spin(7)$\dash-structure~$\Phi$, then the tangent space to the space of all torsion-free $\Spin(7)$\dash-structures on~$M$ at~$\Phi$ can be identified with \cite[Theorem~10.7.1]{Joy00}
\begin{equation}
  \set{\Lie_v \Phi \colon v \in \sections(T M)} \oplus (\mathcal{H}^4_1(M) \oplus \mathcal{H}^4_7(M) \oplus \mathcal{H}^4_{35}(M)) \, \text{,}
\end{equation}
where $\mathcal{H}^4_i(M)$ is the space of all harmonic forms in $\forms^4_i(M)$. If $\chi = \Lie_v \Phi$, then $(\D \tilde{F})_{(0, 0)}(0, \chi)$ lies in the image of~$D$ by \lemmaref{lemma:linearisation-F-tilde-Lie} below. So it does not contribute to the surjectivity of $(\D \tilde{F})_{(0, 0)}$. But the dimension of $\coker D$ is not in general less than $b^4_1 + b^4_7 + b^4_{35}$ (see below). That is why we consider the larger space of all $\Spin(7)$\dash-structures on~$M$.

\begin{lemma} \label{lemma:linearisation-F-tilde-Lie}
  Let $M$ be an $8$\dash-manifold with a $\Spin(7)$\dash-structure~$\Phi$, let $X$ be a closed Cayley submanifold, and let $v \in \sections(T M)$. Then
  \begin{equation}
    (\D \tilde{F})_{(0, 0)}(0, \Lie_v \Phi) = D (v \vert_X)^\perp \, \text{,} \label{eq:linearisation-F-tilde-Lie}
  \end{equation}
  where $(v \vert_X)^\perp$ is the pointwise orthogonal projection of~$v \vert_X$ onto~$\normal[M]{X}$.
\end{lemma}

\begin{proof}
  Let $\varphi_v^t$ denote the flow generated by~$v$, and let $\Phi_t \defeq (\varphi_v^t)^\ast \Phi$. Then $\tau_{\Phi_t} = (\varphi_v^t)^\ast \tau_\Phi$. Hence
  \begin{equation*}
    \frac{\D}{\D t} \tau_{\Phi_t} \bigg\vert_{t = 0} = \frac{\D}{\D t} (\varphi_v^t)^\ast \tau_\Phi \bigg\vert_{t = 0} = \Lie_v \tau_\Phi \, \text{.}
  \end{equation*}
  So the result follows from the proof of \theoremref{thm:explicit-formula} since $\tau_\Phi \vert_X = 0$.
\end{proof}

In \cite[last paragraph of Section~4.4]{Gay14}, Gayet constructs a $7$\dash-manifold~$\tilde{M}$ with a torsion-free $G_2$\dash-structure such that given $n \in \N$, there is a closed associative submanifold~$\tilde{X}$ for which the dimension of the cokernel of the associated Dirac operator has dimension at least~$n$. So if we define $M \defeq S^1 \times \tilde{M}$ and $X \defeq S^1 \times \tilde{X}$, then the dimension of $\coker D$ will be at least~$n$.

\section{Compact Cayley Submanifolds with Boundary}
\label{sec:cayley-boundary}

In this section we present the deformation theory of compact, connected Cayley submanifolds with non-empty boundary. We first show that the moduli space embeds into the solution space of a second-order elliptic boundary problem with index~$0$ in \sectionref{subsec:cayley-boundary-deformations}. Then in \sectionref{subsec:varying-spin7-boundary} we prove that for a generic $\Spin(7)$\dash-structure, the moduli space is a finite set. In \sectionref{subsec:varying-scaffold} we show that also for a generic deformation of the scaffold (the submanifold containing the boundary of the Cayley submanifold), the moduli space is a finite set. We further show that \theoremref{thm:main-spin7-boundary} about the moduli space for a generic $\Spin(7)$\dash-structure remains true if we restrict to the smaller class of all torsion-free $\Spin(7)$-structures in \sectionref{subsec:remark-torsion-free-boundary}. We finish this section with a corollary in \sectionref{subsec:existence-cayley-nearby} about the existence of Cayley submanifolds with boundary for nearby $\Spin(7)$-structures and scaffolds if we start in a generic situation.

\subsection{Deformations of Compact Cayley Submanifolds with Boundary}
\label{subsec:cayley-boundary-deformations}

The next proposition is the first step for the proofs of \hyperref[thm:main-spin7-boundary]{Theorems~\ref*{thm:main-spin7-boundary}} and \ref{thm:main-scaffold}.

\begin{definition}
  Let $M$ be an $8$\dash-manifold with a $\Spin(7)$-structure~$\Phi$, let $X$ be a Cayley submanifold of~$M$ with non-empty boundary, let $W$ be a submanifold of~$M$ with $\partial X \subseteq W$, and let $u \in \sections(\normal[X]{\partial X})$ be a unit normal vector field of~$\partial X$ in~$X$. Then $X$ and $W$ \emph{meet orthogonally} if $u \in \sections(\normal[M]{W} \vert_{\partial X})$.
\end{definition}

\begin{proposition}[\propositionref{prop:main-boundary}] \label{prop:main-proposition-boundary}
  Let $M$ be an $8$\dash-manifold with a $\Spin(7)$-structure~$\Phi$, let $X$ be a compact, connected Cayley submanifold of~$M$ with non-empty boundary, and let $W$ be a submanifold of~$M$ with $\partial X \subseteq W$ such that $X$ and $W$ meet orthogonally.
  
  Then the moduli space of all local deformations of~$X$ as a Cayley submanifold of~$M$ with boundary on~$W$ and meeting~$W$ orthogonally can be embedded into the solution space of the boundary problem~\eqref{eq:boundary-second} below, which is a second-order elliptic boundary problem with index~$0$.
\end{proposition}

In particular, the ``Zariski tangent space'' (the kernel of the linearisation of the deformation map) is finite-dimensional and Elliptic Regularity applies (i.e., all solutions are smooth). Note that the dimension of~$W$ is at least~$3$ since $\partial X \subseteq W$ and $\partial X$ is $3$\dash-dimensional. Furthermore, the dimension of~$W$ is at most~$7$ since $M$ is $8$\dash-dimensional and $X$ and $W$ meet orthogonally. In fact, if $W$ were $8$\dash-dimensional, then there would be essentially no constraint on the boundary, which would lead to an infinite-dimensional moduli space. For a discussion of the dimensions~$3$ and~$7$, see \sectionref{subsec:remarks-dimension-scaffold}.

\begin{proof}
  In the \hyperref[subsec:metric]{next section} we will modify the metric~$g$ to a metric~$\hat{g}$ with $\hat{g}_x = g_x$ for all $x \in \partial X$ such that $W$ is totally geodesic with respect to~$\hat{g}$. Then there is an open tubular neighbourhood $\hat{U} \subseteq \normalhat[M]{X}$ of the $0$\dash-section such that local deformations of~$X$ with boundary on~$W$ are parametrised by sections $\hat{s} \in \sections(\hat{U})$ with $\hat{s} \vert_{\partial X} \in \sections(T W \vert_{\partial X})$ (\propositionref{prop:adapted-tubular-neighbourhood}), where $\normalhat[M]{X}$ is the normal bundle with respect to the metric~$\hat{g}$. Note that $\normalhat[M]{X} \cong (T M \vert_X) / T X \cong \normal[M]{X}$. Let $U \subseteq \normal[M]{X}$ be the image of~$\hat{U}$ under this isomorphism.
  
  For a normal vector field $s \in \sections(U)$, let $\hat{s} \in \sections(\hat{U})$ be the corresponding vector field under the isomorphism $\normal[M]{X} \cong \normalhat[M]{X}$, and define $\widehat{\exp}_s \colon X \to M$, $x \mapsto \widehat{\exp}_x(\hat{s}(x))$, where $\widehat{\exp} \colon \hat{U} \to M$ is the exponential map of the metric~$\hat{g}$. Let
  \begin{equation}
    \pi_E \colon \altforms^2_7 M \vert_X \to E \label{eq:def-pi-E}
  \end{equation}
  denote the orthogonal projection, where the vector bundle~$E$ of rank~$4$ over~$X$ is defined as in~\eqref{eq:def-E}, and let
  \begin{equation}
    F \colon \sections(U) \to \forms^4(X, E) \cong \sections(E) \, \text{,} \quad s \mapsto \pi_E((\widehat{\exp}_s)^\ast(\tau)) \, \text{,} \label{eq:def-F}
  \end{equation}
  where $\tau \in \forms^4(M, \altforms^2_7 M)$ is defined as in~\eqref{eq:def-tau}. Then
  \begin{equation}
    (\D F)_0(s) = D s \, \text{,} \label{eq:linearisation-F}
  \end{equation}
  where $D \colon \sections(\normal[M]{X}) \to \sections(E)$ is defined in~\eqref{eq:def-D}. This follows from the proof of \theoremref{thm:explicit-formula} using that $\tau \vert_X = 0$ (so the result does not depend on the choice of metric for the exponential map).
  
  Let $K$ be the subbundle of~$\normal[M]{X} \vert_{\partial X}$ consisting of all vectors that are orthogonal to~$T W \vert_{\partial X}$, and let
  \begin{equation}
    \pi_K \colon \normal[M]{X} \vert_{\partial X} \to K \label{eq:def-pi-K}
  \end{equation}
  denote the orthogonal projection. So if $s \in \sections(\normal[M]{X})$, then $s \vert_{\partial X} \in \sections(T W \vert_{\partial X})$ if and only if $\pi_K(s \vert_{\partial X}) = 0$.
  
  Let
  \begin{equation}
    \pi_W \colon T M \vert_W \to T W \label{eq:def-pi-W}
  \end{equation}
  and
  \begin{equation}
    \pi_\nu \colon T M \vert_{\partial X} \to \normal[W]{\partial X} \label{eq:def-pi-nu}
  \end{equation}
  denote the orthogonal projections, let $U_{\partial X} \defeq U \cap \normal[M]{X} \vert_{\partial X}$, and let
  \begin{equation}
    H \colon \sections(U_{\partial X}) \to \forms^3(\partial X, \normal[W]{\partial X}) \cong \sections(\normal[W]{\partial X}) \, \text{,} \quad s \mapsto \pi_\nu((\widehat{\exp}_{\pi_\nu(s)})^\ast(\gamma)) \, \text{,} \label{eq:def-H}
  \end{equation}
  where $\gamma \in \forms^3(W, T W)$ is defined by
  \begin{equation}
    \gamma(a, b, c) \defeq \pi_W(a \times b \times c) \label{eq:def-gamma}
  \end{equation}
  for $a, b, c \in T W$. So if $s \in \sections(U)$ defines a Cayley submanifold with boundary on~$W$ (i.e.,~$F(s) = 0$ and $\pi_K(s \vert_{\partial X}) = 0$), then this Cayley submanifold meets $W$ orthogonally if and only if $H(s \vert_{\partial X}) = 0$. This follows because if $(a, b, c)$ is a local orthonormal frame of~$\partial X$, then $a \times b \times c$ is a unit normal vector field of~$\partial X$ in~$X$ since $X$ is Cayley. The projection onto~$\normal[W]{\partial X}$ is enough since the cross product of three vectors is orthogonal to these vectors.
  
  So the moduli space of all local deformations of~$X$ as a Cayley submanifold of~$M$ with boundary on~$W$ and meeting~$W$ orthogonally can be identified with the moduli space of all solutions near the $0$\dash-section of the boundary problem
  \begin{equation}
    \begin{cases}
      \phantom{\pi_K(s \vert_{\partial X}) = 0} \mathllap{F(s) = 0} &\text{in~$X$,} \\
      \pi_K(s \vert_{\partial X}) = 0 &\text{on~$\partial X$,} \\
      \phantom{\pi_K(s \vert_{\partial X}) = 0} \mathllap{H(s \vert_{\partial X}) = 0} &\text{on~$\partial X$.}
    \end{cases} \label{eq:boundary-first}
  \end{equation}
  
  Let $D^\ast \colon \sections(E) \to \sections(\normal[M]{X})$ be the formal adjoint of~$D$, and let
  \begin{equation}
    G \colon \sections(U) \to \sections(\normal[M]{X}) \, \text{,} \quad s \mapsto D^\ast(F(s)) \, \text{.} \label{eq:def-G}
  \end{equation}
  Then
  \begin{equation}
    (\D G)_0(s) = D^\ast D s \label{eq:linearisation-G}
  \end{equation}
  by \eqref{eq:linearisation-F} since $D^\ast$ is linear. Let $u \in \sections(\normal[X]{\partial X})$ be the inward-pointing unit normal vector field of~$\partial X$ in~$X$, and let
  \begin{equation}
    \rho \colon \normal[M]{X} \vert_{\partial X} \to E \vert_{\partial X} \, \text{,} \quad s \mapsto u \times s \, \text{.} \label{eq:def-rho}
  \end{equation}
  Then $\rho$ is an isomorphism of vector bundles. Let
  \begin{equation}
    B \colon \sections(U) \to \sections(\normal[W]{\partial X}) \, \text{,} \quad s \mapsto \pi_\nu(\rho^{-1}(F(s) \vert_{\partial X})) + H(s \vert_{\partial X}) \, \text{.} \label{eq:def-B}
  \end{equation}
  So solutions of \eqref{eq:boundary-first} are also solutions of
  \begin{equation}
    \begin{cases}
      \phantom{\pi_K(s \vert_{\partial X}) = 0} \mathllap{G(s) = 0} &\text{in~$X$,} \\
      \pi_K(s \vert_{\partial X}) = 0 &\text{on~$\partial X$,} \\
      \phantom{\pi_K(s \vert_{\partial X}) = 0} \mathllap{B(s) = 0} &\text{on~$\partial X$.}
    \end{cases} \label{eq:boundary-second}
  \end{equation}
  We will now prove that the boundary problem \eqref{eq:boundary-second} is an elliptic boundary problem with index~$0$.
  
  Let $P \colon \sections(\normal[M]{X} \vert_{\partial X}) \to \sections(\normal[M]{X} \vert_{\partial X})$,
  \begin{equation}
    P s \defeq \sum_{i = 2}^4 u \times e_i \times \nabla_{e_i}^\perp s - \sum_{i = 5}^8 (\nabla_s \Phi)(e_i, e_2, e_3, e_4) e_i \, \text{,} \label{eq:def-P}
  \end{equation}
  where $(e_i)_{i = 2, 3, 4}$ is any positive (i.e., $u = e_2 \times e_3 \times e_4$) local orthonormal frame of~$\partial X$, $(e_i)_{i = 5, \dotsc, 8}$ is any local orthonormal frame of~$\normal[M]{X} \vert_{\partial X}$, and $\nabla^\perp$ is the induced connection on~$\normal[M]{X}$. So if $s \in \sections(\normal[M]{X})$, then
  \begin{equation}
    \rho^{-1}((D s) \vert_{\partial X}) = \nabla_u s \vert_{\partial X} + P(s \vert_{\partial X}) \label{eq:relation-D-P}
  \end{equation}
  by \eqref{eq:def-D} since
  \begin{align*}
    g(\rho^{-1}(e_i \times \nabla_{e_i}^\perp s), e_j) &= h(e_i \times \nabla_{e_i}^\perp s, u \times e_j) = - \Phi(e_i, \nabla_{e_i}^\perp s, u, e_j) \\
    &= g(u \times e_i \times \nabla_{e_i}^\perp s, e_j)
  \end{align*}
  for $i = 2, 3, 4$, $j = 5, \dotsc, 8$ by \eqref{eq:inner-cross-2}, where $h$ is the metric on~$\altforms^2_7 M$.
  
  \begin{lemma} \label{lemma:linearisation-B}
    Let $s \in \sections(\normal[M]{X})$. Then
    \begin{equation}
      (\D B)_0(s) = \pi_\nu(\nabla_u s \vert_{\partial X} - \nabla_{\pi_\nu(s \vert_{\partial X})} u + P(\pi_K(s \vert_{\partial X}))) \, \text{.} \label{eq:linearisation-B}
    \end{equation}
  \end{lemma}
  
  \begin{proof}
    First let $s \in \sections(\normal[W]{\partial X})$. Then
    \begin{equation}
      (\D H)_0(s) = \frac{\D}{\D t} H(t s) \bigg\vert_{t = 0} = \frac{\D}{\D t} \pi_\nu((\widehat{\exp}_{t s})^\ast(\gamma)) \bigg\vert_{t = 0} = \pi_\nu((\Lie_s \gamma) \vert_{\partial X}) \, \text{.}
    \end{equation}
    Let $k$ be the dimension of~$W$, and let $(e_1, \dotsc, e_8)$ be a local orthonormal frame of~$M$ such that $(e_1, \dotsc, e_4)$ is a positive (i.e., $e_1 = e_2 \times e_3 \times e_4$) frame of~$X$ with $e_1 = u$ and $(e_2, \dotsc, e_{k + 1})$ is a frame of~$W$. Then
    \begin{equation}
      \gamma = \sum_{i = 2}^{k + 1} (e_i \interior \Phi) \vert_W \otimes e_i \, \text{.}
    \end{equation}
    So
    \begin{align*}
      (\Lie_s \gamma) \vert_{\partial X} &= \sum_{i = 2}^{k + 1} (\Lie_s (e_i \interior \Phi)) \vert_{\partial X} \otimes e_i + \sum_{i = 2}^{k + 1} (e_i \interior \Phi) \vert_{\partial X} \otimes \pi_W(\nabla_s e_i) \\
      &= \sum_{i = 2}^{k + 1} (\Lie_s (e_i \interior \Phi)) \vert_{\partial X} \otimes e_i
    \end{align*}
    since $(e_i \interior \Phi) \vert_{\partial X} = 0$ as $\Phi(e_i, e_2, e_3, e_4) = g(e_i, e_2 \times e_3 \times e_4) = g(e_i, e_1) = 0$ for $i = 2, \dotsc, 8$. We have $\Lie_s (e_i \interior \Phi) = \Lie_s e_i \interior \Phi + e_i \interior \Lie_s \Phi$. So
    \begin{align*}
      (\Lie_s(e_i \interior \Phi))(e_2, e_3, e_4) &= \Phi(\Lie_s e_i, e_2, e_3, e_4) + (\Lie_s \Phi)(e_i, e_2, e_3, e_4) \\
      &= \Phi(\Lie_s e_i, e_2, e_3, e_4) + (\nabla_s \Phi)(e_i, e_2, e_3, e_4) \\
      &\phantom{{}={}} {}+ \Phi(\nabla_{e_i} s, e_2, e_3, e_4) + \Phi(e_i, \nabla_{e_2} s, e_3, e_4) \\
      &\phantom{{}={}} {}+ \Phi(e_i, e_2, \nabla_{e_3} s, e_4) + \Phi(e_i, e_2, e_3, \nabla_{e_4} s) \, \text{.}
    \end{align*}
    Now
    \begin{align*}
      \Phi(\Lie_s e_i, e_2, e_3, e_4) + \Phi(\nabla_{e_i} s, e_2, e_3, e_4) &= \Phi(\nabla_s e_i, e_2, e_3, e_4) \\
      &= g(\nabla_s e_i, e_2 \times e_3 \times e_4) \\
      &= g(\nabla_s e_i, e_1) = - g(\nabla_s u, e_i)
    \end{align*}
    since $u = e_1 = e_2 \times e_3 \times e_4$.
    
    Recall from the proof of \theoremref{thm:explicit-formula} that
    \begin{align*}
      &\sum_{i = 5}^8 (\Phi(e_i, \nabla_{e_2} s, e_3, e_4) + \Phi(e_i, e_2, \nabla_{e_3} s, e_4) + \Phi(e_i, e_2, e_3, \nabla_{e_4} s)) (e_1 \times e_i) \\
      &= - \sum_{i = 2}^4 \sum_{j = 5}^8 g(\nabla_{e_i} s, e_j) (e_i \times e_j) = - \sum_{i = 2}^4 e_i \times \nabla_{e_i}^\perp s \, \text{.}
    \end{align*}
    So
    \begin{align*}
      &\pi_\nu((\Lie_s \gamma)(e_2, e_3, e_4)) \\
      &= - \pi_\nu(\nabla_s u) - \sum_{i = 2}^4 \pi_\nu(\rho^{-1}(e_i \times \nabla_{e_i}^\perp s)) + \sum_{i = 5}^{k + 1} (\nabla_s \Phi)(e_i, e_2, e_3, e_4) e_i \\
      &= - \pi_\nu(\nabla_s u + P s) \, \text{.}
    \end{align*}
    For general $s \in \sections(\normal[M]{X})$, we have $s \vert_{\partial X} = \pi_\nu(s \vert_{\partial X}) + \pi_K(s \vert_{\partial X})$. Hence
    \begin{align*}
      (\D B)_0(s) &= \pi_\nu(\rho^{-1}((D s) \vert_{\partial X})) + \pi_\nu((\Lie_{\pi_\nu(s \vert_{\partial X})} \gamma)(e_2, e_3, e_4)) \\
      &= \pi_\nu(\nabla_u s \vert_{\partial X} + P(s \vert_{\partial X}) - \nabla_{\pi_\nu(s \vert_{\partial X})} u - P(\pi_\nu(s \vert_{\partial X}))) \\
      &= \pi_\nu(\nabla_u s \vert_{\partial X} - \nabla_{\pi_\nu(s \vert_{\partial X})} u + P(\pi_K(s \vert_{\partial X})))
    \end{align*}
    by \eqref{eq:linearisation-F}, \eqref{eq:def-B}, and \eqref{eq:relation-D-P}.
  \end{proof}
  
  \begin{lemma} \label{lemma:boundary-second-elliptic}
    The linearisation of the boundary problem~\eqref{eq:boundary-second} at the $0$\dash-section is elliptic.
  \end{lemma}
  
  \begin{proof}
    Let $k$ be the dimension of~$W$, and let $(e_1, \dotsc, e_8)$ be a local orthonormal frame of~$M$ such that $(e_1, \dotsc, e_4)$ is a positive (i.e., $e_1 = e_2 \times e_3 \times e_4$) frame of~$X$ with $e_1 = u$ and $(e_2, \dotsc, e_{k + 1})$ is a frame of~$W$. Let $s \in \sections(\normal[M]{X})$, and write $s = s_5 e_5 + \dotsc + s_8 e_8$. Then $\pi_K(s \vert_{\partial X}) = \sum_{i = k + 2}^8 s_j e_j$. Hence
    \begin{equation}
      (\D B)_0(s) = \sum_{i = 5}^{k + 1} (e_1 . s_i) e_i + \sum_{i = 2}^4 \sum_{j = k + 2}^8 (e_i . s_j) \pi_\nu(e_1 \times e_i \times e_j) + \text{l.o.t.}
    \end{equation}
    by \eqref{eq:def-P} and \eqref{eq:linearisation-B}, where ``l.o.t.'' stands for lower order terms (i.e., an operator of order~$0$ in~$s$).
    
    So the symbol $\sigma_\partial(x, \xi)$ (for $x \in \partial X$, $\xi \in T^\ast_x X$) of the boundary operator $B \oplus \pi_K$ is given by
    \begin{equation}
      \left(
      \begin{array}{@{}ccc|ccc@{}}
        \I \xi_1 & 0      & 0        &   &        &                     \\[-0.75ex]
        0        & \ddots & 0        &   & A      &                     \\[-0.75ex]
        0        & 0      & \I \xi_1 &   &        & \vphantom{\ddots}   \\[0.75ex]
        \hline
        0        & 0      & 0        & 1 & 0      & 0 \vphantom{\ddots} \\[-0.75ex]
        0        & \ddots & 0        & 0 & \ddots & 0                   \\[-0.75ex]
        0        & 0      & 0        & 0 & 0      & 1 \vphantom{\ddots}
      \end{array}
      \right)
    \end{equation}
    for some matrix~$A = A(x, \xi)$ in the frame $(e_5, \dotsc, e_8)$ of $\normal[M]{X} \vert_{\partial X} = \normal[W]{\partial X} \oplus K$, where $\xi = \xi_1 e^1 + \dotsb + \xi_4 e^4$ (here $(e^1, \dotsc, e^4)$ is the dual coframe of $(e_1, \dotsc, e_4)$).
    
    The symbol $\sigma_G(x, \xi)$ of the operator $D^\ast D$ is given by
    \begin{equation*}
      \sigma_G(x, \xi) = - \abs{\xi}^2 \Id_{(\normal[M]{X})_x}
    \end{equation*}
    for~$x \in X$, $\xi \in T_x^\ast X$ \cite[Lemma~3.3]{BW93} since $D + D^\ast \colon \sections(\normal[M]{X} \oplus E) \to \sections(\normal[M]{X} \oplus E)$ is an operator of Dirac type. So $\sigma_G(x, \xi - \I \partial_t e^1) f(t) = 0$ means $\partial_t^2 f(t) = \abs{\xi}^2 f(t)$ for~$x \in \partial X$, $\xi \in T_x^\ast \partial X \subseteq T_x^\ast X$. The solutions are given by $f(t) = \E^{- \abs{\xi} t} a + \E^{\abs{\xi} t} b$ for~$a, b \in (\normal[M]{X})_x$. So the set of solutions which are bounded on~$\R_+$ is
    \begin{equation*}
      M_{x, \xi}^+ \defeq \set{e^{- \abs{\xi} t} a \colon a \in (\normal[M]{X})_x} \, \text{.}
    \end{equation*}
    Hence $f^\prime(0) = - \abs{\xi} f(0)$ for~$f \in M_{x, \xi}^+$. Therefore,
    \begin{equation}
      \sigma_\partial(x, \xi - \I \partial_t e^1) f(0) =
      \left(
      \begin{array}{@{}ccc|ccc@{}}
        \I \abs{\xi} & 0      & 0            &   &        &                     \\[-0.75ex]
        0            & \ddots & 0            &   & A      &                     \\[-0.75ex]
        0            & 0      & \I \abs{\xi} &   &        & \vphantom{\ddots}   \\[0.75ex]
        \hline
        0            & 0      & 0            & 1 & 0      & 0 \vphantom{\ddots} \\[-0.75ex]
        0            & \ddots & 0            & 0 & \ddots & 0                   \\[-0.75ex]
        0            & 0      & 0            & 0 & 0      & 1 \vphantom{\ddots}
      \end{array}
      \right)
      f(0) \, \text{.}
    \end{equation}
    This matrix is invertible for~$\xi \ne 0$. So
    \begin{equation*}
      M_{x, \xi}^+ \to (\normal[W]{\partial X} \oplus K)_x \, \text{,} \quad f \mapsto \sigma_\partial(x, \xi - \I \partial_t e^1) f(0)
    \end{equation*}
    is bijective since $M_{x, \xi}^+ \to (\normal[M]{X})_x$, $f \mapsto f(0)$ is bijective. Hence the linearisation of the boundary problem \eqref{eq:boundary-second} at the $0$\dash-section is elliptic ($D^\ast D$ is clearly an elliptic operator).
  \end{proof}
  
  \begin{lemma} \label{lemma:boundary-second-index}
    The linearisation of the boundary problem~\eqref{eq:boundary-second} at the $0$\dash-section has index~$0$.
  \end{lemma}
  
  \begin{proof}
    We first show that the operator $P$ defined in \eqref{eq:def-P} has self-adjoint symbol. Let $s, t \in \sections(\normal[M]{X} \vert_{\partial X})$, let $(e_i)_{i = 2, 3, 4}$ be any positive (i.e., $u = e_2 \times e_3 \times e_4$) local orthonormal frame of~$\partial X$, and let $(e_i)_{i = 5, \dotsc, 8}$ be any local orthonormal frame of~$\normal[M]{X} \vert_{\partial X}$. Then
    \begin{align*}
      g(P s, t) &= \sum_{i = 2}^4 g(u \times e_i \times \nabla_{e_i} s, t) - \sum_{i = 5}^8 (\nabla_s \Phi)(e_i, e_2, e_3, e_4) g(e_i, t) \\
      &= \sum_{i = 2}^4 \Phi(t, u, e_i, \nabla_{e_i} s) - (\nabla_s \Phi)(t, e_2, e_3, e_4)
    \end{align*}
    by \eqref{eq:def-P}. We have
    \begin{align*}
      \Phi(t, u, e_i, \nabla_{e_i} s) &= e_i . (\Phi(t, u, e_i, s)) - (\nabla_{e_i} \Phi)(t, u, e_i, s) - \Phi(\nabla_{e_i} t, u, e_i, s) \\
      &\phantom{{}={}} {}- \Phi(t, \nabla_{e_i} u, e_i, s) - \Phi(t, u, \nabla_{e_i} e_i, s)
    \end{align*}
    and
    \begin{align*}
      \updelta_{\partial X} (\Phi(t, u, {}\cdot{}, s)) &= {}- \sum_{i = 2}^4 e_i \interior \nabla_{e_i}(\Phi(t, u, {}\cdot{}, s)) \\
      &= {}- \sum_{i = 2}^4 (e_i .(\Phi(t, u, e_i, s)) - \Phi(t, u, \nabla_{e_i} e_i, s)) \, \text{.}
    \end{align*}
    Hence
    \begin{align*}
      g(P s, t) - g(s, P t) &= {}- \updelta_{\partial X} (\Phi(t, u, {}\cdot{}, s)) \\
      &\phantom{{}={}} {}- \sum_{i = 2}^4 ((\nabla_{e_i} \Phi)(t, u, e_i, s) + \Phi(t, \nabla_{e_i} u, e_i, s)) \\
      &\phantom{{}={}} {}+ (\nabla_t \Phi)(s, e_2, e_3, e_4) - (\nabla_s \Phi)(t, e_2, e_3, e_4) \, \text{.}
    \end{align*}
    Let $P^\ast \colon \sections(\normal[M]{X} \vert_{\partial X}) \to \sections(\normal[M]{X} \vert_{\partial X})$ be the formal adjoint of~$P$. Since
    \begin{equation*}
      \int_{\partial X} \updelta_{\partial X} (\Phi(t, u, {}\cdot{}, s)) \vol_{\partial X} = {}- \int_{\partial X} \D_{\partial X}(\mathord\ast_{\partial X}(\Phi(t, u, {}\cdot{}, s))) = 0
    \end{equation*}
    by Stokes' Theorem, we therefore get
    \begin{equation}
      \begin{split}
        P s - P^\ast s &= {}- \sum_{i = 2}^4 \sum_{j = 5}^8 ((\nabla_{e_i} \Phi)(e_j, u, e_i, s) + \Phi(e_j, \nabla_{e_i} u, e_i, s)) e_j \\
        &\phantom{{}={}} {}+ \sum_{j = 5}^8 ((\nabla_{e_j} \Phi)(s, e_2, e_3, e_4) - (\nabla_s \Phi)(e_j, e_2, e_3, e_4)) e_j \, \text{,}
      \end{split}
    \end{equation}
    which is an operator of order~$0$ in~$s$.
    
    We will use \emph{Green's formula} \cite[Proposition~3.4\,(b)]{BW93}: If $\tilde{D} \colon \sections(\mathbb{S}) \to \sections(\mathbb{S})$ is an operator of Dirac type, then
    \begin{equation}
      \inner{\tilde{D} s, t}_{L^2(\mathbb{S})} - \inner{s, \tilde{D}^\ast t}_{L^2(\mathbb{S})} = - \inner{u \cdot s, t}_{L^2(\mathbb{S} \vert_{\partial X})} \quad \text{for all~$s, t \in \sections(\mathbb{S})$.} \label{eq:greens-formula}
    \end{equation}
    In particular, if $s, t \in \sections(\normal[M]{X})$, then
    \begin{align*}
      &\inner{D^\ast D s, t}_{L^2(\normal[M]{X})} \\
      &= \inner{D s, D t}_{L^2(E)} + \inner{D s, u \times t}_{L^2(E \vert_{\partial X})} \\
      &= \inner{s, D^\ast D t}_{L^2(\normal[M]{X})} + \inner{D s, u \times t}_{L^2(E \vert_{\partial X})} - \inner{u \times s, D t}_{L^2(E \vert_{\partial X})} \\
      &= \inner{s, D^\ast D t}_{L^2(\normal[M]{X})} + \inner{\nabla_u s, t}_{L^2(\normal[M]{X} \vert_{\partial X})} + \inner{P(s \vert_{\partial X}), t}_{L^2(\normal[M]{X} \vert_{\partial X})} \\
      &\phantom{{}={}} {}- \inner{s, \nabla_u t}_{L^2(\normal[M]{X} \vert_{\partial X})} - \inner{s, P(t \vert_{\partial X})}_{L^2(\normal[M]{X} \vert_{\partial X})} \\
      &= \inner{D^\ast D t, s}_{L^2(\normal[M]{X})} + \inner{t, \nabla_u s}_{L^2(\normal[M]{X} \vert_{\partial X})} \\
      &\phantom{{}={}} {}- \inner{\nabla_u t + (P - P^\ast)(t \vert_{\partial X}), s}_{L^2(\normal[M]{X} \vert_{\partial X})} \, \text{.}
    \end{align*}
    If $r \in \sections(\normal[W]{\partial X})$, then
    \begin{align*}
      g(\nabla_u s - \nabla_{\pi_\nu(s \vert_{\partial X})} u, r) &= g(\nabla_u s, r) + g(u, \nabla_{\pi_\nu(s \vert_{\partial X})} r) \\
      &= g(\nabla_u s, r) + g(u, \nabla_r \pi_\nu(s \vert_{\partial X})) \\
      &= g(r, \nabla_u s) - g(\nabla_r u, \pi_\nu(s \vert_{\partial X}))
    \end{align*}
    since $g(u, [r, \pi_\nu(s \vert_{\partial X})]) = 0$ as $r, \pi_\nu(s \vert_{\partial X}) \in \sections(T W \vert_{\partial X})$ and $u \in \sections(\normal[M]{W} \vert_{\partial X})$. Hence
    \begin{align*}
      &\inner{D^\ast D s, t}_{L^2(\normal[M]{X})} + \inner{\pi_K(s \vert_{\partial X}), k}_{L^2(K)} \\
      &\phantom{{}={}} {}+ \inner{(\D B)_0(s) + \tfrac{1}{2} \pi_\nu((P - P^\ast)(s \vert_{\partial X})), r}_{L^2(\normal[W]{\partial X})} \\
      &= \inner{D^\ast D s, t}_{L^2(\normal[M]{X})} + \inner{\pi_K(s \vert_{\partial X}), k}_{L^2(\normal[M]{X} \vert_{\partial X})} \\
      &\phantom{{}={}} {}+ \inner{\nabla_u s \vert_{\partial X} - \nabla_{\pi_\nu(s \vert_{\partial X})} u + P(\pi_K(s \vert_{\partial X})) + \tfrac{1}{2} (P - P^\ast)(s \vert_{\partial X}), r}_{L^2(\normal[M]{X} \vert_{\partial X})} \\
      &= \inner{D^\ast D t, s}_{L^2(\normal[M]{X})} + \inner{t, \nabla_u s}_{L^2(\normal[M]{X} \vert_{\partial X})} \\
      &\phantom{{}={}} {}- \inner{\nabla_u t + (P - P^\ast)(t \vert_{\partial X}), s}_{L^2(\normal[M]{X} \vert_{\partial X})} + \inner{k, \pi_K(s \vert_{\partial X})}_{L^2(\normal[M]{X} \vert_{\partial X})} \\
      &\phantom{{}={}} {}+ \inner{r, \nabla_u s}_{L^2(\normal[M]{X} \vert_{\partial X})} - \inner{\nabla_r u, \pi_\nu(s \vert_{\partial X})}_{L^2(\normal[M]{X} \vert_{\partial X})} \\
      &\phantom{{}={}} {}+ \inner{P^\ast r, \pi_K(s \vert_{\partial X})}_{L^2(\normal[M]{X} \vert_{\partial X})} - \inner{\tfrac{1}{2} (P - P^\ast)(r), s \vert_{\partial X}}_{L^2(\normal[M]{X} \vert_{\partial X})} \\
      &= \inner{D^\ast D t, s}_{L^2(\normal[M]{X})} + \inner{t + r, \nabla_u s}_{L^2(\normal[M]{X} \vert_{\partial X})} \\
      &\phantom{{}={}} {}- \inner{\nabla_u t + \nabla_r u + (P - P^\ast)(t \vert_{\partial X}) + \tfrac{1}{2} (P - P^\ast)(r), \pi_\nu(s \vert_{\partial X})}_{L^2(\normal[M]{X} \vert_{\partial X})} \\
      &\phantom{{}={}} {}+ \inner{k - \nabla_u t - (P - P^\ast)(t \vert_{\partial X}) + P^\ast r - \tfrac{1}{2} (P - P^\ast)(r), \pi_K(s \vert_{\partial X})}_{L^2(\normal[M]{X} \vert_{\partial X})}
    \end{align*}
    for all $s, t \in \sections(\normal[M]{X})$, $k \in \sections(K)$, $r \in \sections(\normal[W]{\partial X})$. So if $(t, k, r) \in \sections(\normal[M]{X}) \oplus \sections(K) \oplus \sections(\normal[W]{\partial X})$ is $L^2$\dash-orthogonal to the image of the operator
    \begin{equation}
      \begin{split}
        \sections(\normal[M]{X}) &\to \sections(\normal[M]{X}) \oplus \sections(K) \oplus \sections(\normal[W]{\partial X}) \, \text{,} \\
        s &\mapsto (D^\ast D s, \pi_K(s \vert_{\partial X}), (\D B)_0(s) + \tfrac{1}{2} \pi_\nu((P - P^\ast)(s \vert_{\partial X}))) \, \text{,}
      \end{split} \label{eq:def-operator-index-0}
    \end{equation}
    then $\inner{D^\ast D t, s}_{L^2(\normal[M]{X})} = 0$ for all $s \in \sections(\normal[M]{X})$ with compact support in the interior of~$X$. Hence $D^\ast D t = 0$ in~$X$. Using $s \in \sections(\normal[M]{X})$ with $s \vert_{\partial X} = 0$ (so that $\nabla_u s \vert_{\partial X} \in \sections(\normal[M]{X} \vert_{\partial X})$ is arbitrary), we get $t \vert_{\partial X} + r = 0$. In particular, $\pi_K(t \vert_{\partial X}) = 0$. Using $s \in \sections(\normal[M]{X})$ with $\pi_K(s \vert_{\partial X}) = 0$, we get
    \begin{equation*}
      \pi_\nu(\nabla_u t \vert_{\partial X} - \nabla_{\pi_\nu(t \vert_{\partial X})} u + \tfrac{1}{2} (P - P^\ast)(t \vert_{\partial X})) = 0 \, \text{.}
    \end{equation*}
    Furthermore,
    \begin{equation*}
      k = \pi_K(\nabla_u t \vert_{\partial X} + \tfrac{1}{2} (P + P^\ast)(t \vert_{\partial X})) \, \text{.}
    \end{equation*}
    So the cokernel of the operator \eqref{eq:def-operator-index-0} is isomorphic to the kernel, and hence the operator \eqref{eq:def-operator-index-0} has index~$0$. Since $P - P^\ast$ is an operator of order~$0$ and the index of an elliptic boundary problem depends only on the symbols of the operators \cite[Theorem~20.1.8]{Hor85}, this shows that the index of the linearisation of the boundary problem \eqref{eq:boundary-second} at the $0$\dash-section has index~$0$.
  \end{proof}
  
  This finishes the proof of \propositionref{prop:main-proposition-boundary}.
\end{proof}

\subsection{Adapted Tubular Neighbourhood}
\label{subsec:metric}

Let $M$ be an $8$\dash-manifold with a $\Spin(7)$-structure~$\Phi$, let $X$ be a compact, connected Cayley submanifold of~$M$ with non-empty boundary, and let $W$ be a submanifold of~$M$ with $\partial X \subseteq W$ such that $X$ and $W$ meet orthogonally.

We want to parametrise submanifolds near~$X$ with boundary on~$W$ by normal vector fields using the exponential map. In general, we cannot use the exponential map of the metric~$g = g(\Phi)$ since it does not preserve~$W$ (i.e.,~$W$ is not totally geodesic with respect to~$g$). Here we construct a metric~$\hat{g}$ on~$M$ such that $W$ is totally geodesic with respect to~$\hat{g}$. Such a construction was previously used in~\cite{But03} and later in~\cite{KL09} and~\cite{Gay14}.

Since we are only interested in submanifolds near~$X$ and $\partial X$ is compact, we can shrink~$W$, if necessary, so that we can apply the Tubular Neighbourhood Theorem to~$W$.

\begin{lemma}[{cf.~\cite[Proposition~6]{But03}}] \label{lemma:metric}
  There are a tubular neighbourhood~$U$ of~$W$ in~$M$ and a metric~$\hat{g}$ on~$M$ such that $W$ is totally geodesic with respect to~$\hat{g}$ and such that $\hat{g}$ equals~$g$ outside~$U$ and $\hat{g}_x = g_x$ for all~$x \in W$.
\end{lemma}

\begin{proof}
  By the Tubular Neighbourhood Theorem, there is an open subset~$U^\prime \subseteq \normal[M]{W}$ containing the $0$\dash-section and an open subset~$U \subseteq M$ containing~$W$ such that the exponential map $\exp \vert_{U^\prime} \colon U^\prime \to U$ is a diffeomorphism. Let $\pi \colon \normal[M]{W} \to W$ denote the projection, and let $g_\nu$ denote the metric on the fibres of~$\normal[M]{W}$. Then $\tilde{g} \defeq \pi^\ast(g \vert_W) + g_\nu$ defines a metric on~$U^\prime$ such that the $0$\dash-section is totally geodesic with respect to~$\tilde{g}$. Let $\chi \colon M \to [0, 1]$ be a smooth function such that $\chi \equiv 0$ outside~$U$ and $\chi \equiv 1$ in some tubular neighbourhood of~$W$ contained in~$U$. Define
  \begin{equation*}
    \hat{g} \defeq \chi \cdot ((\exp \vert_{U^\prime})^{-1})^\ast(\tilde{g}) + (1 - \chi) \cdot g \, \text{.}
  \end{equation*}
  Then $\hat{g}$ satisfies all the conditions.
\end{proof}

As a consequence of the Tubular Neighbourhood Theorem and \lemmaref{lemma:metric}, we obtain the following version of the Tubular Neighbourhood Theorem, which is adapted to local deformations of~$X$ with boundary on~$W$. Let $\normalhat[M]{X}$ be the normal bundle with respect to the metric~$\hat{g}$, and let $\widehat{\exp}$ be the exponential map of the metric~$\hat{g}$.

\begin{proposition} \label{prop:adapted-tubular-neighbourhood}
  Let $M$ be an $8$\dash-manifold with a $\Spin(7)$-structure, let $X$ be a compact, connected Cayley submanifold of~$M$ with non-empty boundary, and let $W$ be a submanifold of~$M$ with $\partial X \subseteq W$ such that $X$ and $W$ meet orthogonally.
  
  Then there are an open neighbourhood~$U \subseteq \normalhat[M]{X}$ of the $0$\dash-section and an $8$\dash-dimensional submanifold~$T$ of~$M$ with boundary and containing~$X$ such that $\widehat{\exp} \vert_U \colon U \to T$ is a diffeomorphism and such that if $s \in \sections(U)$ with $s \vert_{\partial X} \in \sections(T W \vert_{\partial X})$, then $\widehat{\exp}_x(s(x)) \in W$ for all~$x \in \partial X$.
\end{proposition}

So $\CC^1$\dash-close submanifolds with boundary on~$W$ are parametrised by appropriate sections of the normal bundle with small $\CC^1$\dash-norm using the exponential map~$\widehat{\exp}$ of the metric~$\hat{g}$.

\subsection{Remarks about the Dimension of the Scaffold}
\label{subsec:remarks-dimension-scaffold}

If $\dim W = 3$, then $\normal[W]{\partial X}$ has rank~$0$ and $K = \normal[M]{X} \vert_{\partial X}$. Hence the linearisation of~\eqref{eq:boundary-second} at the $0$\dash-section becomes
\begin{equation}
  \begin{cases}
    D^\ast D s = 0 &\text{in~$X$,} \\
    \phantom{D^\ast D s = 0} \mathllap{s \vert_{\partial X} = 0} &\text{on~$\partial X$.} \\
  \end{cases} \label{eq:boundary-second-dim-3}
\end{equation}
If $s \in \sections(\normal[M]{X})$ satisfies~\eqref{eq:boundary-second-dim-3}, then
\begin{align*}
  0 &= \inner{D^\ast D s, s}_{L^2(\normal[M]{X})} = \inner{D s, D s}_{L^2(E)} + \inner{(D s) \vert_{\partial X}, u \times (s \vert_{\partial X})}_{L^2(E \vert_{\partial X})} \\
  &= \norm{D s}_{L^2(E)}^2
\end{align*}
by~\eqref{eq:greens-formula}. Hence $D s = 0$. So $s = 0$ by the Unique Continuation Property \cite[Corollary~8.3 and Remark~12.2]{BW93}. This shows that $X$ is rigid as a Cayley submanifold of~$M$ with boundary on~$W$ and meeting~$W$ orthogonally.

If $\dim W = 7$, then $K$ has rank~$0$, and hence \eqref{eq:boundary-first} becomes
\begin{equation}
  \begin{cases}
    \phantom{H(s \vert_{\partial X}) = 0} \mathllap{F(s) = 0} &\text{in~$X$,} \\
    H(s \vert_{\partial X}) = 0 &\text{on~$\partial X$.}
  \end{cases} \label{eq:boundary-first-dim-7}
\end{equation}
Suppose that $W$ is orientable, and let $\varphi \defeq \mathord\ast_W (\Phi \vert_W)$. Then $\varphi$ is a $G_2$\dash-structure on~$W$. Since $g(\gamma, t) = t \interior (\Phi \vert_W)$ for all $t \in \sections(T W)$, we have $H(s \vert_{\partial X}) = 0$ if and only if $\gamma \vert_{\partial X_s} = 0$ if and only if $(t \interior \Phi) \vert_{\partial X_s} = 0$ for all $t \in \sections(T W)$ if and only if $\varphi \vert_{\partial X_s} = \pm \vol_{\partial X_s}$, where $s \in \sections(\normal[M]{X})$ and $X_s \defeq \widehat{\exp}_s(X)$. So $s \in \sections(\normal[M]{X})$ satisfies~\eqref{eq:boundary-first-dim-7} if and only if $X_s$ is a Cayley submanifold of~$M$ and $\partial X_s$ is an associative submanifold of $(W, \varphi)$.

The proof of \cite[Theorem~1.2]{Gay14} implies that associative submanifolds of~$W$ are rigid for a generic $G_2$\dash-structure. An argument like in the case $\dim W = 3$ shows that this implies that also Cayley submanifolds of~$M$ with boundary on~$W$ and meeting~$W$ orthogonally are rigid for a generic $\Spin(7)$\dash-structure (here we use the fact that the map sending a $\Spin(7)$-structure to its restriction as a $G_2$\dash-structure on~$W$ is continuous and has the property that the preimage of a dense set is again dense). In the \hyperref[subsec:varying-spin7-boundary]{next section} we will show that this is also true for the other possible dimensions of~$W$, and in \sectionref{subsec:varying-scaffold} we will see that Cayley submanifolds are also rigid for a generic deformation of~$W$.

\subsection{Varying the Spin(7)-Structure}
\label{subsec:varying-spin7-boundary}

In \sectionref{subsec:varying-spin7} we showed that for a generic $\Spin(7)$\dash-structure, closed Cayley submanifolds form a smooth moduli space (see \hyperref[thm:main-theorem-spin7-closed]{Theorems~\ref*{thm:main-theorem-spin7-closed}} and \ref{thm:main-theorem-spin7-closed-c-k-alpha}). Here we prove this for compact, connected Cayley submanifolds with non-empty boundary. Note that we use a second-order operator compared to a first-order operator in the closed case. We use the $\CC^{k, \alpha}$\dash-topology for the space of all $\Spin(7)$\dash-structures.

\begin{theorem}[\theoremref{thm:main-spin7-boundary}] \label{thm:main-theorem-spin7-boundary}
  Let $k \ge 2$, let $0 < \alpha < 1$, let $M$ be an $8$\dash-manifold of class~$\CC^{k + 2, \alpha}$ with a $\Spin(7)$-structure~$\Phi$ of class~$\CC^{k + 2, \alpha}$, let $X$ be a compact, connected Cayley submanifold of~$M$ of class~$\CC^{k + 2, \alpha}$ with non-empty boundary, and let $W$ be a submanifold of~$M$ of class~$\CC^{k + 2, \alpha}$ with $\partial X \subseteq W$ such that $X$ and $W$ meet orthogonally.
  
  Then for every generic $\Spin(7)$\dash-structure~$\Psi$ of class~$\CC^{k, \alpha}$ that is $\CC^{k, \alpha}$\dash-close to~$\Phi$, the moduli space of all Cayley submanifolds of $(M, \Psi)$ of class~$\CC^{k + 1, \alpha}$ that are $\CC^{2, \alpha}$\dash-close to~$X$ with boundary on~$W$ and meeting~$W$ orthogonally (with respect to the metric induced by~$\Psi$) is a finite set (possibly empty).
\end{theorem}

As we will see in the proof, we may restrict the class of $\Spin(7)$\dash-structures~$\Psi$ to those inducing the same metric as~$\Phi$.

\begin{proof}
  Recall from the proof of \theoremref{thm:main-theorem-spin7-closed} that there are an open tubular neighbourhood $V \subseteq \altforms^4_1 M \oplus \altforms^4_7 M \oplus \altforms^4_{35} M$ of the $0$\dash-section and a smooth bundle morphism $\Theta \colon V \to \altforms^4 M$ which parametrises $\Spin(7)$\dash-structures on~$M$ that are $\CC^0$\dash-close to~$\Phi$. Let
  \begin{equation}
    \pi_E \colon \altforms^2 M \vert_X \to E
  \end{equation}
  denote the orthogonal projection, where the vector bundle~$E$ of rank~$4$ over~$X$ is defined as in~\eqref{eq:def-E}. Furthermore, let $U \subseteq \normal[M]{X}$ and $\widehat{\exp}$ be defined as in the proof of \propositionref{prop:main-proposition-boundary}, and let
  \begin{equation}
    \tilde{F} \colon \sections(U) \oplus \sections(V) \to \sections(E) \, \text{,} \quad (s, \chi) \mapsto \pi_E((\widehat{\exp}_s)^\ast(\tau_{\Theta(\chi)})) \, \text{,} \label{eq:def-F-tilde}
  \end{equation}
  where $\tau_{\Theta(\chi)} \in \forms^4(M, \altforms^2 M)$ is defined as in \eqref{eq:def-tau} with respect to the $\Spin(7)$-structure $\Theta(\chi)$. Then
  \begin{equation}
    (\D \tilde{F})_{(0, 0)}(0, \chi) = - \pi_E(e \vert_X) \label{eq:linearisation-F-tilde}
  \end{equation}
  for $e \in \sections(\altforms^2_7 M)$ and $\chi = h(\tau_\Phi, e)$ by \lemmaref{lemma:linearisation-F-tilde}, where $h$ is the metric on~$\altforms^2_7 M$. Let $D^\ast \colon \sections(E) \to \sections(\normal[M]{X})$ be the formal adjoint of~$D$, where $D \colon \sections(\normal[M]{X}) \to \sections(E)$ is defined in~\eqref{eq:def-D}, and let
  \begin{equation}
    \tilde{G} \colon \sections(U) \oplus \sections(V) \to \sections(\normal[M]{X}) \, \text{,} \quad (s, \chi) \mapsto D^\ast(\tilde{F}(s, \chi)) \, \text{.} \label{eq:def-G-tilde}
  \end{equation}
  Then
  \begin{equation}
    (\D \tilde{G})_{(0, 0)}(0, \chi) = - D^\ast(\pi_E(e \vert_X)) \label{eq:linearisation-G-tilde}
  \end{equation}
  for $e \in \sections(\altforms^2_7 M)$ and $\chi = h(\tau_\Phi, e)$ since $D^\ast$ is linear. Let $U_{\partial X} \subseteq \normal[M]{X} \vert_{\partial X}$ be defined as in the proof of \propositionref{prop:main-proposition-boundary}, and let
  \begin{equation}
    \tilde{H} \colon \sections(U_{\partial X}) \oplus \sections(V) \to \sections(\normal[W]{\partial X}) \, \text{,} \quad (s, \chi) \mapsto \pi_\nu((\widehat{\exp}_{\pi_\nu(s)})^\ast(\gamma_{\Theta(\chi)})) \, \text{,} \label{eq:def-H-tilde}
  \end{equation}
  where $\pi_\nu$ is defined in~\eqref{eq:def-pi-nu} and $\gamma_{\Theta(\chi)} \in \forms^3(W, T W)$ is defined as in \eqref{eq:def-gamma} with respect to the $\Spin(7)$-structure $\Theta(\chi)$.
  
  \begin{lemma} \label{lemma:linearisation-H-tilde}
    Let $e \in \sections(\altforms^2_7 M)$, and let $\chi = h(\tau_\Phi, e)$, where $h$ is the metric on~$\altforms^2_7 M$ (note that $\chi \in \forms^4_7(M)$ by \eqref{eq:inner-tau} and \eqref{eq:forms-4-7}). Then
    \begin{equation}
      (\D \tilde{H})_{(0, 0)}(0, \chi) = \pi_\nu(\rho^{-1}(\pi_E(e \vert_{\partial X}))) \, \text{,} \label{eq:linearisation-H-tilde}
    \end{equation}
    where $\rho$ is defined in \eqref{eq:def-rho}.
  \end{lemma}
  
  \begin{proof}
    Since $\chi \in \forms^4_7(M)$, there is a path $(\Phi_t)_{t \in (- \eps, \eps)}$ of $\Spin(7)$\dash-structures on~$M$ with $\Phi_0 = \Phi$ and $\frac{\D}{\D t} \Phi_t \big\vert_{t = 0} = \chi$ such that the metric induced by~$\Phi_t$ is the same metric as the metric induced by~$\Phi$ for $t \in (- \eps, \eps)$ \cite[Proposition~5.3.1]{Kar05}. Write $\Phi_t = \Theta(\chi_t)$ with $\frac{\D}{\D t} \chi_t \big\vert_{t = 0} = \chi$.
    
    Let $k$ be the dimension of~$W$, and let $(e_1, \dotsc, e_8)$ be a local orthonormal frame of~$M$ such that $(e_1, \dotsc, e_4)$ is a positive (i.e., $e_1 = e_2 \times e_3 \times e_4$) frame of~$X$ with $e_1 = u$ (where $u \in \sections(\normal[X]{\partial X})$ is the inward-pointing unit normal vector field of~$\partial X$ in~$X$) and $(e_2, \dotsc, e_{k + 1})$ is a frame of~$W$. Then
    \begin{align*}
      (\D \tilde{H})_{(0, 0)}(0, \chi) &= \frac{\D}{\D t} \tilde{H}(0, \chi_t) \biggr\vert_{t = 0} = \frac{\D}{\D t} \pi_\nu(\gamma_{\Theta(\chi_t)}) \biggr\vert_{t = 0} \\
      &= \frac{\D}{\D t} \sum_{i = 2}^{k + 1} (e_i \interior \Phi_t) \vert_{\partial X} \otimes \pi_\nu(e_i) \biggr\vert_{t = 0} \\
      &= \sum_{i = 2}^{k + 1} (e_i \interior \chi) \vert_{\partial X} \otimes \pi_\nu(e_i) \\
      &= \sum_{i = 5}^{k + 1} \chi(e_i, e_2, e_3, e_4) \vol_{\partial X} \otimes e_i \\
      &= \sum_{i = 5}^{k + 1} h(\tau(e_i, e_2, e_3, e_4), e) \vol_{\partial X} \otimes e_i \\
      &= - \sum_{i = 5}^{k + 1} h(e_i \times (e_2 \times e_3 \times e_4), e) \vol_{\partial X} \otimes e_i \\
      &= \sum_{i = 5}^{k + 1} h(e, e_1 \times e_i) \vol_{\partial X} \otimes e_i \\
      &= \pi_\nu(\rho^{-1}(\pi_E(e \vert_{\partial X})))
    \end{align*}
    by \eqref{eq:def-tau}.
  \end{proof}
  
  Let
  \begin{equation}
    \begin{split}
      \tilde{B} \colon \sections(U) \oplus \sections(V) &\to \sections(\normal[W]{\partial X}) \, \text{,} \\
      (s, \chi) &\mapsto \pi_\nu(\rho^{-1}(\tilde{F}(s, \chi) \vert_{\partial X})) + \tilde{H}(s \vert_{\partial X}, \chi) \, \text{.} 
    \end{split} \label{eq:def-B-tilde}
  \end{equation}
  Then
  \begin{equation}
    (\D \tilde{B})_{(0, 0)}(0, \chi) = \pi_\nu(\rho^{-1}(- \pi_E(e \vert_{\partial X}))) + \pi_\nu(\rho^{-1}(\pi_E(e \vert_{\partial X}))) = 0 \label{eq:linearisation-B-tilde}
  \end{equation}
  for $e \in \sections(\altforms^2_7 M)$ and $\chi = h(\tau_\Phi, e)$ by \eqref{eq:linearisation-F-tilde} and \eqref{eq:linearisation-H-tilde}, where $h$ is the metric on~$\altforms^2_7 M$.
  
  \begin{lemma} \label{lemma:spin7-boundary-surjective}
    The map
    \begin{equation}
      \begin{gathered}
        \sections(\normal[M]{X}) \oplus \sections(\altforms^4_1 M \oplus \altforms^4_7 M \oplus \altforms^4_{35} M) \to \sections(\normal[M]{X}) \oplus \sections(K) \oplus \sections(\normal[W]{\partial X}) \, \text{,} \\
        (s, \chi) \mapsto ((\D \tilde{G})_{(0, 0)}(s, \chi), \pi_K(s \vert_{\partial X}), (\D \tilde{B})_{(0, 0)}(s, \chi))
      \end{gathered}
    \end{equation}
    is surjective, where the vector bundle~$K$ over~$\partial X$ is defined as in the proof of \propositionref{prop:main-proposition-boundary} and $\pi_K$ is defined in~\eqref{eq:def-pi-K}.
  \end{lemma}
  
  \begin{proof}
    We have seen that this map is given by
    \begin{equation}
      \begin{split}
        (s, \chi) &\mapsto (D^\ast D s - D^\ast(\pi_E(e \vert_X)), \pi_K(s \vert_{\partial X}), \\
        &\phantom{{}\mapsto (} \pi_\nu(\nabla_u s \vert_{\partial X} - \nabla_{\pi_\nu(s \vert_{\partial X})} u + P(\pi_K(s \vert_{\partial X}))))
      \end{split}
    \end{equation}
    for $e \in \sections(\altforms^2_7 M)$ such that $\chi = h(\tau_\Phi, e)$ (where $h$ is the metric on~$\altforms^2_7 M$) by \eqref{eq:linearisation-G}, \eqref{eq:linearisation-B}, \eqref{eq:linearisation-G-tilde}, and \eqref{eq:linearisation-B-tilde} since $\pi_K$ is linear, where $P$ is defined in~\eqref{eq:def-P}.
    
    Let $(t, k, r) \in \sections(\normal[M]{X}) \oplus \sections(K) \oplus \sections(\normal[W]{\partial X})$. Then there is some $s \in \sections(\normal[M]{X})$ such that
    \begin{equation*}
      \pi_K(s \vert_{\partial X}) = k \, \text{,} \quad \pi_\nu(s \vert_{\partial X}) = 0 \, \text{,} \quad \text{and} \quad \nabla_u s \vert_{\partial X} = r - P k
    \end{equation*}
    since
    \begin{equation*}
      \sections(\normal[M]{X}) \to \sections(\normal[M]{X} \vert_{\partial X}) \oplus \sections(\normal[M]{X} \vert_{\partial X}) \, \text{,} \quad s \mapsto (s \vert_{\partial X}, \nabla_u s \vert_{\partial X})
    \end{equation*}
    is surjective. Furthermore, the operator $D^\ast \colon \sections(E) \to \sections(\normal[M]{X})$ is surjective by \cite[Theorem~9.1]{BW93}. So there is some $e \in \sections(\altforms^2_7 M)$ such that
    \begin{equation*}
      D^\ast(\pi_E(e \vert_X)) = D^\ast D s - t \, \text{.}
    \end{equation*}
    Hence
    \begin{equation*}
      (\D \tilde{G})_{(0, 0)}(s, \chi) = t \, \text{,} \quad \pi_K(s \vert_{\partial X}) = k \, \text{,} \quad \text{and} \quad (\D \tilde{B})_{(0, 0)}(s, \chi) = r
    \end{equation*}
    for $\chi = h(\tau_\Phi, e)$.
  \end{proof}
  
  \begin{lemma} \label{lemma:compactness}
    Let $0 < \beta < \alpha$. Then the moduli space of all solutions in $\CC^{2, \alpha}(\normal[M]{X})$ of the boundary problem~\eqref{eq:boundary-second} that are $\CC^{2, \alpha}$\dash-close to~$0$ is compact in $\CC^{2, \beta}(\normal[M]{X})$.
  \end{lemma}
  
  \begin{proof}
    Let $\mathcal{M}_\alpha \subseteq \CC^{2, \alpha}(\normal[M]{X})$ be the moduli space of all solutions of the boundary problem~\eqref{eq:boundary-second} that are $\CC^{2, \alpha}$\dash-close to~$0$. Since the embedding operator $\CC^{2, \alpha}(\normal[M]{X}) \to \CC^{2, \beta}(\normal[M]{X})$ is compact (\cite[Proposition~1 in Section~2.3.2.6]{RS82}), the closure of $\mathcal{M}_\alpha$ in $\CC^{2, \beta}(\normal[M]{X})$ is compact. Furthermore, if $R > 0$ and $s_n \in \mathcal{M}_\alpha$ satisfies $\norm{s_n}_{\CC^{2, \alpha}} \le R$ for all~$n$ and $s_n$ converges in $\CC^{2, \beta}(\normal[M]{X})$ to some $s \in \CC^{2, \beta}(\normal[M]{X})$, then $s \in \CC^{2, \alpha}(\normal[M]{X})$ with $\norm{s}_{\CC^{2, \alpha}} \le R$. Note that $s$ also satisfies \eqref{eq:boundary-second} since $G$, $\pi_K$, and $B$ are also continuous with respect to $\CC^{2, \beta}(\normal[M]{X})$. Hence $s \in \mathcal{M}_\alpha$. So $\mathcal{M}_\alpha$ is a compact subset of $\CC^{2, \beta}(\normal[M]{X})$.
  \end{proof}
  
  The coefficients of~$D$ are in~$\CC^{k + 1, \alpha}$ by~\eqref{eq:def-D} since $\Phi \in \CC^{k + 2, \alpha}(\altforms^4 M)$. The first-order coefficients of~$D$ are even in~$\CC^{k + 2, \alpha}$. Hence the coefficients of~$D^\ast$ are in~$\CC^{k + 1, \alpha}$. So the coefficients of~$D^\ast D$ are in~$\CC^{k, \alpha}$. Also the coefficients of~$P$ and $P^\ast$ are in~$\CC^{k + 1, \alpha}$ by~\eqref{eq:relation-D-P}. Further, the coefficients of~$\pi_K$ and $\pi_\nu$ are in~$\CC^{k + 2, \alpha}$.
  
  Now $\tilde{G}$ extends to a map
  \begin{equation}
    \tilde{G}_{k, \alpha} \colon \CC^{2, \alpha}(U) \oplus \CC^{k, \alpha}(V) \to \CC^{0, \alpha}(\normal[M]{X})
  \end{equation}
  of class~$\CC^{k - 1}$ by \propositionref{prop:c-ell} since $D^\ast \colon \CC^{1, \alpha}(E) \to \CC^{0, \alpha}(\normal[M]{X})$ is a linear first-order differential operator. Furthermore, $\tilde{B}$ extends to a map
  \begin{equation}
    \tilde{B}_{k, \alpha} \colon \CC^{2, \alpha}(U) \oplus \CC^{k, \alpha}(V) \to \CC^{1, \alpha}(\normal[W]{\partial X})
  \end{equation}
  of class~$\CC^{k - 1}$.
  
  The proof of \lemmaref{lemma:spin7-boundary-surjective} shows that the map
  \begin{equation*}
    \begin{split}
      &\CC^{k + 1, \alpha}(\normal[M]{X}) \oplus \CC^{k, \alpha}(\altforms^4_1 M \oplus \altforms^4_7 M \oplus \altforms^4_{35} M) \\
      &\to \CC^{k - 1, \alpha}(\normal[M]{X}) \oplus \CC^{k + 1, \alpha}(K) \oplus \CC^{k, \alpha}(\normal[W]{\partial X}) \, \text{,} \\
      (s, \chi) &\mapsto ((\D \tilde{G}_{k, \alpha})_{(0, 0)}(s, \chi), \pi_K(s \vert_{\partial X}), (\D \tilde{B}_{k, \alpha})_{(0, 0)}(s, \chi))
    \end{split}
  \end{equation*}
  is surjective since the maps
  \begin{equation*}
    \CC^{k + 1, \alpha}(\normal[M]{X}) \to \CC^{k + 1, \alpha}(\normal[M]{X} \vert_{\partial X}) \oplus \CC^{k, \alpha}(\normal[M]{X} \vert_{\partial X}) \, \text{,} \quad s \mapsto (s \vert_{\partial X}, \nabla_u s \vert_{\partial X})
  \end{equation*}
  and $D^\ast \colon \CC^{k, \alpha}(E) \to \CC^{k - 1, \alpha}(\normal[M]{X})$ are surjective.
  
  The $L^2$\dash-orthogonal complement of the image of the linearisation of~\eqref{eq:boundary-second} at the $0$\dash-section is given by all $(t, k, r) \in \CC^{2, \alpha}(\normal[M]{X}) \oplus \CC^{1, \alpha}(K) \oplus \CC^{2, \alpha}(\normal[W]{\partial X})$ such that (compare the proof of \lemmaref{lemma:boundary-second-index})
  \begin{equation}
    \begin{cases}
      \phantom{\pi_\nu(\nabla_u t \vert_{\partial X} - \nabla_{\pi_\nu(t \vert_{\partial X})} u + (P - P^\ast)(t \vert_{\partial X})) = 0} \mathllap{D^\ast D t = 0} &\text{in $X$,} \\
      \phantom{\pi_\nu(\nabla_u t \vert_{\partial X} - \nabla_{\pi_\nu(t \vert_{\partial X})} u + (P - P^\ast)(t \vert_{\partial X})) = 0} \mathllap{\pi_K(t \vert_{\partial X}) = 0} &\text{on $\partial X$,} \\
      \pi_\nu(\nabla_u t \vert_{\partial X} - \nabla_{\pi_\nu(t \vert_{\partial X})} u + (P - P^\ast)(t \vert_{\partial X})) = 0 &\text{on $\partial X$}
    \end{cases} \label{eq:orthogonal-complement-t}
  \end{equation}
  with
  \begin{equation}
    k = \pi_K(\nabla_u t \vert_{\partial X} + P(t \vert_{\partial X})) \quad \text{and} \quad r = t \vert_{\partial X} \, \text{.} \label{eq:orthogonal-complement-k-r}
  \end{equation}
  So $t \in \CC^{k + 2, \alpha}(\normal[M]{X})$ by Elliptic Regularity. Hence $k \in \CC^{k + 1, \alpha}(K)$ and $r \in \CC^{k + 2, \alpha}(\normal[W]{\partial X})$. Thus the map
  \begin{equation*}
    \begin{split}
      &\CC^{2, \alpha}(\normal[M]{X}) \oplus \CC^{k, \alpha}(\altforms^4_1 M \oplus \altforms^4_7 M \oplus \altforms^4_{35} M) \\
      &\to \CC^{0, \alpha}(\normal[M]{X}) \oplus \CC^{2, \alpha}(K) \oplus \CC^{1, \alpha}(\normal[W]{\partial X}) \, \text{,} \\
      (s, \chi) &\mapsto ((\D \tilde{G}_{k, \alpha})_{(0, 0)}(s, \chi), \pi_K(s \vert_{\partial X}), (\D \tilde{B}_{k, \alpha})_{(0, 0)}(s, \chi))
    \end{split}
  \end{equation*}
  is surjective. Hence there is a ($\CC^{2, \alpha} \oplus \CC^{k, \alpha}$)\dash-neighbourhood $\tilde{U}_1 \subseteq \CC^{2, \alpha}(U) \oplus \CC^{k, \alpha}(V)$ of $(0, 0)$ such that $((\D \tilde{G}_{k, \alpha})_{(s, \chi)}, \pi_K, (\D \tilde{B}_{k, \alpha})_{(s, \chi)})$ is surjective for all $(s, \chi) \in \tilde{U}_1$.
  
  For fixed $\chi \in \CC^{k, \alpha}(V)$, the equation
  \begin{equation}
    \tilde{G}_{k, \alpha}(s, \chi) = D^\ast(\pi_E((\widehat{\exp}_s)^\ast(\tau_{\Theta(\chi)}))) = 0
  \end{equation}
  is a nonlinear partial differential equation of order~$2$ in~$s$ and the equation
  \begin{equation}
    \tilde{B}_{k, \alpha}(s, \chi) = \pi_\nu(\rho^{-1}(\pi_E((\widehat{\exp}_s)^\ast(\tau_{\Theta(\chi)}))) + (\widehat{\exp}_{\pi_\nu(s \vert_{\partial X})})^\ast(\gamma_{\Theta(\chi)})) = 0
  \end{equation}
  is a nonlinear partial differential equation of order~$1$ in~$s$ on the boundary. Since the linearisation of
  \begin{equation}
    \begin{cases}
      \phantom{\tilde{B}_{k, \alpha}(s, \chi) = 0} \mathllap{\tilde{G}_{k, \alpha}(s, \chi) = 0} &\text{in~$X$,} \\
      \phantom{\tilde{B}_{k, \alpha}(s, \chi) = 0} \mathllap{\pi_K(s \vert_{\partial X}) = 0} &\text{on~$\partial X$,} \\
      \tilde{B}_{k, \alpha}(s, \chi) = 0 &\text{on~$\partial X$}
    \end{cases} \label{eq:boundary-second-chi}
  \end{equation}
  at~$0$ is elliptic for $\chi = 0$, there is a ($\CC^{2, \alpha} \oplus \CC^{k, \alpha}$)\dash-neighbourhood $\tilde{U}_2 \subseteq \CC^{2, \alpha}(U) \oplus \CC^{k, \alpha}(V)$ of $(0, 0)$ such that
  \begin{equation*}
    \begin{cases}
      \phantom{(\D \tilde{B}_{k, \alpha})_{(s, \chi)}(s^\prime, 0) = 0} \mathllap{(\D \tilde{G}_{k, \alpha})_{(s, \chi)}(s^\prime, 0) = 0} &\text{in~$X$,} \\
      \phantom{(\D \tilde{B}_{k, \alpha})_{(s, \chi)}(s^\prime, 0) = 0} \mathllap{\pi_K(s^\prime \vert_{\partial X}) = 0} &\text{on~$\partial X$,} \\
      (\D \tilde{B}_{k, \alpha})_{(s, \chi)}(s^\prime, 0) = 0 &\text{on~$\partial X$}
    \end{cases}
  \end{equation*}
  is an elliptic boundary problem in $s^\prime \in \CC^{2, \alpha}(\normal[M]{X})$ for all $(s, \chi) \in \tilde{U}_2$. In particular, if $(s, \chi) \in \tilde{U}_2$ satisfies~\eqref{eq:boundary-second-chi}, then $s \in \CC^{k + 1, \alpha}(\normal[M]{X})$ by Elliptic Regularity \cite[Theorem~6.8.2]{Mor66} since $\tau_{\Theta(\chi)} \in \CC^{k, \alpha}(\altforms^4 M \otimes \altforms^2 M)$ and $\gamma_{\Theta(\chi)} \in \CC^{k, \alpha}(\altforms^3 W \otimes T W)$.
  
  Let $\tilde{U} \defeq \tilde{U}_1 \cap \tilde{U}_2$. Then the above argumentation shows that if $(s, \chi) \in \tilde{U}$ satisfies~\eqref{eq:boundary-second-chi}, then
  \begin{equation*}
    \begin{split}
      ((\D \tilde{G}_{k, \alpha})_{(s, \chi)}, \pi_K, (\D \tilde{B}_{k, \alpha})_{(s, \chi)}) \colon &\CC^{2, \alpha}(\normal[M]{X}) \oplus \CC^{k, \alpha}(\altforms^4_1 M \oplus \altforms^4_7 M \oplus \altforms^4_{35} M) \\
      &\to \CC^{0, \alpha}(\normal[M]{X}) \oplus \CC^{2, \alpha}(K) \oplus \CC^{1, \alpha}(\normal[W]{\partial X})
    \end{split}
  \end{equation*}
  is surjective and
  \begin{equation*}
    \begin{split}
      &((\D \tilde{G}_{k, \alpha})_{(s, \chi)}, \pi_K, (\D \tilde{B}_{k, \alpha})_{(s, \chi)}) \vert_{\CC^{2, \alpha}(\normal[M]{X})} \colon \\
      &\CC^{2, \alpha}(\normal[M]{X}) \to \CC^{0, \alpha}(\normal[M]{X}) \oplus \CC^{2, \alpha}(K) \oplus \CC^{1, \alpha}(\normal[W]{\partial X})
    \end{split}
  \end{equation*}
  is Fredholm. Furthermore, the Fredholm index is the same for all of these operators as we may assume \wolog\ that $\tilde{U}$ is connected. So the Fredholm index is~$0$ by \lemmaref{lemma:boundary-second-index}.
  
  So \theoremref{thm:genericity-general} implies that for every generic $\chi \in \CC^{k, \alpha}(V)$ that is $\CC^{k, \alpha}$\dash-close to~$0$, the moduli space of all solutions $s \in \CC^{k + 1, \alpha}(\normal[M]{X})$ of the boundary problem \eqref{eq:boundary-second-chi} that are $\CC^{2, \alpha}$\dash-close to~$0$ is either empty or a $\CC^{k - 1}$\dash-manifold of dimension~$0$ (i.e., a discrete set). This discrete set is finite by compactness (\lemmaref{lemma:compactness}). The result follows since the boundary problem \eqref{eq:boundary-first} implies \eqref{eq:boundary-second} and a subset of a finite set is again a finite set.
\end{proof}

\subsection{Varying the Scaffold}
\label{subsec:varying-scaffold}

In the \hyperref[subsec:varying-spin7-boundary]{last section} we proved that for a generic $\Spin(7)$\dash-structure, compact, connected Cayley submanifolds with non-empty boundary are rigid (see \theoremref{thm:main-theorem-spin7-boundary}). Here we show that also for a generic deformation of the scaffold (the submanifold containing the boundary of the Cayley submanifold), compact, connected Cayley submanifolds with non-empty boundary are rigid. Note that there are small differences to \theoremref{thm:main-theorem-spin7-boundary} in terms of the regularity of the objects involved. We use the $\CC^{k, \alpha}$\dash-topology for the space of all local deformations of~$W$.

\begin{theorem}[\theoremref{thm:main-scaffold}] \label{thm:main-theorem-scaffold}
  Let $k \ge 2$, let $0 < \alpha < 1$, let $M$ be an $8$\dash-manifold of class~$\CC^{k + 1, \alpha}$ with a $\Spin(7)$-structure~$\Phi$ of class~$\CC^{k + 1, \alpha}$, let $X$ be a compact, connected Cayley submanifold of~$M$ of class~$\CC^{k + 1, \alpha}$ with non-empty boundary, and let $W$ be a submanifold of~$M$ of class~$\CC^{k + 1, \alpha}$ with $\partial X \subseteq W$ such that $X$ and $W$ meet orthogonally.
  
  Then for every generic local deformation~$W^\prime$ of~$W$ of class~$\CC^{k, \alpha}$ that is $\CC^{k, \alpha}$\dash-close to~$W$, the moduli space of all Cayley submanifolds of~$M$ of class~$\CC^{k, \alpha}$ that are $\CC^{2, \alpha}$\dash-close to~$X$ with boundary on~$W^\prime$ and meeting~$W^\prime$ orthogonally is a finite set (possibly empty).
\end{theorem}

As we will see in the proof, we may restrict the deformations of~$W$ to submanifolds~$W^\prime$ with $\partial X \subseteq W^\prime$.

\begin{proof}
  We start with the following lemma.
  
  \begin{lemma} \label{lemma:extension}
    There is a linear operator $\sigma \colon \sections(\normal[M]{W}) \to \sections(T M)$ such that for all $t \in \sections(\normal[M]{W})$:
    \begin{compactenum}[(i)]
      \item $\sigma(t) \vert_W = t$,
      \item \label{lemma:extension:nabla} $\nabla_r(\sigma(t)) \vert_{\partial X} = 0$ for all $r \in \sections(\normal[M]{W} \vert_{\partial X})$, and
      \item \label{lemma:extension:partial-X-0} if $t \vert_{\partial X} = 0$, then $\sigma(t) \vert_X = 0$.
    \end{compactenum}
    Furthermore, $\sigma$ extends to a bounded linear operator
    \begin{equation}
      \sigma_{k, \alpha} \colon \CC^{k, \alpha}(\normal[M]{W}) \to \CC^{k, \alpha}(T M)
    \end{equation}
    for all $k \ge 0$, $0 \le \alpha \le 1$.
  \end{lemma}
  
  \begin{proof}
    Similarly to the proof of \lemmaref{lemma:metric}, there is a metric $\tilde{g}$ on~$M$ such that $X$ is totally geodesic with respect to~$\tilde{g}$ and such that $\tilde{g}_x = g_x$ for all~$x \in \partial X$. By the Tubular Neighbourhood Theorem, there are an open tubular neighbourhood $\tilde{U} \subseteq \normaltilde[M]{W}$ of the $0$\dash-section and a neighbourhood $U^\prime$ of~$W$ in~$M$ such that the exponential map $\widetilde{\exp} \vert_{\tilde{U}} \colon \tilde{U} \to U^\prime$ of the metric~$\tilde{g}$ is a diffeomorphism.
    
    For $(x, v) \in \tilde{U}$, let
    \begin{equation*}
      \pi_{x, v} \colon T_x M \to T_{\scriptwidetilde{\exp}_x(v)} M
    \end{equation*}
    denote the parallel transport along the curve $[0, 1] \to M$, $t \mapsto \widetilde{\exp}_x(t v)$. Define $\tilde{\sigma} \colon \sections(\normaltilde[M]{W}) \to \sections(T M \vert_{U^\prime})$ by
    \begin{equation*}
      \tilde{\sigma}(t)_{\scriptwidetilde{\exp}_x(v)} \defeq \pi_{x, v}(t_x)
    \end{equation*}
    for $t \in \sections(\normaltilde[M]{W})$ and $(x, v) \in \tilde{U}$. Let $\chi \colon M \to [0, 1]$ be a smooth function such that $\chi \equiv 0$ outside~$U^\prime$ and $\chi \equiv 1$ in some tubular neighbourhood of~$W$ contained in~$U^\prime$. For $t \in \sections(\normal[M]{W})$, let $\tilde{t} \in \sections(\normaltilde[M]{W})$ be the corresponding vector field under the isomorphism $\normal[M]{W} \cong (T M \vert_W) / T W \cong \normaltilde[M]{W}$, and define
    \begin{equation*}
      \sigma(t) \defeq \chi \cdot \tilde{\sigma}(\tilde{t}) \, \text{.}
    \end{equation*}
    Then $\sigma$ satisfies all the conditions.
  \end{proof}
  
  Note that if $t \in \sections(\normal[M]{W})$ has small $\CC^1$\dash-norm, then $\widehat{\exp}_{\sigma(t)} \colon M \to M, x \mapsto \widehat{\exp}_x((\sigma(t))_x)$ is a diffeomorphism, where $\widehat{\exp}$ is defined as in the proof of \propositionref{prop:main-proposition-boundary}. So if $W^\prime = \widehat{\exp}_t(W)$ for some $t \in \sections(\normal[M]{W})$ with small $\CC^1$\dash-norm, then the submanifolds of~$M$ that are $\CC^1$\dash-close to~$X$ with boundary on~$W^\prime$ are of the form $\widehat{\exp}_{\sigma(t)}(\widehat{\exp}_s(X))$ for some $s \in \sections(\normal[M]{X})$ with small $\CC^1$\dash-norm such that $s \vert_{\partial X} \in \sections(T W \vert_{\partial X})$.
  
  Let $V^\prime \subseteq \normal[M]{W}$ be an open tubular neighbourhood of the $0$\dash-section such that the exponential map $\widehat{\exp} \vert_{V^\prime} \colon V^\prime \to \widehat{\exp}(V^\prime)$ is a diffeomorphism, let $U \subseteq \normal[M]{X}$ be defined as in the proof of \propositionref{prop:main-proposition-boundary}, and let
  \begin{equation}
    \begin{split}
      \hat{F} \colon \sections(U) \oplus \sections(V^\prime) &\to \forms^4(X, E) \cong \sections(E) \, \text{,} \\
      (s, t) &\mapsto \pi_E((\widehat{\exp}_s)^\ast((\widehat{\exp}_{\sigma(t)})^\ast(\tau))) \, \text{,}
    \end{split} \label{eq:def-F-hat}
  \end{equation}
  where the vector bundle~$E$ of rank~$4$ over~$X$ is defined as in~\eqref{eq:def-E}, $\pi_E$ is defined in \eqref{eq:def-pi-E}, and $\tau \in \forms^4(M, \altforms^2_7 M)$ is defined as in~\eqref{eq:def-tau}. Then
  \begin{equation}
    (\D \hat{F})_{(0, 0)}(0, t) = D (\sigma(t) \vert_X)^\perp \, \text{,} \label{eq:linearisation-F-hat}
  \end{equation}
  where $D \colon \sections(\normal[M]{X}) \to \sections(E)$ is defined in~\eqref{eq:def-D} and $(\sigma(t) \vert_X)^\perp$ is the pointwise orthogonal projection of~$\sigma(t) \vert_X$ onto $\normal[M]{X}$. This follows from the proof of \theoremref{thm:explicit-formula} using that $\tau \vert_X = 0$. Let
  \begin{equation}
    \hat{G} \colon \sections(U) \oplus \sections(V^\prime) \to \sections(\normal[M]{X}) \, \text{,} \quad (s, t) \mapsto D^\ast(\hat{F}(s, t)) \, \text{,} \label{eq:def-G-hat}
  \end{equation}
  where $D^\ast \colon \sections(E) \to \sections(\normal[M]{X})$ is the formal adjoint of~$D$. Then
  \begin{equation}
    (\D \hat{G})_{(0, 0)}(0, t) = D^\ast D (\sigma(t) \vert_X)^\perp \label{eq:linearisation-G-hat}
  \end{equation}
  since $D^\ast$ is linear.
  
  Similarly to \lemmaref{lemma:extension}, there is also a linear operator
  \begin{equation}
    \hat{\sigma} \colon \sections(\normal[W]{\partial X}) \to \sections(T W) \label{eq:def-sigma-hat}
  \end{equation}
  such that $\hat{\sigma}(r) \vert_{\partial X} = r$ for all $r \in \sections(\normal[W]{\partial X})$. Furthermore, $\hat{\sigma}$ extends to a bounded linear operator $\hat{\sigma}_{k, \alpha} \colon \CC^{k, \alpha}(\normal[W]{\partial X}) \to \CC^{k, \alpha}(T W)$ for all $k \ge 0$, $0 \le \alpha \le 1$. Note that if $r \in \sections(\normal[W]{\partial X})$ has small $\CC^1$\dash-norm, then $\widehat{\exp}_{\hat{\sigma}(r)} \colon W \to W, x \mapsto \widehat{\exp}_x((\hat{\sigma}(r))_x)$ is a diffeomorphism.
  
  Define $\hat{\gamma} \in \forms^3(M, T M)$ and $\mu \in \forms^1(M, T M)$ by
  \begin{equation}
    \hat{\gamma}(a, b, c) \defeq a \times b \times c \quad \text{and} \quad \mu(a) \defeq a \label{eq:def-gamma-hat-mu}
  \end{equation}
  for $a, b, c \in \sections(T M)$, let $U_{\partial X} \subseteq \normal[M]{X} \vert_{\partial X}$ be defined as in the proof of \propositionref{prop:main-proposition-boundary}, let
  \begin{equation}
    \begin{split}
      \hat{H}_1 \colon \sections(U_{\partial X}) \oplus \sections(V^\prime) &\to \forms^3(\partial X, T M \vert_{\partial X}) \cong \sections(T M \vert_{\partial X}) \, \text{,} \\
      (s, t) &\mapsto (\widehat{\exp}_{\pi_\nu(s)})^\ast((\widehat{\exp}_{\sigma(t)})^\ast(\hat{\gamma})) \, \text{,}
    \end{split} \label{eq:def-H-hat-1}
  \end{equation}
  where $\pi_\nu$ is defined in \eqref{eq:def-pi-nu}, let
  \begin{equation}
    \begin{split}
      \hat{H}_2 \colon \sections(U_{\partial X}) \oplus \sections(V^\prime) &\to \forms^1(W, T M \vert_W) \, \text{,} \\
      (s, t) &\mapsto (\widehat{\exp}_{\hat{\sigma}(\pi_\nu(s))})^\ast((\widehat{\exp}_{\sigma(t)})^\ast(\mu)) \, \text{,}
    \end{split} \label{eq:def-H-hat-2}
  \end{equation}
  and let
  \begin{equation}
    \begin{split}
      \hat{H} \colon \sections(U_{\partial X}) \oplus \sections(V^\prime) &\to \sections(\normal[W]{\partial X}) \, \text{,} \\
      (s, t) &\mapsto \pi_\nu((g(\hat{H}_1(s, t), \hat{H}_2(s, t)))^\sharp) \, \text{.}
    \end{split} \label{eq:def-H-hat}
  \end{equation}
  Here we view $\hat{H}_2(s, t)$ as an element of $\sections(T^\ast W \vert_{\partial X} \otimes T M \vert_{\partial X})$ and use the inner product with $\hat{H}_1(s, t)$ on the $T M \vert_{\partial X}$\dash-factor to get an element in $\sections(T^\ast W \vert_{\partial X})$, which we identify with an element of $\sections(T W \vert_{\partial X})$ using the metric isomorphism~$\sharp$. Then
  \begin{equation}
    (\D \hat{H})_{(0, 0)}(s, 0) = (\D H)_0(s) \, \text{.} \label{eq:linearisation-H-hat}
  \end{equation}
  This follows similarly to the proofs of \hyperref[lemma:linearisation-B]{Lemmas~\ref*{lemma:linearisation-B}} and \ref{lemma:linearisation-B-hat} below. Let
  \begin{equation}
    \begin{split}
      \hat{B} \colon \sections(U) \oplus \sections(V^\prime) &\to \sections(\normal[W]{\partial X}) \, \text{,} \\
      (s, t) &\mapsto \pi_\nu(\rho^{-1}(\hat{F}(s, t) \vert_{\partial X})) + \hat{H}(s \vert_{\partial X}, t) \, \text{.}
    \end{split} \label{eq:def-B-hat}
  \end{equation}
  
  \begin{lemma} \label{lemma:linearisation-B-hat}
    Let $t \in \sections(\normal[M]{W})$. Then
    \begin{equation}
      (\D \hat{B})_{(0, 0)}(0, t) = \pi_\nu((g(\nabla t, u))^\sharp) \, \text{,} \label{eq:linearisation-B-hat}
    \end{equation}
    where $u \in \sections(\normal[X]{\partial X})$ is the inward-pointing unit normal vector field of~$\partial X$ in~$X$. Here we view $\nabla t$ as an element of $\sections(T^\ast W \vert_{\partial X} \otimes T M \vert_{\partial X})$ and use the inner product with $u$ on the $T M \vert_{\partial X}$\dash-factor to get an element in $\sections(T^\ast W \vert_{\partial X})$, which we identify with an element of $\sections(T W \vert_{\partial X})$ using the metric isomorphism~$\sharp$.
  \end{lemma}
  
  \begin{proof}
    Let $k$ be the dimension of~$W$, let $(e_1, \dotsc, e_8)$ be a local orthonormal frame of~$M$ such that $(e_1, \dotsc, e_4)$ is a positive (i.e., $e_1 = e_2 \times e_3 \times e_4$) frame of~$X$ with $e_1 = u$ and $(e_2, \dotsc, e_{k + 1})$ is a frame of~$W$, and let $(e^1, \dotsc, e^8)$ be the dual coframe. Then
    \begin{equation}
      \hat{\gamma} = \sum_{i = 1}^8 (e_i \interior \Phi) \otimes e_i \quad \text{and} \quad \mu = \sum_{i = 1}^8 e^i \otimes e_i \, \text{.}
    \end{equation}
    Now
    \begin{align*}
      &(\D \hat{H})_{(0, 0)}(0, t) \\
      &= \pi_\nu((g((\D \hat{H}_1)_{(0, 0)}(0, t), \hat{H}_2(0, 0)))^\sharp) + \pi_\nu((g(\hat{H}_1(0, 0), (\D \hat{H}_2)_{(0, 0)}(0, t)))^\sharp) \\
      &= \pi_\nu((\D \hat{H}_1)_{(0, 0)}(0, t)) + \pi_\nu((g((\D \hat{H}_2)_{(0, 0)}(0, t), u))^\sharp)
    \end{align*}
    since $(g(v, \hat{H}_2(0, 0)))^\sharp = v$ for $v \in \sections(T M \vert_W)$ and $\hat{H}_1(0, 0) = u$ as $(v \interior \Phi) \vert_{\partial X} = g(v, u) \vol_{\partial X}$ for $v \in \sections(T M)$. We have
    \begin{align*}
      (\D \hat{H}_1)_{(0, 0)}(0, t) &= \frac{\D}{\D r} \hat{H}_1(0, r t) \bigg\vert_{r = 0} = \frac{\D}{\D r} (\widehat{\exp}_{\sigma(r t)})^\ast(\hat{\gamma}) \bigg\vert_{r = 0} = (\Lie_{\sigma(t)} \hat{\gamma}) \vert_{\partial X} \\
      &= \sum_{i = 1}^8 (\Lie_{\sigma(t)} (e_i \interior \Phi)) \vert_{\partial X} \otimes e_i + \sum_{i = 1}^8 (e_i \interior \Phi) \vert_{\partial X} \otimes \nabla_{\sigma(t)} e_i \\
      &= \vol_{\partial X} \otimes \Biggl(\sum_{i = 1}^8 (\Lie_{\sigma(t)} (e_i \interior \Phi))(e_2, e_3, e_4) e_i + \nabla_{\sigma(t)} e_1\Biggr)
    \end{align*}
    since $\Phi(e_i, e_2, e_3, e_4) = g(e_i, e_2 \times e_3 \times e_4) = g(e_i, e_1)$ for $i = 1, \dotsc, 8$. So similarly to the proof of \lemmaref{lemma:linearisation-B}, we get
    \begin{align*}
      &\pi_\nu((\D \hat{H}_1)_{(0, 0)}(0, t)) \\
      &= - \pi_\nu(P(\sigma(t) \vert_{\partial X})^\perp) + \sum_{i = 5}^{k + 1} g(\nabla_{\sigma(t)} e_i, e_1) e_i + \nabla_{\sigma(t)} e_1 \\
      &= - \pi_\nu(\rho^{-1}((D(\sigma(t) \vert_X)^\perp) \vert_{\partial X})) + \sum_{i = 5}^{k + 1} (g(\nabla_{\sigma(t)} e_i, e_1) + g(e_i, \nabla_{\sigma(t)} e_1)) e_i \\
      &= - (\D \hat{F})_{(0, 0)}(0, t) \vert_{\partial X}
    \end{align*}
    since $(\nabla_{e_1} \sigma(t)) \vert_{\partial X} = 0$ by condition~\eqref{lemma:extension:nabla} in \lemmaref{lemma:extension}. Further,
    \begin{align*}
      (\D \hat{H}_2)_{(0, 0)}(0, t) &= \frac{\D}{\D r} \hat{H}_2(0, r t) \bigg\vert_{r = 0} = \frac{\D}{\D r} (\widehat{\exp}_{\sigma(t)})^\ast(\mu) \bigg\vert_{r = 0} = \Lie_{\sigma(t)} \mu \\
      &= \sum_{i = 1}^8 \Lie_{\sigma(t)} e^i \otimes e_i + \sum_{i = 1}^8 e^i \otimes \nabla_{\sigma(t)} e_i \\
      &= - \sum_{i, j = 1}^8 g(\Lie_{\sigma(t)} e_i, e_j) e^i \otimes e_j + \sum_{i, j = 1}^8 g(\nabla_{\sigma(t)} e_i, e_j) e^i \otimes e_j \\
      &= \sum_{i, j = 1}^8 g(\nabla_{e_i} \sigma(t), e_j) e^i \otimes e_j \\
      &= \sum_{i = 1}^8 e^i \otimes \nabla_{e_i} \sigma(t) \\
      &= \nabla \sigma(t) \, \text{.}
    \end{align*}
    So together we get
    \begin{equation*}
      (\D \hat{B})_{(0, 0)}(0, t) = (\D \hat{H})_{(0, 0)}(0, t) + (\D \hat{F})_{(0, 0)}(0, t) = \pi_\nu((g(\nabla t, u))^\sharp) \, \text{.} \qedhere
    \end{equation*}
  \end{proof}
  
  \begin{lemma} \label{lemma:scaffold-surjective}
    The map
    \begin{equation}
      \begin{split}
        \sections(\normal[M]{X}) \oplus \sections(\normal[M]{W}) &\to \sections(\normal[M]{X}) \oplus \sections(K) \oplus \sections(\normal[W]{\partial X}) \, \text{,} \\
        (s, t) &\mapsto ((\D \hat{G})_{(0, 0)}(s, t), \pi_K(s \vert_{\partial X}), (\D \hat{B})_{(0, 0)}(s, t))
      \end{split}
    \end{equation}
    is surjective, where the vector bundle~$K$ over~$\partial X$ is defined as in the proof of \propositionref{prop:main-proposition-boundary} and $\pi_K$ is defined in~\eqref{eq:def-pi-K}.
  \end{lemma}
  
  \begin{proof}
    We have seen that this map is given by
    \begin{equation}
      \begin{split}
        (s, t) &\mapsto (D^\ast D s + D^\ast D (\sigma(t) \vert_X)^\perp, \pi_K(s \vert_{\partial X}), \\
        &\phantom{{}\mapsto (} (\D B)_0(s \vert_{\partial X}) + \pi_\nu((g(\nabla t, u))^\sharp))
      \end{split}
    \end{equation}
    by \eqref{eq:linearisation-G}, \eqref{eq:linearisation-G-hat}, \eqref{eq:linearisation-H-hat}, and \eqref{eq:linearisation-B-hat} since $\pi_K$ is linear.
    
    Let $(t^\prime, k, r) \in \sections(\normal[M]{X}) \oplus \sections(K) \oplus \sections(\normal[W]{\partial X})$. Then there is some $s \in \sections(\normal[M]{X})$ such that
    \begin{equation*}
      D^\ast D s = t^\prime \quad \text{and} \quad \pi_K(s \vert_{\partial X}) = k
    \end{equation*}
    since
    \begin{equation*}
      \sections(\normal[M]{X}) \to \sections(\normal[M]{X}) \oplus \sections(\normal[M]{X} \vert_{\partial X}) \, \text{,} \quad s \mapsto (D^\ast D s, s \vert_{\partial X})
    \end{equation*}
    is surjective as the boundary problem
    \begin{equation*}
      \begin{cases}
        D^\ast D s = 0 &\text{in~$X$,} \\
        \phantom{D^\ast D s = 0} \mathllap{s \vert_{\partial X} = 0} &\text{on~$\partial X$} \\
      \end{cases}
    \end{equation*}
    has index~$0$ (compare the proof of \lemmaref{lemma:boundary-second-index}) and $0$\dash-dimensional kernel by the Unique Continuation Property \cite[Corollary~8.3 and Remark~12.2]{BW93} (compare the discussion of $\dim W = 3$ in \sectionref{subsec:remarks-dimension-scaffold}). Furthermore, there is some $t \in \sections(\normal[M]{W})$ such that
    \begin{equation*}
      t \vert_{\partial X} = 0 \quad \text{and} \quad \pi_\nu((g(\nabla t, u))^\sharp) = r - (\D B)_0(s \vert_{\partial X})
    \end{equation*}
    since
    \begin{equation*}
      \sections(\normal[M]{W}) \to \sections(\normal[M]{W} \vert_{\partial X}) \oplus \sections(\normal[W]{\partial X}) \, \text{,} \quad t \mapsto (t \vert_{\partial X}, \pi_\nu((g(\nabla t, u))^\sharp))
    \end{equation*}
    is surjective. Note that $\sigma(t) \vert_X = 0$ by condition~\eqref{lemma:extension:partial-X-0} in \lemmaref{lemma:extension}. So
    \begin{equation*}
      (\D \hat{G})_{(0, 0)}(s, t) = t^\prime \, \text{,} \quad \pi_K(s \vert_{\partial X}) = k \, \text{,} \quad \text{and} \quad (\D \hat{B})_{(0, 0)}(s, t) = r \, \text{.} \qedhere
    \end{equation*}
  \end{proof}
  
  The coefficients of~$D$ are in~$\CC^{k, \alpha}$ by~\eqref{eq:def-D} since $\Phi \in \CC^{k + 1, \alpha}(\altforms^4 M)$. The first-order coefficients of~$D$ are even in~$\CC^{k + 1, \alpha}$. Hence the coefficients of~$D^\ast$ are in~$\CC^{k, \alpha}$. So the coefficients of~$D^\ast D$ are in~$\CC^{k - 1, \alpha}$. Also the coefficients of~$P$ and $P^\ast$ are in~$\CC^{k, \alpha}$ by~\eqref{eq:relation-D-P}. Further, the coefficients of~$\pi_K$ and $\pi_\nu$ are in~$\CC^{k + 1, \alpha}$.
  
  Now $\hat{G}$ extends to a map
  \begin{equation}
    \hat{G}_{k, \alpha} \colon \CC^{2, \alpha}(U) \oplus \CC^{k, \alpha}(V^\prime) \to \CC^{0, \alpha}(\normal[M]{X})
  \end{equation}
  of class~$\CC^k$ by \propositionref{prop:c-ell} since $\Phi \in \CC^{k + 1, \alpha}(\altforms^4 M)$ and since $D^\ast \colon \CC^{1, \alpha}(E) \to \CC^{0, \alpha}(\normal[M]{X})$ is a linear first-order differential operator. Furthermore, $\hat{B}$ extends to a map
  \begin{equation}
    \hat{B}_{k, \alpha} \colon \CC^{2, \alpha}(U) \oplus \CC^{k, \alpha}(V^\prime) \to \CC^{1, \alpha}(\normal[W]{\partial X})
  \end{equation}
  of class~$\CC^k$.
  
  The proof of \lemmaref{lemma:scaffold-surjective} shows that the map
  \begin{equation*}
    \begin{split}
      \CC^{k, \alpha}(\normal[M]{X}) \oplus \CC^{k, \alpha}(\normal[M]{W}) &\to \CC^{k - 2, \alpha}(\normal[M]{X}) \oplus \CC^{k, \alpha}(K) \oplus \CC^{k - 1, \alpha}(\normal[W]{\partial X}) \, \text{,} \\
      (s, t) &\mapsto ((\D \hat{G}_{k, \alpha})_{(0, 0)}(s, t), \pi_K(s \vert_{\partial X}), (\D \hat{B}_{k, \alpha})_{(0, 0)}(s, t))
    \end{split}
  \end{equation*}
  is surjective since the maps
  \begin{equation*}
    \CC^{k, \alpha}(\normal[M]{X}) \to \CC^{k - 2, \alpha}(\normal[M]{X}) \oplus \CC^{k, \alpha}(\normal[M]{X} \vert_{\partial X}) \, \text{,} \quad s \mapsto (D^\ast D s, s \vert_{\partial X})
  \end{equation*}
  and
  \begin{equation*}
    \CC^{k, \alpha}(\normal[M]{W}) \to \CC^{k, \alpha}(\normal[M]{W} \vert_{\partial X}) \oplus \CC^{k - 1, \alpha}(\normal[W]{\partial X}) \, \text{,} \quad t \mapsto (t \vert_{\partial X}, \pi_\nu((g(\nabla t, u))^\sharp))
  \end{equation*}
  are surjective.
  
  If $(t^\prime, k, r) \in \CC^{2, \alpha}(\normal[M]{X}) \oplus \CC^{1, \alpha}(K) \oplus \CC^{2, \alpha}(\normal[W]{\partial X})$ is in the $L^2$\dash-orthogonal complement of the image of the linearisation of~\eqref{eq:boundary-second} at the $0$\dash-section, then $t^\prime \in \CC^{k + 1, \alpha}(\normal[M]{X})$, $k \in \CC^{k, \alpha}(K)$, and $r \in \CC^{k + 1, \alpha}(\normal[W]{\partial X})$. This follows from the proof of \theoremref{thm:main-theorem-spin7-boundary}. Thus the map
  \begin{equation*}
    \begin{split}
      \CC^{2, \alpha}(\normal[M]{X}) \oplus \CC^{k, \alpha}(\normal[M]{W}) &\to \CC^{0, \alpha}(\normal[M]{X}) \oplus \CC^{2, \alpha}(K) \oplus \CC^{1, \alpha}(\normal[W]{\partial X}) \, \text{,} \\
      (s, t^\prime) &\mapsto ((\D \hat{G}_{k, \alpha})_{(0, 0)}(s, t^\prime), \pi_K(s \vert_{\partial X}), (\D \hat{B}_{k, \alpha})_{(0, 0)}(s, t^\prime))
    \end{split}
  \end{equation*}
  is surjective. Hence there is a ($\CC^{2, \alpha} \oplus \CC^{k, \alpha}$)\dash-neighbourhood $\tilde{U}_1 \subseteq \CC^{2, \alpha}(U) \oplus \CC^{k, \alpha}(V^\prime)$ of $(0, 0)$ such that $((\D \hat{G}_{k, \alpha})_{(s, t)}, \pi_K, (\D \hat{B}_{k, \alpha})_{(s, t)})$ is surjective for all $(s, t) \in \tilde{U}_1$.
  
  For fixed $t \in \CC^{k, \alpha}(V^\prime)$, the equation
  \begin{equation}
    \hat{G}_{k, \alpha}(s, t) = D^\ast(\pi_E((\widehat{\exp}_s)^\ast((\widehat{\exp}_{\sigma(t)})^\ast(\tau)))) = 0
  \end{equation}
  is a nonlinear partial differential equation of order~$2$ in~$s$ and the equation
  \begin{equation}
    \begin{split}
      &\hat{B}_{k, \alpha}(s, t) \\
      &= \pi_\nu(\rho^{-1}(\pi_E((\widehat{\exp}_s)^\ast((\widehat{\exp}_{\sigma(t)})^\ast(\tau))))) \\
      &\phantom{{}={}} {}+ \pi_\nu((g((\widehat{\exp}_{\pi_\nu(s)})^\ast((\widehat{\exp}_{\sigma(t)})^\ast(\hat{\gamma})), (\widehat{\exp}_{\hat{\sigma}(\pi_\nu(s))})^\ast((\widehat{\exp}_{\sigma(t)})^\ast(\mu))))^\sharp) \\
      &= 0
    \end{split}
  \end{equation}
  is a nonlinear partial differential equation of order~$1$ in~$s$ on the boundary. Since the linearisation of
  \begin{equation}
    \begin{cases}
      \phantom{\pi_K(s \vert_{\partial X}) = 0} \mathllap{\hat{G}_{k, \alpha}(s, t) = 0} &\text{in~$X$,} \\
      \pi_K(s \vert_{\partial X}) = 0 &\text{on~$\partial X$,} \\
      \phantom{\pi_K(s \vert_{\partial X}) = 0} \mathllap{\hat{B}_{k, \alpha}(s, t) = 0} &\text{on~$\partial X$}
    \end{cases} \label{eq:boundary-second-t}
  \end{equation}
  at~$0$ is elliptic for $t = 0$, there is a ($\CC^{2, \alpha} \oplus \CC^{k, \alpha}$)\dash-neighbourhood $\tilde{U}_2 \subseteq \CC^{2, \alpha}(U) \oplus \CC^{k, \alpha}(V^\prime)$ of $(0, 0)$ such that
  \begin{equation*}
    \begin{cases}
      \phantom{(\D \hat{B}_{k, \alpha})_{(s, t)}(s^\prime, 0) = 0} \mathllap{(\D \hat{G}_{k, \alpha})_{(s, t)}(s^\prime, 0) = 0} &\text{in~$X$,} \\
      \phantom{(\D \hat{B}_{k, \alpha})_{(s, t)}(s^\prime, 0) = 0} \mathllap{\pi_K(s^\prime \vert_{\partial X}) = 0} &\text{on~$\partial X$,} \\
      (\D \hat{B}_{k, \alpha})_{(s, t)}(s^\prime, 0) = 0 &\text{on~$\partial X$}
    \end{cases}
  \end{equation*}
  is an elliptic boundary problem in $s^\prime \in \CC^{2, \alpha}(\normal[M]{X})$ for all $(s, t) \in \tilde{U}_2$. In particular, if $(s, t) \in \tilde{U}_2$ satisfies~\eqref{eq:boundary-second-t}, then $s \in \CC^{k, \alpha}(\normal[M]{X})$ by Elliptic Regularity \cite[Theorem~6.8.2]{Mor66} since $(\widehat{\exp}_{\sigma(t)})^\ast(\tau) \in \CC^{k - 1, \alpha}(\altforms^4 M \otimes \altforms^2_7 M)$, $(\widehat{\exp}_{\sigma(t)})^\ast(\hat{\gamma}) \in \CC^{k - 1, \alpha}(\altforms^3 M \otimes T M)$, and $(\widehat{\exp}_{\sigma(t)})^\ast(\mu) \in \CC^{k - 1, \alpha}(T^\ast M \otimes T M)$.
  
  Let $\tilde{U} \defeq \tilde{U}_1 \cap \tilde{U}_2$. Then the above argumentation shows that if $(s, t) \in \tilde{U}$ satisfies~\eqref{eq:boundary-second-t}, then
  \begin{equation*}
    \begin{split}
      ((\D \hat{G}_{k, \alpha})_{(s, t)}, \pi_K, (\D \hat{B}_{k, \alpha})_{(s, t)}) \colon &\CC^{2, \alpha}(\normal[M]{X}) \oplus \CC^{k, \alpha}(\normal[M]{W}) \\
      &\to \CC^{0, \alpha}(\normal[M]{X}) \oplus \CC^{2, \alpha}(K) \oplus \CC^{1, \alpha}(\normal[W]{\partial X})
    \end{split}
  \end{equation*}
  is surjective and
  \begin{equation*}
    \begin{split}
      &((\D \hat{G}_{k, \alpha})_{(s, t)}, \pi_K, (\D \hat{B}_{k, \alpha})_{(s, t)}) \vert_{\CC^{2, \alpha}(\normal[M]{X})} \colon \\
      &\CC^{2, \alpha}(\normal[M]{X}) \to \CC^{0, \alpha}(\normal[M]{X}) \oplus \CC^{2, \alpha}(K) \oplus \CC^{1, \alpha}(\normal[W]{\partial X})
    \end{split}
  \end{equation*}
  is Fredholm. Furthermore, the Fredholm index is the same for all of these operators as we may assume \wolog\ that $\tilde{U}$ is connected. So the Fredholm index is~$0$ by \lemmaref{lemma:boundary-second-index}.
  
  So \theoremref{thm:genericity-general} implies that for every generic $t \in \CC^{k, \alpha}(V^\prime)$ that is $\CC^{k, \alpha}$\dash-close to~$0$, the moduli space of all solutions $s \in \CC^{k, \alpha}(\normal[M]{X})$ of the boundary problem \eqref{eq:boundary-second-t} that are $\CC^{2, \alpha}$\dash-close to~$0$ is either empty or a $\CC^k$\dash-manifold of dimension~$0$ (i.e., a discrete set). This discrete set is finite by compactness (\lemmaref{lemma:compactness}). The result follows since the boundary problem \eqref{eq:boundary-first} implies \eqref{eq:boundary-second} and a subset of a finite set is again a finite set.
\end{proof}

\subsection{Remark about Torsion-Free Spin(7)-Structures}
\label{subsec:remark-torsion-free-boundary}

In contrast to the case of closed Cayley submanifolds (see \sectionref{subsec:remark-torsion-free-closed}), we get the following genericity statement for torsion-free $\Spin(7)$-structures (rather than $\Spin(7)$-structures that are not necessarily torsion-free in \theoremref{thm:main-theorem-spin7-boundary}) in the case of a compact, connected Cayley submanifold with non-empty boundary. Note that there are small differences to \theoremref{thm:main-theorem-spin7-boundary} in terms of the regularity of the objects involved. We use the $\CC^{k, \alpha}$\dash-topology for the space of all torsion-free $\Spin(7)$\dash-structures.

\begin{theorem}
  Let $k \ge 2$, let $0 < \alpha < 1$, let $M$ be an $8$\dash-manifold of class~$\CC^{k + 1, \alpha}$ with a torsion-free $\Spin(7)$-structure~$\Phi$ of class~$\CC^{k + 1, \alpha}$, let $X$ be a compact, connected Cayley submanifold of~$M$ of class~$\CC^{k + 1, \alpha}$ with non-empty boundary, and let $W$ be a submanifold of~$M$ of class~$\CC^{k + 1, \alpha}$ with $\partial X \subseteq W$ such that $X$ and $W$ meet orthogonally.
  
  Then for every generic torsion-free $\Spin(7)$\dash-structure~$\Psi$ of class~$\CC^{k, \alpha}$ that is $\CC^{k, \alpha}$\dash-close to~$\Phi$, the moduli space of all Cayley submanifolds of $(M, \Psi)$ of class~$\CC^{k, \alpha}$ that are $\CC^{2, \alpha}$\dash-close to~$X$ with boundary on~$W$ and meeting~$W$ orthogonally (with respect to the metric induced by~$\Psi$) is a finite set (possibly empty).
\end{theorem}

\begin{proof}
  The following lemma replaces \lemmaref{lemma:linearisation-H-tilde} and \hyperref[eq:linearisation-B-tilde]{Equation~(\ref*{eq:linearisation-B-tilde})} in the proof of \theoremref{thm:main-theorem-spin7-boundary}.
  
  \begin{lemma}
    Let $M$ be an $8$\dash-manifold with a $\Spin(7)$\dash-structure~$\Phi$, let $X$ be a compact Cayley submanifold with boundary, and let $v \in \sections(T M)$. Then
    \begin{equation}
      (\D \tilde{B})_{(0, 0)}(0, \Lie_v \Phi) = \pi_\nu(\nabla_u v \vert_{\partial X}) + \pi_\nu((g(\nabla (v \vert_W), u))^\sharp) \, \text{,}
    \end{equation}
    where $\pi_\nu$ is defined in \eqref{eq:def-pi-nu} and $u \in \sections(\normal[X]{\partial X})$ is the inward-pointing unit normal vector field of~$\partial X$ in~$X$. Here we view $\nabla (v \vert_W)$ as an element of $\sections(T^\ast W \vert_{\partial X} \otimes T M \vert_{\partial X})$ and use the inner product with $u$ on the $T M \vert_{\partial X}$\dash-factor to get an element in $\sections(T^\ast W \vert_{\partial X})$, which we identify with an element of $\sections(T W \vert_{\partial X})$ using the metric isomorphism~$\sharp$.
  \end{lemma}
  
  \begin{proof}
    In the proof of \lemmaref{lemma:linearisation-H-tilde}, if we let the orthonormal frame $(e_1(t), \dotsc, \allowbreak e_8(t))$ depend on~$t$ such that $(e_1(t), \dotsc, e_4(t))$ is a frame of~$X$, $(e_2(t), \dotsc, e_4(t))$ is a frame of~$\partial X$, and $(e_2(t), \dotsc, e_{k + 1}(t))$ is a frame of~$W$, then
    \begin{align*}
      (\D \tilde{H})_{(0, 0)}(0, \chi) &= \frac{\D}{\D t} \sum_{i = 2}^{k + 1} (e_i(t) \interior \Phi_t) \vert_{\partial X} \otimes \pi_\nu(e_i(t)) \biggr\vert_{t = 0} \\
      &= \sum_{i = 2}^{k + 1} ((e_i)^\prime \interior \Phi) \vert_{\partial X} \otimes \pi_\nu(e_i) + \sum_{i = 2}^{k + 1} (e_i \interior \chi) \vert_{\partial X} \otimes \pi_\nu(e_i) \\
      &\phantom{{}={}} {}+ \sum_{i = 2}^{k + 1} (e_i \interior \Phi) \vert_{\partial X} \otimes \frac{\D}{\D t} \pi_\nu(e_i(t)) \biggr\vert_{t = 0} \\
      &= \sum_{i = 5}^{k + 1} (e_i \interior \chi) \vert_{\partial X} \otimes e_i
    \end{align*}
    since
    \begin{align*}
      ((e_i)^\prime \interior \Phi) \vert_{\partial X} &= g((e_i)^\prime, e_1) \vol_{\partial X} = 0 \quad \text{and} \\
      (e_i \interior \Phi) \vert_{\partial X} &= g(e_i, e_1) \vol_{\partial X} = 0
    \end{align*}
    for $i = 5, \dotsc, k + 1$ as $(e_i)^\prime$ is tangent to~$W$ because $e_i(t)$ is tangent to~$W$ for all~$t$. So
    \begin{align*}
      (\D \tilde{H})_{(0, 0)}(0, \Lie_v \Phi) &= \sum_{i = 5}^{k + 1} (e_i \interior \Lie_v \Phi) \vert_{\partial X} \otimes e_i \\
      &= \sum_{i = 5}^{k + 1} (\Lie_v (e_i \interior \Phi) - \Lie_v e_i \interior \Phi) \vert_{\partial X} \otimes e_i \, \text{.}
    \end{align*}
    The proof of \lemmaref{lemma:linearisation-B} shows that
    \begin{align*}
      \sum_{i = 5}^{k + 1} (\Lie_v (e_i \interior \Phi)) \vert_{\partial X} \otimes e_i &= - \pi_\nu(\rho^{-1}((D (v \vert_X)^\perp) \vert_{\partial X})) + \pi_\nu(\nabla_u v \vert_{\partial X}) \\
      &\phantom{{}={}} {}+ \sum_{i = 5}^{k + 1} g(\nabla_v e_i, e_1) \vol_{\partial X} \otimes e_i \, \text{.}
    \end{align*}
    Hence
    \begin{align*}
      (\D \tilde{B})_{(0, 0)}(0, \Lie_v \Phi) &= \pi_\nu(\rho^{-1}(((\D \tilde{F})_{(0, 0)}(0, \Lie_v \Phi)) \vert_{\partial X})) + (\D \tilde{H})_{(0, 0)}(0, \Lie_v \Phi) \\
      &= \pi_\nu(\nabla_u v \vert_{\partial X}) + \sum_{i = 5}^{k + 1} (g(\nabla_v e_i, e_1) - g(\Lie_v e_i, e_1)) \vol_{\partial X} \otimes e_i \\
      &= \pi_\nu(\nabla_u v \vert_{\partial X}) + \sum_{i = 5}^{k + 1} g(\nabla_{e_i} v, e_1) \vol_{\partial X} \otimes e_i \\
      &= \pi_\nu(\nabla_u v \vert_{\partial X}) + \pi_\nu((g(\nabla (v \vert_W), u))^\sharp)
    \end{align*}
    by \eqref{eq:linearisation-F-tilde-Lie}.
  \end{proof}
  
  So if $\mathcal{X}$ denotes the space of all torsion-free $\Spin(7)$-structures on~$M$, then
  \begin{equation*}
    \begin{split}
      \sections(\normal[M]{X}) \oplus \sections(T_\Phi \mathcal{X}) &\to \sections(\normal[M]{X}) \oplus \sections(K) \oplus \sections(\normal[W]{\partial X}) \\
      (s, \chi) &\mapsto ((\D \tilde{G})_{(0, 0)}(s, \chi), \pi_K(s \vert_{\partial X}), (\D \tilde{B})_{(0, 0)}(s, \chi))
    \end{split}
  \end{equation*}
  is surjective by the proof of \lemmaref{lemma:scaffold-surjective}, and hence
  \begin{equation*}
    \begin{split}
      \CC^{2, \alpha}(\normal[M]{X}) \oplus \CC^{k, \alpha}(T_\Phi \mathcal{X}) &\to \CC^{0, \alpha}(\normal[M]{X}) \oplus \CC^{2, \alpha}(K) \oplus \CC^{1, \alpha}(\normal[W]{\partial X}) \, \text{,} \\
      (s, \chi) &\mapsto ((\D \tilde{G}_{k, \alpha})_{(0, 0)}(s, \chi), \pi_K(s \vert_{\partial X}), (\D \tilde{B}_{k, \alpha})_{(0, 0)}(s, \chi))
    \end{split}
  \end{equation*}
  is surjective by similar arguments as in the proof of \theoremref{thm:main-theorem-scaffold}. The rest of the proof is analogous to the proof of \theoremref{thm:main-theorem-spin7-boundary}.
\end{proof}

\begin{remark}
  A similar result holds in the case of associative submanifolds inside a $G_2$\dash-manifold (this can be proved by the methods of \cite{Gay14} although Gayet did not prove it): \emph{Let $M$ be a $7$\dash-manifold with a torsion-free $G_2$\dash-structure~$\varphi$, let $X$ be a compact, connected associative submanifold of~$M$ with non-empty boundary, and let $W$ be a coassociative submanifold of~$M$ with $\partial X \subseteq W$.}
  
  \emph{Then for every generic torsion-free $G_2$\dash-structure~$\psi$ that is close to~$\varphi$, the moduli space of all associative submanifolds of $(M, \psi)$ that are close to~$X$ with boundary on~$W$ is either empty or a smooth manifold of dimension equal to the index.}
\end{remark}

\subsection{Existence of Cayley Submanifolds for Nearby Spin(7)-Structures and Scaffolds}
\label{subsec:existence-cayley-nearby}

Here is an application of the deformation theory we have developed.

\begin{corollary} \label{cor:existence-cayley-nearby}
  Let $M$ be an $8$\dash-manifold with a $\Spin(7)$-structure~$\Phi$, let $X$ be a compact, connected Cayley submanifold of~$M$ with non-empty boundary, let $W$ be a submanifold of~$M$ with $\partial X \subseteq W$ such that $X$ and $W$ meet orthogonally, and let $0 < \alpha < 1$. Suppose that the kernel of the linearisation of the boundary problem~\eqref{eq:boundary-second} at the $0$\dash-section is $0$\dash-dimensional (so we are in a generic situation).
  
  Then for every $\Spin(7)$\dash-structure~$\Psi$ that is $\CC^{2, \alpha}$\dash-close to~$\Phi$ and for every local deformation~$W^\prime$ of~$W$ that is $\CC^{2, \alpha}$\dash-close to~$W$, there exists a unique Cayley submanifold of $(M, \Psi)$ that is $\CC^{2, \alpha}$\dash-close to~$X$ with boundary on~$W^\prime$ and meeting~$W^\prime$ orthogonally (with respect to the metric induced by~$\Psi$).
\end{corollary}

Note that only the initial $\Spin(7)$-structure~$\Phi$ and scaffold~$W$ are assumed to be generic, not the deformations $\Psi$ and~$W^\prime$ (for which the corollary states that they are in fact generic).

\begin{proof}
  Define
  \begin{align*}
    \bar{F} \colon \sections(U) \oplus \sections(V) \oplus \sections(V^\prime) &\to \sections(E) \, \text{,} \\
    (s, \chi, t) &\mapsto \pi_E((\widehat{\exp}_s)^\ast((\widehat{\exp}_{\sigma(t)})^\ast(\tau_{\Theta(\chi)}))) \, \text{,} \\
    \bar{G} \colon \sections(U) \oplus \sections(V) \oplus \sections(V^\prime) &\to \sections(\normal[M]{X}) \, \text{,} \\
    (s, \chi, t) &\mapsto D^\ast(\bar{F}(s, \chi, t)) \, \text{,} \\
    \bar{H}_1 \colon \sections(U_{\partial X}) \oplus \sections(V) \oplus \sections(V^\prime) &\to \forms^3(\partial X, T M \vert_{\partial X}) \cong \sections(T M \vert_{\partial X}) \, \text{,} \\
    (s, \chi, t) &\mapsto (\widehat{\exp}_{\pi_\nu(s)})^\ast((\widehat{\exp}_{\sigma(t)})^\ast(\bar{\gamma}_{\Theta(\chi)})) \, \text{,} \\
    \bar{H}_2 \colon \sections(U_{\partial X}) \oplus \sections(V^\prime) &\to \forms^1(W, T M \vert_W) \, \text{,} \\
    (s, t) &\mapsto (\widehat{\exp}_{\hat{\sigma}(\pi_\nu(s))})^\ast((\widehat{\exp}_{\sigma(t)})^\ast(\mu)) \, \text{,} \\
    \bar{H} \colon \sections(U_{\partial X}) \oplus \sections(V) \oplus \sections(V^\prime) &\to \sections(\normal[W]{\partial X}) \, \text{,} \\
    (s, \chi, t) &\mapsto \pi_\nu((g_{\Theta(\chi)}(\bar{H}_1(s, \chi, t), \bar{H}_2(s, t)))^\sharp) \, \text{,} \\
    \bar{B} \colon \sections(U) \oplus \sections(V) \oplus \sections(V^\prime) &\to \sections(\normal[W]{\partial X}) \, \text{,} \\
    (s, \chi, t) &\mapsto \pi_\nu(\rho^{-1}(\bar{F}(s, \chi, t) \vert_{\partial X})) + \bar{H}(s \vert_{\partial X}, \chi, t)
  \end{align*}
  as in the proofs of \hyperref[thm:main-theorem-spin7-boundary]{Theorems~\ref*{thm:main-theorem-spin7-boundary}} and~\ref{thm:main-theorem-scaffold}. Then
  \begin{equation*}
    \begin{split}
      \sections(\normal[M]{X}) &\to \sections(\normal[M]{X}) \oplus \sections(K) \oplus \sections(\normal[W]{\partial X}) \, \text{,} \\
      s &\mapsto ((\D \bar{G})_{(0, 0, 0)}(s, 0, 0), \pi_K(s \vert_{\partial X}), (\D \bar{B})_{(0, 0, 0)}(s, 0, 0))
    \end{split}
  \end{equation*}
  is bijective by assumption (it is the linearisation of the boundary problem~\eqref{eq:boundary-second} at the $0$\dash-section). Now $\bar{G}$ and $\bar{B}$ extend to maps
  \begin{align*}
    \bar{G}_{2, \alpha} \colon \CC^{2, \alpha}(U) \oplus \CC^{2, \alpha}(V) \oplus \CC^{2, \alpha}(V^\prime) &\to \CC^{0, \alpha}(\normal[M]{X}) \quad \text{and} \\
    \bar{B}_{2, \alpha} \colon \CC^{2, \alpha}(U) \oplus \CC^{2, \alpha}(V) \oplus \CC^{2, \alpha}(V^\prime) &\to \CC^{1, \alpha}(\normal[W]{\partial X})
  \end{align*}
  of class~$\CC^1$ by \propositionref{prop:c-ell}. Since the linearisation of the boundary problem~\eqref{eq:boundary-second} at the $0$\dash-section is an elliptic boundary problem (\lemmaref{lemma:boundary-second-elliptic}), also
  \begin{equation*}
    \begin{split}
      \CC^{2, \alpha}(\normal[M]{X}) &\to \CC^{0, \alpha}(\normal[M]{X}) \oplus \CC^{2, \alpha}(K) \oplus \CC^{1, \alpha}(\normal[W]{\partial X}) \, \text{,} \\
      s &\mapsto ((\D \bar{G}_{2, \alpha})_{(0, 0, 0)}(s, 0, 0), \pi_K(s \vert_{\partial X}), (\D \bar{B}_{2, \alpha})_{(0, 0, 0)}(s, 0, 0))
    \end{split}
  \end{equation*}
  is bijective. So the Implicit Function Theorem implies that there is a $\CC^1$\dash-map
  \begin{equation*}
    L \colon \CC^{2, \alpha}(V) \oplus \CC^{2, \alpha}(V^\prime) \to \CC^{2, \alpha}(U)
  \end{equation*}
  such that
  \begin{equation*}
    (\bar{G}_{2, \alpha}(s, \chi, t), \pi_K(s \vert_{\partial X}), \bar{B}_{2, \alpha}(s, \chi, t)) = (0, 0, 0) \quad \text{if and only if} \quad s = L(\chi, t)
  \end{equation*}
  for all $(s, \chi, t) \in \CC^{2, \alpha}(U) \oplus \CC^{2, \alpha}(V) \oplus \CC^{2, \alpha}(V^\prime)$ that are $\CC^{2, \alpha}$\dash-close to~$(0, 0, 0)$. Note that $s$ is smooth if both $\chi$ and $t$ are smooth by Elliptic Regularity.
\end{proof}

\section{Examples}
\label{sec:examples}

In this section we present and discuss some examples for the deformation theory of compact Cayley submanifolds with boundary. We first show versions of the volume minimising property for this class of submanifolds in \sectionref{subsec:volume-minimising}. In \sectionref{subsec:example-holonomy-su2-su2} we discuss an example with a smooth $k$\dash-dimensional moduli space (for $0 \le k \le 4$) inside a manifold with holonomy $\SU(2) \times \SU(2)$. Then we construct a rigid Cayley submanifold of a manifold with holonomy $\Spin(7)$ in \sectionref{subsec:examples-bryant-salamon}. We further relate the deformation theory of compact Cayley submanifolds with boundary to the deformation theories of compact special Lagrangian, coassociative, and associative submanifolds with boundary in \hyperref[subsec:special-lagrangian-boundary]{Sections~\ref*{subsec:special-lagrangian-boundary}}--\ref{subsec:associative-boundary}. Throughout this section, any $\Spin(7)$\dash-structure will be assumed to be torsion-free.

\subsection{Volume Minimising Property}
\label{subsec:volume-minimising}

Harvey and Lawson \cite{HL82} showed that closed calibrated submanifolds are volume-minimising in their homology class. This applies, in particular, to Cayley submanifolds of a $\Spin(7)$\dash-manifold. Note that the condition that the $\Spin(7)$\dash-structure is torsion-free is crucial for this. Gayet \cite{Gay14} proved that a compact associative submanifold of a $G_2$\dash-manifold with boundary in a coassociative submanifold is volume-minimising in its relative homology class. Here we show versions of the volume minimising property for Cayley submanifolds with boundary. There are different versions for different dimensions of the scaffold~$W$: \propositionref{prop:minimal-boundary-dim-4} for $\dim W = 4$, \propositionref{prop:minimal-boundary-dim-5} for $\dim W = 5$, and \propositionref{prop:minimal-boundary-dim-6} for $\dim W = 6$.

With the same arguments as in \cite{Gay14}, one can prove the following proposition.

\begin{proposition} \label{prop:minimal-boundary-dim-4}
  Let $M$ be an $8$\dash-manifold with a torsion-free $\Spin(7)$\dash-structure~$\Phi$, let $X$ be a compact Cayley submanifold of~$M$ with boundary, let $W$ be a $4$\dash-dimensional submanifold of~$M$ with $\partial X \subseteq W$ such that $\Phi \vert_W = 0$ (which implies that $X$ and $W$ meet orthogonally), and let $Y$ be a compact, oriented $4$\dash-dimensional submanifold of~$M$ with $\partial Y \subseteq W$ which lies in the same relative homology class as $X$ in $H_4(M, W)$.
  
  Then the volume of~$Y$ is greater than or equal to the volume of~$X$, and equality holds if and only if $Y$ is Cayley.
\end{proposition}

The condition $\Phi \vert_W = 0$ forces $W$ to have dimension at most~$4$ since every $5$\dash-dimensional subspace of $(\R^8, \Phi_0)$ contains a (unique) Cayley subspace.

\begin{proof}
  There are a $5$\dash-chain~$S$ in~$M$ and a $4$\dash-chain~$Z$ in~$W$ such that $X - Y = \partial S + Z$. So
  \begin{equation*}
    \int_X \Phi - \int_Y \Phi = \int_{\partial S} \Phi + \int_Z \Phi = \int_S \D \Phi + \int_Z \Phi = 0
  \end{equation*}
  by Stokes' Theorem since $\D \Phi = 0$ as $\Phi$ is torsion-free and $\Phi \vert_Z = 0$ as $\Phi \vert_W = 0$ and $Z \subseteq W$. Hence
  \begin{equation*}
    \vol(X) = \int_X \vol_X = \int_X \Phi = \int_Y \Phi \le \int_Y \vol_Y = \vol(Y)
  \end{equation*}
  since $X$ is Cayley, and equality holds if and only if $\Phi \vert_Y = \vol_Y$, that is, if and only if $Y$ is Cayley.
\end{proof}

As a preparation for dimension~$5$, we have the following lemma.

\begin{lemma} \label{lemma:cayley-orthogonal-dim-5}
  Let $M$ be an $8$\dash-manifold with a $\Spin(7)$-structure~$\Phi$, let $X$ be a Cayley submanifold of~$M$ with boundary, let $W$ be an oriented $5$\dash-dimensional submanifold of~$M$ with $\partial X \subseteq W$, and let $n \defeq (\mathord\ast_W (\Phi \vert_W))^\sharp$. Then $X$ and $W$ meet orthogonally if and only if $n \vert_{\partial X} \in \sections(T \partial X)$.
\end{lemma}

\begin{proof}
  Let $u \in \sections(\normal[X]{\partial X})$ be the inward-pointing unit normal vector field of~$\partial X$ in~$X$, let $x \in \partial X$, let $a, b \in T_x \partial X$ be orthonormal and orthogonal to~$n_x$, let $V$ be the subspace of~$T_x M$ consisting of all vectors that are orthogonal to~$a$ and~$b$, and let $Z$ be the subspace of $V \cap T_x W$ consisting of all vectors that are orthogonal to~$n_x$ (note that $Z$ is $2$\dash-dimensional). Then $V \to V$, $v \mapsto a \times b \times v$ is orthogonal. Furthermore, $Z$~is invariant under this map since the orthogonal complement of~$n_x$ in $T_x W$ is a Cayley subspace. Hence $u_x \in Z^\perp$ if and only if $a \times b \times u_x \in Z^\perp$.
  
  Now suppose that $u_x$ is orthogonal to $T_x W$. Then $u_x \in Z^\perp$ since $Z \subseteq T_x W$. So $a \times b \times u_x \in Z^\perp$. But $a \times b \times u_x \in T_x \partial X \subseteq T_x W$ since $X$ is Cayley. Hence $a \times b \times u_x = \pm n_x$ since $a \times b \times u_x \perp a, b$. So $n_x \in T_x \partial X$.
  
  Conversely, suppose that $n_x \in T_x \partial X$. Then $n_x = \pm a \times b \times u_x$ since $X$ is Cayley. So $u_x \in Z^\perp$ since $n_x \in Z^\perp$. But also $u_x \perp n_x$ since $u_x = \mp a \times b \times n_x$. Hence $u_x$ is orthogonal to $T_x W$ since $u_x \perp a, b$.
\end{proof}

\begin{proposition} \label{prop:minimal-boundary-dim-5}
  Let $M$ be an $8$\dash-manifold with a torsion-free $\Spin(7)$\dash-structure~$\Phi$, let $X$ be a compact Cayley submanifold of~$M$ with boundary, let $W$ be an oriented $5$\dash-dimensional submanifold of~$M$ with $\partial X \subseteq W$ such that $X$ and $W$ meet orthogonally, and let $n \defeq (\mathord\ast_W (\Phi \vert_W))^\sharp$ (note that $n \vert_{\partial X} \in \sections(T \partial X)$ by \lemmaref{lemma:cayley-orthogonal-dim-5}). Suppose that $H^4(W) = 0$ and that $n$ is parallel with respect to the Levi-Civita connection of~$W$. Furthermore, let $Y$ be a compact, oriented $4$\dash-dimensional submanifold of~$M$ with $\partial Y \subseteq W$ which lies in the same relative homology class as $X$ in $H_4(M, W)$ such that $\partial Y$ is a local deformation of~$\partial X$ with $n \vert_{\partial Y} \in \sections(T \partial Y)$.
  
  Then the volume of~$Y$ is greater than or equal to the volume of~$X$, and equality holds if and only if $Y$ is Cayley.
\end{proposition}

Later, we will apply this proposition and \propositionref{prop:minimal-boundary-dim-6} below in the situation that $Y$ is a local deformation of~$X$. In that case, we can restrict $W$ to an open tubular neighbourhood~$W^\prime$ of~$\partial X$ in~$W$. Then $H^4(W^\prime) = 0$ as $\partial X$ is a deformation retract of~$W^\prime$ and $\partial X$ is $3$\dash-dimensional.

\begin{remark}
  The orientability of~$W$ is not necessary but the vector field~$n$ may not be well-defined if $W$ is not orientable. However, all conditions involving~$n$ are essentially local, and locally $W$ is orientable.
\end{remark}

\begin{proof}
  There are a $5$\dash-chain~$S$ in~$M$ and a $4$\dash-chain~$Z$ in~$W$ such that $X - Y = \partial S + Z$. So
  \begin{equation*}
    \int_X \Phi - \int_Y \Phi = \int_{\partial S} \Phi + \int_Z \Phi = \int_S \D \Phi + \int_Z \Phi = \int_Z \Phi
  \end{equation*}
  by Stokes' Theorem since $\D \Phi = 0$ as $\Phi$ is torsion-free. We have $\Phi \vert_W = \D \Psi$ for some $\Psi \in \forms^3(W)$ since $H^4(W) = 0$. Hence if $Z^\prime$ is another $4$\dash-chain in~$W$ such that $\partial Z^\prime = \partial Z = \partial X - \partial Y$, then
  \begin{equation*}
    \int_{Z^\prime} \Phi = \int_{Z^\prime} \D \Psi = \int_{\partial Z^\prime} \Psi = \int_{\partial Z} \Psi = \int_Z \D \Psi = \int_Z \Phi
  \end{equation*}
  by Stokes' Theorem.
  
  Write $\partial Y = (\exp_s)(\partial X)$ for some $s \in \sections(\normal[W]{\partial X})$ with small $\norm{s}_{\CC^0}$, where $\exp$ is the exponential map of~$W$. Let $x \in \partial X$, let $\eps > 0$ be small, and let
  \begin{equation*}
    f \colon [0, 1] \times [0, \eps] \to W \, \text{,} \quad (t, r) \mapsto \exp_{\exp_x(t s_x)}(r n(\exp_x(t s_x))) \, \text{.}
  \end{equation*}
  Note that $f(0, r) \in \partial X$ and $f(1, r) \in \partial Y$ for all $r \in [0, \eps]$ since $n \vert_{\partial X} \in \sections(T \partial X)$, $n \vert_{\partial Y} \in \sections(T \partial Y)$, and $n$ is parallel. Furthermore, $R(a, b) n = 0$ for all $a, b \in \sections(T W)$ since $n$ is parallel, where $R$ denotes the curvature tensor of~$W$. Hence
  \begin{equation*}
    f(t, r) = \exp_{\exp_x(r n_x)}(t \tilde{s}(r))
  \end{equation*}
  for all $t \in [0, 1]$, $r \in [0, \eps]$ since $n$ is parallel, where $\tilde{s}$ is the parallel transport of~$s_x$ along the curve $[0, \eps] \to \partial X$, $r \mapsto f(0, r)$. Note that $\tilde{s}(\eps)$ is normal to~$\partial X$ as a similar argument to above shows that $T_{f(0, \eps)} \partial X$ is given by the parallel transport of $T_x \partial X$ along the curve $[0, \eps] \to \partial X$, $r \mapsto f(0, r)$. Hence $s(f(0, \eps)) = \tilde{s}(\eps)$. This shows that $n(\exp_x(t s_x)) \in T_{\exp_x(t s_x)} (\exp_{t s})(\partial X)$ for all $t \in [0, 1]$. In particular, $\partial Y$ is isotopic to~$\partial X$ through submanifolds such that $n$ is tangent to them, and we may assume that $n$ is tangent to~$Z$. But this implies that $\Phi \vert_Z = 0$.
  
  Hence
  \begin{equation*}
    \vol(X) = \int_X \vol_X = \int_X \Phi = \int_Y \Phi \le \int_Y \vol_Y = \vol(Y)
  \end{equation*}
  since $X$ is Cayley, and equality holds if and only if $\Phi \vert_Y = \vol_Y$, that is, if and only if $Y$ is Cayley.
\end{proof}

As a preparation for dimension~$6$, we have the following two lemmas.

\begin{lemma}
  Let $M$ be an $8$\dash-manifold with a $\Spin(7)$-structure~$\Phi$, let $W$ be an oriented $6$\dash-dimensional submanifold of~$M$, and let $\omega \defeq - \mathord\ast_W (\Phi \vert_W)$. Then
  \begin{equation}
    \tfrac{1}{2} \omega^2 = - \Phi \vert_W \quad \text{and} \quad \tfrac{1}{6} \omega^3 = \vol_W \, \text{.} \label{eq:Phi-omega-dim-6}
  \end{equation}
\end{lemma}

\begin{proof}
  Let $(e_1, \dotsc, e_8)$ be a local $\Spin(7)$\dash-frame such that $(e_1, \dotsc, e_6)$ is a positive frame of~$W$. Then
  \begin{equation*}
    \Phi \vert_W = e^{1234} + e^{1256} - e^{3456}
  \end{equation*}
  by \eqref{eq:phi-spin7-frame}. So
  \begin{equation*}
    \omega = - \mathord\ast_W (\Phi \vert_W) = e^{12} - e^{34} - e^{56} \, \text{.}
  \end{equation*}
  Hence
  \begin{equation*}
    \tfrac{1}{2} \omega^2 = - e^{1234} - e^{1256} + e^{3456} = - \Phi \vert_W
  \end{equation*}
  and
  \begin{equation*}
    \tfrac{1}{6} \omega^3 = e^{123456} = \vol_W \, \text{.} \qedhere
  \end{equation*}
\end{proof}

\begin{lemma} \label{lemma:cayley-orthogonal-lagrangian}
  Let $M$ be an $8$\dash-manifold with a $\Spin(7)$-structure~$\Phi$, let $X$ be a Cayley submanifold of~$M$ with boundary, let $W$ be an oriented $6$\dash-dimensional submanifold of~$M$ with $\partial X \subseteq W$, and let $\omega \defeq - \mathord\ast_W (\Phi \vert_W)$. Then $X$ and $W$ meet orthogonally if and only if $\omega \vert_{\partial X} = 0$.
\end{lemma}

\begin{proof}
  Let $u \in \sections(\normal[X]{\partial X})$ be the interior unit normal vector field of~$\partial X$ in~$X$.
  
  First suppose that $X$ and $W$ meet orthogonally. Then there is a local $\Spin(7)$\dash-frame $(e_1, \dotsc, e_8)$ such that $e_1 = u$, $(e_2, e_3, e_4)$ is a frame of~$\partial X$, and $(e_2, e_6, e_3, e_7, e_4, e_8)$ is a positive frame of~$W$. Then
  \begin{equation*}
    \Phi \vert_W = e^{2367} + e^{2468} + e^{3478}
  \end{equation*}
  by \eqref{eq:phi-spin7-frame}. So
  \begin{equation*}
    \omega = - \mathord\ast_W (\Phi \vert_W) = e^{26} + e^{37} + e^{48} \, \text{.}
  \end{equation*}
  Hence $\omega \vert_{\partial X} = 0$.
  
  Conversely, suppose that $\omega \vert_{\partial X} = 0$. Let $(a, b, c)$ be a local orthonormal frame of~$\partial X$, and let $d \in \sections(T W)$. Then $u = \pm a \times b \times c$ and $\omega(a, b) = \omega(a, c) = \omega(b, c) = 0$. So
  \begin{align*}
    g(d, u) &= \pm g(d, a \times b \times c) = \pm \Phi(d, a, b, c) = \mp \tfrac{1}{2} (\omega \wedge \omega)(d, a, b, c) \\
    &= \mp \tfrac{1}{2} (\omega(d, a) \omega(b, c) - \omega(d, b) \omega(a, c) + \omega(d, c) \omega(a, b)) = 0
  \end{align*}
  by~\eqref{eq:Phi-omega-dim-6}. Hence $X$ and $W$ meet orthogonally.
\end{proof}

\begin{proposition} \label{prop:minimal-boundary-dim-6}
  Let $M$ be an $8$\dash-manifold with a torsion-free $\Spin(7)$\dash-structure~$\Phi$, let $X$ be a compact Cayley submanifold of~$M$ with boundary, let $W$ be an oriented $6$\dash-dimensional submanifold of~$M$ with $\partial X \subseteq W$ such that $X$ and $W$ meet orthogonally, and let $\omega \defeq - \mathord\ast_W (\Phi \vert_W)$. Suppose that $H^4(W) = 0$ and that $(W, \omega)$ is a symplectic manifold (note that $\partial X$ is a Lagrangian submanifold of $(W, \omega)$ by \lemmaref{lemma:cayley-orthogonal-lagrangian}). Furthermore, let $Y$ be a compact, oriented $4$\dash-dimensional submanifold of~$M$ with $\partial Y \subseteq W$ which lies in the same relative homology class as $X$ in $H_4(M, W)$ such that $\partial Y$ is a Lagrangian submanifold of $(W, \omega)$ which is a local deformation of~$\partial X$.
  
  Then the volume of~$Y$ is greater than or equal to the volume of~$X$, and equality holds if and only if $Y$ is Cayley.
\end{proposition}

Note that $\D \omega = 0$ suffices for $(W, \omega)$ to be a symplectic manifold as $\omega$ is non-degenerate by~\eqref{eq:Phi-omega-dim-6}.

\begin{proof}
  There are a $5$\dash-chain~$S$ in~$M$ and a $4$\dash-chain~$Z$ in~$W$ such that $X - Y = \partial S + Z$. So
  \begin{equation*}
    \int_X \Phi - \int_Y \Phi = \int_{\partial S} \Phi + \int_Z \Phi = \int_S \D \Phi + \int_Z \Phi = \int_Z \Phi
  \end{equation*}
  by Stokes' Theorem since $\D \Phi = 0$ as $\Phi$ is torsion-free. We have $\Phi \vert_W = \D \Psi$ for some $\Psi \in \forms^3(W)$ since $H^4(W) = 0$. Hence if $Z^\prime$ is another $4$\dash-chain in~$W$ such that $\partial Z^\prime = \partial Z = \partial X - \partial Y$, then
  \begin{equation*}
    \int_{Z^\prime} \Phi = \int_{Z^\prime} \D \Psi = \int_{\partial Z^\prime} \Psi = \int_{\partial Z} \Psi = \int_Z \D \Psi = \int_Z \Phi
  \end{equation*}
  by Stokes' Theorem.
  
  Since $(W, \omega)$ is a symplectic manifold and $\partial X$ is a Lagrangian submanifold of $(W, \omega)$, there is an open neighbourhood of~$\partial X$ in~$W$ which is symplectomorphic to an open tubular neighbourhood of the $0$\dash-section in~$T^\ast \partial X$ (Weinstein Tubular Neighbourhood Theorem). Furthermore, local deformations of~$\partial X$ that are Lagrangian submanifolds of $(W, \omega)$ correspond to closed $1$\dash-forms (with small $\CC^0$\dash-norm) under this identification. In particular, $\partial Y$ is isotopic to~$\partial X$ through Lagrangian submanifolds, and we may assume that for every $x \in Z$, the tangent space~$T_x Z$ contains a Lagrangian subspace of~$T_x W$. This implies that $\Phi \vert_Z = 0$ since $\Phi \vert_W = - \frac{1}{2} \omega \wedge \omega$ by~\eqref{eq:Phi-omega-dim-6}.
  
  Hence
  \begin{equation*}
    \vol(X) = \int_X \vol_X = \int_X \Phi = \int_Y \Phi \le \int_Y \vol_Y = \vol(Y)
  \end{equation*}
  since $X$ is Cayley, and equality holds if and only if $\Phi \vert_Y = \vol_Y$, that is, if and only if $Y$ is Cayley.
\end{proof}

\begin{remark}
  There are also versions of the volume minimising property for other calibrations. For special Lagrangian submanifolds, see \lemmaref{lemma:special-Lagrangian-minimal-boundary} in \sectionref{subsec:special-lagrangian-boundary}. A version for coassociative submanifolds follows from \propositionref{prop:minimal-boundary-dim-6} and \propositionref{prop:cayley-coassociative} in \sectionref{subsec:coassociative-boundary}.
\end{remark}

\subsection{Holonomy Contained in SU(2) x SU(2)}
\label{subsec:example-holonomy-su2-su2}

Let $Z_1$ and $Z_2$ be two Riemannian $4$\dash-manifolds with holonomy contained in $\SU(2)$, and let $M \defeq Z_1 \times Z_2$ be their product (endowed with the product metric). Then the holonomy of~$M$ is contained in $\SU(2) \times \SU(2) \subseteq \SU(4) \subseteq \Spin(7)$.

Let $\tilde{X}$ be a compact, connected $4$\dash-dimensional submanifold of~$Z_1$ with non-empty boundary, let $Y$ be a $k$\dash-dimensional submanifold of~$Z_2$ ($0 \le k \le 4$), and let $p \in Y$. Define $X \defeq \tilde{X} \times \set{p}$ and $W \defeq \partial X \times Y$. Then $X$ is a Cayley submanifold of~$M$ since $X$ is a complex surface. Furthermore, $W$ is a $(k + 3)$\dash-dimensional submanifold of~$M$ with $\partial X \subseteq W$ such that $X$ and $W$ meet orthogonally.

\begin{lemma}
  Let $M$, $\tilde{X}$, $X$, $W$, and $k$ be as above, and let $Y$ be a local deformation of~$X$ such that $\partial Y \subseteq W$.
  
  Then $Y$ is a Cayley submanifold of~$M$ such that $Y$ and $W$ meet orthogonally if and only if $Y = \tilde{X} \times \set{q}$ for some $q \in Y$. So the moduli space of all local deformations of~$X$ as a Cayley submanifold of~$M$ with boundary on~$W$ and meeting~$W$ orthogonally is a smooth $k$\dash-dimensional manifold.
\end{lemma}

\begin{proof}
  The submanifolds $\tilde{X} \times \set{q}$ for $q \in Y$ are clearly Cayley submanifolds of~$M$ with boundary on~$W$ and meeting $W$ orthogonally. So the moduli space contains a smooth $k$\dash-dimensional manifold. We will now show that the dimension of the kernel of the linearisation of the boundary problem~\eqref{eq:boundary-second} at the $0$\dash-section is at most~$k$.
  
  First note that the normal bundle~$\normal[M]{X}$ is flat (with respect to the induced connection~$\nabla^\perp$). Let $(e_5, \dotsc, e_8)$ be a parallel orthonormal frame of~$\normal[M]{X}$ such that $e_5, \dotsc, e_{k + 4} \in \sections(T W \vert_{\partial X})$ and $e_{k + 5}, \dotsc, e_8 \in \sections(\normal[M]{W} \vert_{\partial X})$, and let
  \begin{equation*}
    s = s_1 e_5 + \dotsb + s_4 e_8 \in \sections(\normal[M]{X})
  \end{equation*}
  be in the kernel of the linearisation of~\eqref{eq:boundary-second} at the $0$\dash-section, where $s_1, \dotsc, s_4 \in \CC^\infty(X)$.
  
  Let $(e_1, \dotsc, e_4)$ be a local orthonormal frame of~$X$ such that $(e_1, \dotsc, e_8)$ is a $\Spin(7)$-frame (\wolog~we may assume that $(e_5, \dotsc, e_8)$ is positively oriented). We have $D^\ast D s = 0$ by \eqref{eq:linearisation-G}, where $D \colon \sections(\normal[M]{X}) \to \sections(E)$ is the Dirac operator defined in~\eqref{eq:def-D} and the vector bundle~$E$ of rank~$4$ over~$X$ is defined as in~\eqref{eq:def-E}. Since $\normal[M]{X}$ and $E$ are flat, this means
  \begin{align*}
    &
    \begin{pmatrix}
      \partial_1 &          -  \partial_2 &          -  \partial_3 &          -  \partial_4 \\
      \partial_2 & \phantom{-} \partial_1 & \phantom{-} \partial_4 &          -  \partial_3 \\
      \partial_3 &          -  \partial_4 & \phantom{-} \partial_1 & \phantom{-} \partial_2 \\
      \partial_4 & \phantom{-} \partial_3 &          -  \partial_2 & \phantom{-} \partial_1
    \end{pmatrix}
    \begin{pmatrix}
      \phantom{-} \partial_1 & \phantom{-} \partial_2 & \phantom{-} \partial_3 & \phantom{-} \partial_4 \\
              -  \partial_2 & \phantom{-} \partial_1 &          -  \partial_4 & \phantom{-} \partial_3 \\
              -  \partial_3 & \phantom{-} \partial_4 & \phantom{-} \partial_1 &          -  \partial_2 \\
              -  \partial_4 &          -  \partial_3 & \phantom{-} \partial_2 & \phantom{-} \partial_1
    \end{pmatrix}
    \begin{pmatrix}
      s_1 \\
      s_2 \\
      s_3 \\
      s_4
    \end{pmatrix}
    \\
    &\;\;=
    \begin{pmatrix}
      \Delta & 0      & 0      & 0      \\
      0      & \Delta & 0      & 0      \\
      0      & 0      & \Delta & 0      \\
      0      & 0      & 0      & \Delta
    \end{pmatrix}
    \begin{pmatrix}
      s_1 \\
      s_2 \\
      s_3 \\
      s_4
    \end{pmatrix}
    = 0
  \end{align*}
  in~$X$ by \eqref{eq:def-D} and \eqref{eq:cross-product-table} since $\nabla \Phi = 0$, where we used $(e_1 \times e_5, \dotsc, e_1 \times e_8)$ as a frame of~$E$. So $s_1, \dotsc, s_4$ are harmonic functions.
  
  We further have $\pi_K(s \vert_{\partial X}) = 0$, where $\pi_K$ is defined in~\eqref{eq:def-pi-K}. So $s_{k + 1}, \dotsc, \allowbreak s_4 = 0$ on~$\partial X$. Hence $s_{k + 1} = \dotsb = s_4 = 0$ in~$X$ since the solution of the Laplace equation with Dirichlet boundary condition is unique.
  
  Furthermore, $\pi_\nu(\nabla_u s \vert_{\partial X} - \nabla_{\pi_\nu(s \vert_{\partial X})} u) = 0$ by~\eqref{eq:linearisation-B} since $\pi_K(s \vert_{\partial X}) = 0$, where $\pi_\nu$ is defined in~\eqref{eq:def-pi-nu} and $u \in \sections(\normal[X]{\partial X})$ is the inward-pointing unit normal vector field of~$\partial X$ in~$X$. Note that $\pi_\nu(\nabla_{\pi_\nu(s \vert_{\partial X})} u) = 0$ since $M$ is endowed with the product metric and $u \in \sections(T Z_1 \vert_{\partial X})$, $\pi_\nu(s \vert_{\partial X}) \in \sections(T Z_2 \vert_{\partial X})$. So $\pi_\nu(\nabla_u s \vert_{\partial X}) = 0$. Hence $\partial_u s_1 = \dotsb = \partial_u s_k = 0$ on~$\partial X$, where $\partial_u$ is the (interior) normal derivative. So $s_1, \dotsc, s_k$ satisfy the Laplace equation with Neumann boundary condition, and hence are constant. Therefore, the dimension of the kernel of the linearisation of~\eqref{eq:boundary-second} at the $0$\dash-section is at most~$k$.
\end{proof}

\subsection{Bryant--Salamon Construction}
\label{subsec:examples-bryant-salamon}

Bryant and Salamon \cite[Theorem~2 of Section~4]{BS89} constructed a $\Spin(7)$-structure~$\Phi$ on the total space of the spin bundle~$\mathbb{S}_-$ over~$S^4$ (with the round metric) such that the resulting manifold is a complete Riemannian manifold with holonomy equal to~$\Spin(7)$ (in particular, $\nabla \Phi = 0$). As noted in \cite[Section~6]{McL98}, the $0$\dash-section of~$\mathbb{S}_-$ is a (rigid) closed Cayley submanifold of~$\mathbb{S}_-$ with respect to~$\Phi$. The metric has the form
\begin{equation*}
  g = f_s(r) \pi^\ast g_s + f_\nu(r) g_\nu \, \text{,}
\end{equation*}
where $g_s$ is the round metric on~$S^4$, $g_\nu$ is the flat metric on the fibres of $\mathbb{S}_-$ induced by~$g_s$, $r$ is its associated norm, $\pi \colon \mathbb{S}_- \to S^4$ is the natural projection, and $f_s, f_\nu \colon [0, \infty) \to \R$ are given by
\begin{equation*}
  f_s(r) = 5 (1 + r^2)^{\frac{3}{5}} \quad \text{and} \quad f_\nu(r) = 4 (1 + r^2)^{- \frac{2}{5}} \, \text{.}
\end{equation*}
The particular choice of $f_s$ and $f_\nu$ will not be important in the following. In particular, one can also use the solution $f_s(r) = - 5 (1 - r^2)^{\frac{3}{5}}$, $f_\nu(r) = 4 (1 - r^2)^{- \frac{2}{5}}$ (case~(i) in~\cite{BS89}), which is only defined on the open subset of~$\mathbb{S}_-$ where $r < 1$. The associated metric is not complete.

Let $M \defeq \mathbb{S}_-$, let $X$ be a compact $4$\dash-dimensional submanifold of~$S^4$ with boundary (where we view $S^4$ as the $0$\dash-section of~$\mathbb{S}_-$), and let $W$ be a subbundle of $\mathbb{S}_- \vert_{\partial X}$ of rank~$k$ ($0 \le k \le 4$). Then $X$ is a Cayley submanifold of~$M$ with boundary and $W$ is a $(k + 3)$\dash-dimensional submanifold of~$M$ with $\partial X \subseteq W$ such that $X$ and $W$ meet orthogonally.

\begin{lemma}
  Let $M$, $X$, and $W$ be as above. Then $X$ is rigid as a Cayley submanifold of~$M$ with boundary on~$W$ and meeting $W$ orthogonally.
\end{lemma}

Note that \corollaryref{cor:existence-cayley-nearby} implies that also for every local deformation $W^\prime$ of~$W$, there exists a (unique) Cayley submanifold of~$M$ near~$X$ with boundary on~$W^\prime$ and meeting $W^\prime$ orthogonally.

\begin{proof}
  We will show that the kernel of the linearisation of the boundary problem~\eqref{eq:boundary-second} at the $0$\dash-section contains only the $0$\dash-section. So let $s \in \sections(\normal[M]{X})$ be in the kernel of the linearisation of~\eqref{eq:boundary-second} at the $0$\dash-section. Then $D^\ast D s = 0$ by \eqref{eq:linearisation-G}, where $D \colon \sections(\normal[M]{X}) \to \sections(E)$ is the Dirac operator defined in~\eqref{eq:def-D} and the vector bundle~$E$ of rank~$4$ over~$X$ is defined as in~\eqref{eq:def-E}. Also $\pi_K(s \vert_{\partial X}) = 0$, where $\pi_K$ is defined in~\eqref{eq:def-pi-K}. Furthermore, $\pi_\nu(\nabla_u s \vert_{\partial X} - \nabla_{\pi_\nu(s \vert_{\partial X})} u) = 0$ by~\eqref{eq:linearisation-B} since $\pi_K(s \vert_{\partial X}) = 0$, where $\pi_\nu$ is defined in~\eqref{eq:def-pi-nu} and $u \in \sections(\normal[X]{\partial X})$ is the inward-pointing unit normal vector field of~$\partial X$ in~$X$. \lemmaref{lemma:bryant-salamon-connection-0} below implies that $\pi_\nu(\nabla_{\pi_\nu(s \vert_{\partial X})} u) = 0$. So $\pi_\nu(\nabla_u s \vert_{\partial X}) = 0$.
  
  We have
  \begin{equation*}
    D^\ast D s = (\nabla^\perp)^\ast \nabla^\perp s + 3 s
  \end{equation*}
  by the Lichnerowicz formula \cite[Theorem~II.8.8]{LM89}, where $\nabla^\perp$ is the induced connection on~$\normal[M]{X}$. Hence
  \begin{align*}
    0 &= \inner{D^\ast D s, s}_{L^2(\normal[M]{X})} \\
    &= \inner{(\nabla^\perp)^\ast \nabla^\perp s, s}_{L^2(\normal[M]{X})} + 3 \norm{s}^2_{L^2(\normal[M]{X})} \\
    &= \norm{\nabla^\perp s}^2_{L^2(T^\ast X \otimes \normal[M]{X})} + \inner{\nabla_u s, s}_{L^2(\normal[M]{X} \vert_{\partial X})} + 3 \norm{s}^2_{L^2(\normal[M]{X})} \\
    &= \norm{\nabla^\perp s}^2_{L^2(T^\ast X \otimes \normal[M]{X})} + 3 \norm{s}^2_{L^2(\normal[M]{X})}
  \end{align*}
  since $\nabla_u s$ and $s$ are pointwise orthogonal on~$\partial X$ as $\pi_\nu(\nabla_u s \vert_{\partial X}) = 0$ and $\pi_K(s \vert_{\partial X}) = 0$. So $s = 0$.
\end{proof}

\begin{lemma} \label{lemma:bryant-salamon-connection-0}
  Let $s, t \in \sections(T M)$ such that $(\D \pi)(s) = (\D \pi)(t) = 0$. Then $\nabla_s t \vert_X \in \sections(\normal[M]{X})$.
\end{lemma}

\begin{proof}
  Let $(x_1, \dotsc, x_4)$ be coordinates on~$S^4$, let $(x_5, \dotsc, x_8)$ be linear coordinates on the fibres of~$\mathbb{S}_-$, and let $\tilde{\nabla}$ be the spin connection on~$\mathbb{S}_-$. Then
  \begin{equation*}
    g\Bigl(\tfrac{\partial}{\partial x_i}, \tfrac{\partial}{\partial x_j}\Bigr) = f_\nu(r) \, g_\nu\Bigl(\tfrac{\partial}{\partial x_i}, \tfrac{\partial}{\partial x_j}\Bigr)
  \end{equation*}
  for $i, j = 5, \dotsc, 8$ and
  \begin{equation*}
    g\Bigl(\tfrac{\partial}{\partial x_i}, \tfrac{\partial}{\partial x_j}\Bigr) = f_\nu(r) \sum_{k = 5}^8 x_k \, g_\nu\Bigl(\tilde{\nabla}_{\!\frac{\partial}{\partial x_i}} \tfrac{\partial}{\partial x_k}, \tfrac{\partial}{\partial x_j}\Bigr)
  \end{equation*}
  for $i = 1, \dotsc, 4$ and $j = 5, \dotsc, 8$. Let
  \begin{equation*}
    \nabla_{\!\frac{\partial}{\partial x_i}} \tfrac{\partial}{\partial x_j} = \sum_{k = 1}^8 \Gamma_{i j}^k \tfrac{\partial}{\partial x_k} \quad \text{and} \quad g_{i j} = g\Bigl(\tfrac{\partial}{\partial x_i}, \tfrac{\partial}{\partial x_j}\Bigr)
  \end{equation*}
  for $i, j = 1, \dotsc, 8$. Then $\Gamma_{i j}^k = \tfrac{1}{2} (\partial_i g_{j k} + \partial_j g_{i k} - \partial_k g_{i j})$ for $i, j , k = 1, \dotsc, 8$. Hence
  \begin{align*}
    \Gamma_{i j}^k &= \frac{1}{2} f_\nu(r) \biggl(g_\nu\Bigl(\tilde{\nabla}_{\!\frac{\partial}{\partial x_k}} \tfrac{\partial}{\partial x_i}, \tfrac{\partial}{\partial x_j}\Bigr) + g_\nu\Bigl(\tilde{\nabla}_{\!\frac{\partial}{\partial x_k}} \tfrac{\partial}{\partial x_j}, \tfrac{\partial}{\partial x_i}\Bigr) - \partial_k \Bigl(g_\nu\Bigl(\tfrac{\partial}{\partial x_i}, \tfrac{\partial}{\partial x_j}\Bigr)\Bigr)\biggr) \\
    &\phantom{{}={}} {}+ \frac{1}{2} (\partial_i f_\nu)(r) \sum_{\ell = 5}^8 x_\ell \, g_\nu\Bigl(\tilde{\nabla}_{\!\frac{\partial}{\partial x_k}} \tfrac{\partial}{\partial x_\ell}, \tfrac{\partial}{\partial x_j}\Bigr) \\
    &\phantom{{}={}} {}+ \frac{1}{2} (\partial_j f_\nu)(r) \sum_{\ell = 5}^8 x_\ell \, g_\nu\Bigl(\tilde{\nabla}_{\!\frac{\partial}{\partial x_k}} \tfrac{\partial}{\partial x_\ell}, \tfrac{\partial}{\partial x_i}\Bigr)
  \end{align*}
  for $i, j = 5, \dotsc, 8$ and $k = 1, \dotsc, 4$. So $\Gamma_{i j}^k = 0$ for $i, j = 5, \dotsc, 8$ and $k = 1, \dotsc, 4$ at $r = 0$ since $g_\nu$ is compatible with~$\tilde{\nabla}$.
\end{proof}

\subsection{Cayley Deformations of Special Lagrangian Submanifolds}
\label{subsec:special-lagrangian-boundary}

Let $M$ be a Calabi--Yau $4$\dash-fold with Kähler form~$\omega$ and holomorphic volume form~$\Omega$, which we assume to be \emph{normalised}, that is, $\omega^4 = \frac{3}{2} \Omega \wedge \bar{\Omega}$. Then
\begin{equation*}
  \Phi \defeq - \frac{1}{2} \, \omega \wedge \omega + \Re \Omega
\end{equation*}
defines a $\Spin(7)$\dash-structure on~$M$ \cite[Proposition~1.32 in Chapter~IV]{HL82}. This $\Spin(7)$\dash-structure is torsion-free since $\omega$ and $\Omega$ are closed.

An orientable $4$\dash-dimensional submanifold~$X$ of~$M$ is called \emph{special Lagrangian} if $(\Re \Omega) \vert_X = \vol_X$ for some orientation of~$X$. This is equivalent to $\omega \vert_X = 0$, $(\Im \Omega) \vert_X = 0$ \cite[Corollary~1.11 in Chapter~III]{HL82}. So every special Lagrangian submanifold is Cayley, but not every Cayley submanifold is special Lagrangian (for example, complex $2$\dash-dimensional submanifolds are also Cayley).

\begin{proposition} \label{prop:cayley-special-Lagrangian}
  Let $M$ be a Calabi--Yau $4$\dash-fold, let $X$ be a compact special Lagrangian submanifold of~$M$ with boundary, let $W$ be a complex $3$\dash-dimensional submanifold of~$M$ such that $\partial X \subseteq W$ (which implies that $X$ and $W$ meet orthogonally by \lemmaref{lemma:lagrangian-complex-orthogonal} below), and let $Y$ be a local deformation of~$X$ with $\partial Y \subseteq W$.
  
  Then $Y$ is a Cayley submanifold of~$M$ such that $Y$ and $W$ meet orthogonally if and only if $Y$ is a special Lagrangian submanifold of~$M$. So the moduli space of all local deformations of~$X$ as a Cayley submanifold of~$M$ with boundary on~$W$ and meeting~$W$ orthogonally can be identified with the moduli space of all local deformations of~$X$ as a special Lagrangian submanifold of~$M$ with boundary on~$W$.
\end{proposition}

The moduli space of all local deformations of~$X$ as a special Lagrangian submanifold of~$M$ with boundary on~$W$ is a smooth manifold of dimension~$b_1(X)$. This follows from Butscher's work \cite{But03}, as we will see after the proof. The following lemma is analogous to \propositionref{prop:minimal-boundary-dim-4}.

\begin{lemma} \label{lemma:special-Lagrangian-minimal-boundary}
  Let $M$ be a Calabi--Yau $n$\dash-fold, let $X$ be a compact special Lagrangian submanifold of~$M$ with boundary, let $W$ be a complex $(n - 1)$\dash-dimensional submanifold of~$M$ with $\partial X \subseteq W$, and let $Y$ be an oriented real $n$\dash-dimensional submanifold of~$M$ with $\partial Y \subseteq W$ which lies in the same relative homology class as $X$ in $H_n(M, W)$.
  
  Then the volume of~$Y$ is greater than or equal to the volume of~$X$, and equality holds if and only if $Y$ is special Lagrangian.
\end{lemma}

\begin{proof}
  The proof is similar to the proof of \propositionref{prop:minimal-boundary-dim-4} using $\Re \Omega$ instead of $\Phi$, where $\Omega$ is the holomorphic volume form of~$M$. Here $(\Re \Omega) \vert_W = 0$ as $W$ is complex.
\end{proof}

As noted in \cite[Remark after Definition~1]{But03}, we have the following.

\begin{lemma} \label{lemma:lagrangian-complex-orthogonal}
  Let $M$ be a Kähler manifold, let $X$ be a Lagrangian submanifold of~$M$ with boundary, and let $W$ be a complex submanifold of~$M$ with complex codimension~$1$ such that $\partial X \subseteq W$. Then $X$ and $W$ meet orthogonally.
\end{lemma}

\begin{proof}
  Let $J \colon T M \to T M$ be the complex structure of~$M$, let $u \in \sections(\normal[X]{\partial X})$ be the inward-pointing unit normal vector field of~$\partial X$ in~$X$, and let $(e_1, \dotsc, e_{n - 1})$ be a local orthonormal frame of~$\partial X$ (where $n \defeq \dim_\C M$). Then $J u$, $J e_1$, \dots, $J e_{n - 1}$ are normal to~$X$ since $X$ is Lagrangian and $M$ is Kähler. Moreover, $J e_1$, \dots, $J e_{n - 1}$ are tangent to~$W$ since~$W$ is complex and $\partial X \subseteq W$. Furthermore, $(u, e_1, \dotsc, e_{n - 1}, J u, J e_1, \dotsc, J e_{n - 1})$ is an orthonormal frame of~$T M \vert_{\partial X}$ since $J$ is orthogonal. So $(e_1, \dotsc, e_{n - 1}, J e_1, \dotsc, J e_{n - 1})$ is an orthonormal frame of~$W$. Hence $X$ and $W$ meet orthogonally.
\end{proof}

\begin{proof}[\proofname\ (\propositionref{prop:cayley-special-Lagrangian})]
  If $Y$ is a special Lagrangian submanifold of~$M$, then $Y$ is clearly also a Cayley submanifold of~$M$. Furthermore, $Y$ and $W$ meet orthogonally by \lemmaref{lemma:lagrangian-complex-orthogonal}.
  
  Conversely, suppose that $Y$ is a Cayley submanifold of~$M$ such that $Y$ and $W$ meet orthogonally. Then $\vol(Y) = \vol(X)$ by \propositionref{prop:minimal-boundary-dim-6}. So $Y$ is special Lagrangian by \lemmaref{lemma:special-Lagrangian-minimal-boundary}.
\end{proof}

Butscher proved the following theorem about minimal Lagrangian deformations of compact special Lagrangian submanifolds with boundary.

\begin{theorem}[{\cite[Main Theorem]{But03}}] \label{thm:deformations-special-lagrangian}
  Let $M$ be a Calabi--Yau manifold with Kähler form~$\omega$, let $X$ be a compact special Lagrangian submanifold of~$M$ with boundary, let $u \in \sections(\normal[X]{\partial X})$ be the inward-pointing unit normal vector field of~$\partial X$ in~$X$, and let $W$ be a submanifold of~$M$ with real codimension~$2$ such that
  \begin{compactenum}[(i)]
    \item $(W, \omega \vert_W)$ is a symplectic manifold,
    \item $\partial X \subseteq W$,
    \item $u \in \sections((T W \vert_{\partial X})^\omega)$, and
    \item the bundle $(T W)^\omega$ is trivial.
  \end{compactenum}
  Here $S^\omega$ denotes the symplectic orthogonal complement of a subspace~$S$ of a symplectic vector space~$V$, defined as $S^\omega \defeq \{v \in V \colon \omega(v, s) = 0 \text{ } \forall \, s \in S\}$.
  
  Then the moduli space of all local deformations of~$X$ as a minimal Lagrangian submanifold with boundary on~$W$ is a finite-dimensional manifold which is parametrised over
  \begin{equation*}
    \mathcal{H}^1_N(X) \defeq \set{\eta \in \forms^1(X) \colon \D \eta = 0, \updelta \eta = 0, u \interior \eta \vert_{\partial X} = 0} \, \text{.}
  \end{equation*}
\end{theorem}

The space $\mathcal{H}^1_N(X)$ of harmonic $1$\dash-fields on~$X$ with Neumann boundary condition is isomorphic to~$H^1(X; \R)$ \cite[page~927]{CDGM06}. So it has dimension~$b_1(X)$. Butscher stated the theorem for minimal Lagrangian submanifolds but we have the following.

\begin{lemma} \label{lemma:special-lagrangian-deformation}
  Let $M$ be a Calabi--Yau $n$\dash-fold, let $X$ be a compact special Lagrangian submanifold of~$M$ with boundary, let $W$ be a complex $(n - 1)$\dash-dimensional submanifold of~$M$ such that $\partial X \subseteq W$, and let $Y$ be a local deformation of~$X$ with $\partial Y \subseteq W$. Suppose that $Y$ is a minimal Lagrangian submanifold. Then $Y$ is special Lagrangian.
\end{lemma}

\begin{proof}
  We have $\vol(Y) \ge \vol(X)$ by \lemmaref{lemma:special-Lagrangian-minimal-boundary}. Furthermore, $Y$~is special Lagrangian with phase angle~$\theta$ for some~$\theta \in \R$ by \cite[Proposition~2.17 in Chapter~III]{HL82} (if $Y$ is not connected, then $\theta$ can be viewed as a locally constant function). So it is calibrated with respect to $\Re (\E^{\I \theta} \Omega)$, where $\Omega$ is the holomorphic volume form of~$M$. Therefore, $\vol(X) \ge \vol(Y)$ by a lemma analogous to \lemmaref{lemma:special-Lagrangian-minimal-boundary} for special Lagrangian submanifolds with phase angle~$\theta$. So $\vol(Y) = \vol(X)$, and hence $Y$ is special Lagrangian by \lemmaref{lemma:special-Lagrangian-minimal-boundary}.
\end{proof}

\begin{corollary}
  Let $M$ be a Calabi--Yau $n$\dash-fold, let $X$ be a compact special Lagrangian submanifold of~$M$ with boundary, and let $W$ be a complex $(n - 1)$-dimensional submanifold of~$M$ such that $\partial X \subseteq W$.
  
  Then the moduli space of all local deformations of~$X$ as a special Lagrangian submanifold of~$M$ with boundary on~$W$ is a smooth manifold of dimension~$b_1(X)$.
\end{corollary}

\begin{proof}
  First note that the proof of \lemmaref{lemma:lagrangian-complex-orthogonal} also shows that $(u, J u)$ is an orthonormal frame of~$\normal[M]{W} \vert_{\partial X}$, where $u \in \sections(\normal[X]{\partial X})$ is the inward-pointing unit normal vector field of~$\partial X$ in~$X$. In particular, $u \in \sections((T W \vert_{\partial X})^\omega)$. Furthermore, $\normal[M]{W} \vert_{\partial X}$ is trivial. Since we are only interested in local deformations of~$X$ and $\partial X$ is compact, we therefore may assume that $\normal[M]{W}$ is trivial. So we can apply \theoremref{thm:deformations-special-lagrangian}, which gives the desired result by \lemmaref{lemma:special-lagrangian-deformation}.
\end{proof}

\subsection{Cayley Deformations of Coassociative Submanifolds}
\label{subsec:coassociative-boundary}

Let $\tilde{M}$ be a $7$\dash-manifold with a $G_2$\dash-structure~$\tilde{\varphi}$, let $\tilde{\psi}$ be the Hodge-dual of~$\tilde{\varphi}$, let $M \defeq \R \times \tilde{M}$, and let $t$ denote the coordinate on the $\R$\dash-factor. Then
\begin{equation*}
  \Phi \defeq \D t \wedge \tilde{\varphi} + \tilde{\psi}
\end{equation*}
defines a $\Spin(7)$\dash-structure on~$M$ \cite[Proposition~1.30 in Chapter~IV]{HL82}. The $\Spin(7)$\dash-structure~$\Phi$ is torsion-free if the $G_2$\dash-structure~$\tilde{\varphi}$ is torsion-free.

An orientable $4$\dash-dimensional submanifold~$\tilde{X}$ of~$\tilde{M}$ is called \emph{coassociative} if $\tilde{\psi} \vert_{\tilde{X}} = \vol_{\tilde{X}}$ for some orientation of~$\tilde{X}$. This is equivalent to $\tilde{\varphi} \vert_{\tilde{X}} = 0$ \cite[Corollary~1.20 in Chapter~IV]{HL82}. So $\tilde{X}$ is a coassociative submanifold of~$\tilde{M}$ if and only if $X \defeq \set{0} \times \tilde{X}$ is a Cayley submanifold of~$M$.

\begin{proposition} \label{prop:cayley-coassociative}
  Let $\tilde{M}$ be a $7$\dash-manifold with a torsion-free $G_2$\dash-structure~$\tilde{\varphi}$, let $\tilde{\psi}$ be the Hodge-dual of~$\tilde{\varphi}$, let $\tilde{W}$ be an oriented $6$\dash-dimensional submanifold of~$\tilde{M}$ such that $(\tilde{W}, \mathord\ast_{\tilde{W}}(\tilde{\psi} \vert_{\tilde{W}}))$ is a symplectic manifold, and let $\tilde{X}$ be a compact coassociative submanifold of~$\tilde{M}$ with $\partial \tilde{X} \subseteq \tilde{W}$ such that $\tilde{X}$ and $\tilde{W}$ meet orthogonally. Define $M \defeq \R \times \tilde{M}$, $W \defeq \set{0} \times \tilde{W}$, and $X \defeq \set{0} \times \tilde{X}$. Furthermore, let $Y$ be a local deformation of~$X$ such that $\partial Y \subseteq W$.
  
  Then $Y$ is a Cayley submanifold of~$M$ such that $Y$ and $W$ meet orthogonally if and only if $Y = \set{0} \times \tilde{Y}$ for some coassociative submanifold~$\tilde{Y}$ of~$\tilde{M}$ with $\partial \tilde{Y} \subseteq \tilde{W}$ such that $\tilde{Y}$ and $\tilde{W}$ meet orthogonally. So the moduli space of all local deformations of~$X$ as a Cayley submanifold of~$M$ with boundary on~$W$ and meeting~$W$ orthogonally can be identified with the moduli space of all local deformations of~$\tilde{X}$ as a coassociative submanifold of~$\tilde{M}$ with boundary on~$\tilde{W}$ and meeting~$\tilde{W}$ orthogonally.
\end{proposition}

Kovalev and Lotay \cite{KL09} showed that the moduli space of all local deformations of~$\tilde{X}$ as a coassociative submanifold of~$\tilde{M}$ with boundary on~$\tilde{W}$ and meeting~$\tilde{W}$ orthogonally is a smooth manifold of dimension not greater than $b_1(\partial \tilde{X})$.

\begin{remark}
  Using \cite{Bar97}, one can show that the dimension is not greater than the minimum of the first Betti numbers of the connected components of the boundary. Indeed, by \cite[Theorem~4.10 and Corollary~3.8]{KL09}, the moduli space is parametrised by
  \begin{equation*}
    (\mathcal{H}_+^2)_{\text{bc}} \defeq \set{\alpha \in \forms^2_+(\tilde{X}) \colon \text{$\D \alpha = 0$ in $\tilde{X}$ and $\D_{\partial \tilde{X}} (u \interior \alpha \vert_{\partial \tilde{X}}) = 0$ on $\partial \tilde{X}$}} \, \text{.}
  \end{equation*}
  Let $Z$ be a connected component of the boundary~$\partial \tilde{X}$, and suppose that $\alpha \in (\mathcal{H}^2_+)_{\text{bc}}$ satisfies $\alpha \vert_Z = 0$. Then also $u \interior \alpha \vert_Z = \mathord\ast_Z (\alpha \vert_Z) = 0$ since $\alpha$ is self-dual. Now \cite[Main Theorem]{Bar97} implies that $\alpha = 0$ since $\forms^2_+(\tilde{X}) \oplus \forms^4(\tilde{X}) \to \forms^3(\tilde{X})$, $(\alpha, \beta) \mapsto \D \alpha + \updelta \beta$ is a Dirac operator and $Z$ has Hausdorff dimension~$3$. Note further that $\alpha \in (\mathcal{H}^2_+)_{\text{bc}}$ implies that $\alpha \vert_Z$ is a harmonic form on the boundary since $\mathord\ast_Z (\alpha \vert_Z) = u \interior \alpha \vert_Z$. Hence the map $(\mathcal{H}_+^2)_{\text{bc}} \to \mathcal{H}^1(Z)$, $\alpha \mapsto \alpha \vert_Z$ is injective.
\end{remark}

\begin{proof}
  If $\tilde{Y}$ is a coassociative submanifold of~$\tilde{M}$ with $\partial \tilde{Y} \subseteq \tilde{W}$ such that $\tilde{Y}$ and $\tilde{W}$ meet orthogonally, then $Y \defeq \set{0} \times \tilde{Y}$ is clearly a Cayley submanifold of~$M$ with $\partial Y \subseteq W$ such that $Y$ and $W$ meet orthogonally.
  
  Conversely, suppose that $Y$ is a Cayley submanifold of~$M$ such that $Y$ and $W$ meet orthogonally. Let $\pi \colon M = \R \times \tilde{M} \to \tilde{M}$ be the projection. If $Z$ is a local deformation of~$X$, then the volume of $Z$ is greater than or equal to the volume of $\pi(Z)$, and equality holds if and only if $Z$ is a submanifold of $\set{t} \times \tilde{M}$ for some~$t \in \R$. So $Y = \set{t} \times \tilde{Y}$ for some $t \in \R$ and a $4$\dash-dimensional submanifold~$\tilde{Y}$ of~$\tilde{M}$ since the volume of~$Y$ is equal to the volume of~$X$ and the volume of~$\pi(Y)$ is greater than or equal to the volume of~$X$ by \propositionref{prop:minimal-boundary-dim-6} (note that $\set{0} \times \pi(Y)$ is a submanifold of~$M$ with $\partial (\set{0} \times \pi(Y)) = \set{0} \times \partial (\pi(Y)) = \set{0} \times \pi(\partial Y) = \partial Y$ since $\partial Y \subseteq \set{0} \times \tilde{W}$). We have $t = 0$ since $\partial Y \subseteq \set{0} \times \tilde{W}$. Furthermore, $\tilde{Y}$ is coassociative since~$Y$ is Cayley.
\end{proof}

\subsection{Cayley Deformations of Associative Submanifolds}
\label{subsec:associative-boundary}

Recall from the \hyperref[subsec:coassociative-boundary]{last section} that if $\tilde{M}$ is a $7$\dash-manifold with a $G_2$\dash-structure~$\tilde{\varphi}$, $\tilde{\psi}$~is the Hodge-dual of~$\tilde{\varphi}$, $M \defeq S^1 \times \tilde{M}$, and $t$ denotes the coordinate on the $S^1$\dash-factor, then $\Phi \defeq \D t \wedge \tilde{\varphi} + \tilde{\psi}$ defines a $\Spin(7)$\dash-structure on~$M$. Furthermore, the $\Spin(7)$\dash-structure~$\Phi$ is torsion-free if the $G_2$\dash-structure~$\tilde{\varphi}$ is torsion-free.

An orientable $3$\dash-dimensional submanifold~$\tilde{X}$ of~$\tilde{M}$ is called \emph{associative} if $\tilde{\varphi} \vert_{\tilde{X}} = \vol_{\tilde{X}}$ for some orientation of~$\tilde{X}$. So $\tilde{X}$ is an associative submanifold of~$\tilde{M}$ if and only if $X \defeq S^1 \times \tilde{X}$ is a Cayley submanifold of~$M$. Also recall from the \hyperref[subsec:coassociative-boundary]{last section} that a $4$\dash-dimensional submanifold~$\tilde{W}$ of~$\tilde{M}$ is called \emph{coassociative} if $\tilde{\varphi} \vert_{\tilde{W}} = 0$. Note that if $\tilde{X}$ is associative, $\tilde{W}$ is coassociative, and $\partial \tilde{X} \subseteq \tilde{W}$, then $\tilde{X}$ and $\tilde{W}$ meet orthogonally as the $3$\dash-form~$\tilde{\varphi}$ defines a cross product $T \tilde{M} \times T \tilde{M} \to T \tilde{M}$, $(v, w) \mapsto v \times w$ such that $g(u, v \times w) = \tilde{\varphi}(u, v, w)$, and $\tilde{\varphi} \vert_{\tilde{X}} = \vol_{\tilde{X}}$, $\tilde{\varphi} \vert_{\tilde{W}} = 0$ imply that if $v, w \in T \partial \tilde{X}$, then $v \times w$ is both tangent to~$\tilde{X}$ and orthogonal to~$\tilde{W}$.

\begin{proposition} \label{prop:cayley-associative}
  Let $\tilde{M}$ be a $7$\dash-manifold with a torsion-free $G_2$\dash-structure, let $\tilde{X}$ be a compact associative submanifold of~$\tilde{M}$ with boundary, and let $\tilde{W}$ be a coassociative submanifold of~$\tilde{M}$ with $\partial \tilde{X} \subseteq \tilde{W}$ (which implies that $\tilde{X}$ and $\tilde{W}$ meet orthogonally). Define $M \defeq S^1 \times \tilde{M}$, $X \defeq S^1 \times \tilde{X}$, and $W \defeq S^1 \times \tilde{W}$. Furthermore, let $Y$ be a local deformation of~$X$ such that $\partial Y \subseteq W$.
  
  Then $Y$ is a Cayley submanifold of~$M$ such that $Y$ and $W$ meet orthogonally if and only if $Y = S^1 \times \tilde{Y}$ for some associative submanifold~$\tilde{Y}$ of~$\tilde{M}$ with $\partial \tilde{Y} \subseteq \tilde{W}$. So the moduli space of all local deformations of~$X$ as a Cayley submanifold of~$M$ with boundary on~$W$ and meeting~$W$ orthogonally can be identified with the moduli space of all local deformations of~$\tilde{X}$ as an associative submanifold of~$\tilde{M}$ with boundary on~$\tilde{W}$.
\end{proposition}

Gayet and Witt \cite{GW11} proved that the boundary problem for associative submanifolds of a $G_2$\dash-manifold with boundary in a coassociative submanifold is an elliptic boundary problem. They also computed the index of this boundary problem.

\begin{proof}
  If $\tilde{Y}$ is an associative submanifold of~$\tilde{M}$ with $\partial \tilde{Y} \subseteq \tilde{W}$ (which implies that $\tilde{Y}$ and $\tilde{W}$ meet orthogonally), then $Y \defeq S^1 \times \tilde{Y}$ is clearly a Cayley submanifold of~$M$ with $\partial Y \subseteq W$ such that $Y$ and $W$ meet orthogonally.
  
  Conversely, suppose that $Y$ is a Cayley submanifold of~$M$ such that $Y$ and $W$ meet orthogonally. Write $Y_t \defeq Y \cap (\set{t} \times \tilde{M})$ for $t \in S^1$. Note that $Y_t$ is a submanifold of $\tilde{M}$ for all $t \in S^1$ as $Y$ is a local deformation of~$X$. Then
  \begin{equation*}
    \vol(Y) \ge \int_{S^1} \vol(Y_t) \, \D t
  \end{equation*}
  with equality if and only if $Y = S^1 \times Y_0$. Furthermore, $\vol(Y_t) \ge \vol(\tilde{X})$ since $Y_t$ is a local deformation of~$\tilde{X}$ and $\tilde{X}$ is volume-minimising in its homology class as it is associative. Here we have equality if and only if $Y_t$ is associative. So
  \begin{equation*}
    \vol(Y) \ge \int_{S^1} \vol(Y_t) \, \D t \ge \int_{S^1} \vol(\tilde{X}) \, \D t = \vol(X) \, \text{.}
  \end{equation*}
  But $\vol(Y) = \vol(X)$ by \propositionref{prop:minimal-boundary-dim-5}. Hence we get equality for both inequalities, which implies that $Y = S^1 \times Y_0$ with $Y_0$ associative. Clearly, $\partial Y_0 \subseteq \tilde{W}$ as $\partial Y \subseteq W$.
\end{proof}

\end{document}